\documentclass[11pt,oneside]{amsart}
\usepackage{amssymb,amsthm}
\usepackage{amsfonts}
\usepackage{amsmath}
\usepackage{amssymb}
\usepackage{graphicx}
\usepackage{setspace}
\usepackage{bm}
\usepackage{cite}
\usepackage{CJKutf8}
\usepackage{color}
\usepackage{hyperref}
\usepackage{geometry}
\usepackage{multirow}
\allowdisplaybreaks[1]
\usepackage{enumerate}

\newcommand{\la}{\lambda}
\newcommand{\al}{\alpha}

\newcommand{\ga}{\gamma}

\newcommand{\ve}{\varepsilon}

\newcommand{\vp}{\varphi}
\newcommand{\f}{\frac}

\newcommand{\R}{\mathbb{R}}

\newcommand{\T}{\mathbb{T}}
\newcommand{\Z}{\mathbb{Z}}

\newcommand{\Om}{\Omega}
\newcommand{\om}{\omega}

\newcommand{\n}[1]{\Vert #1\Vert}
\newcommand{\bn}[1]{\big \Vert #1 \big \Vert}
\newcommand{\bbn}[1]{\Big\Vert #1 \Big \Vert}
\newcommand{\lr}[1]{\left\{ #1\right\}}
\newcommand{\lrc}[1]{\left[ #1\right]}
\newcommand{\lrs}[1]{\left( #1\right)}
\newcommand{\lrv}[1]{\left| #1\right|}
\newcommand{\lra}[1]{\langle #1\rangle}
\newcommand{\blra}[1]{\big\langle #1\big\rangle}
\newcommand{\bblra}[1]{\Big\langle #1\Big\rangle}
\newcommand{\abs}[1]{|#1|}
\newcommand{\babs}[1]{\big | #1 \big|}
\newcommand{\bbabs}[1]{\Big | #1 \Big|}
\newcommand{\wt}[1]{\widetilde{#1}}
\newcommand{\pa}{\partial}
\newcommand{\ol}[1]{\overline{#1}}

\newtheorem{theorem}{Theorem}[section]
\newtheorem{lemma}[theorem]{Lemma}

\newtheorem{proposition}[theorem]{Proposition}

\title[Asymptotic stability of the symmetric flow]{Asymptotic stability of the symmetric flow via inviscid damping and enhanced dissipation}

\author[Q. Chen]{Qi Chen}
\address{School of Mathematical Sciences, Zhejiang University, Hangzhou 310058, China}
\email{chenqi123@zju.edu.cn}

\author[H. Li]{Hao Li}
\address{School of Mathematical Sciences, University of Electronic Science and Technology of China, Chengdu, 611731, China}
\email{lihao\_pde@163.com}

\author[S. Shen]{Shunlin Shen}
\address{School of Mathematical Sciences, University of Science and Technology of China, Hefei, 230026, China}
\email{slshen@ustc.edu.cn}

\author[Z. Zhang]{Zhifei Zhang}
\address{School of Mathematical Sciences, Peking University, Beijing, 100871, China}

\email{zfzhang@pku.edu.cn}

\begin{document}


\begin{abstract}
In this paper, we establish the inviscid damping and enhanced dissipation estimates for the linearized Navier-Stokes system around the symmetric flow in a finite channel with the non-slip boundary condition.  As an immediate consequence, we prove the asymptotic stability of the symmetric flow in the high Reynolds number regime. Namely, if the initial velocity perturbation $u^{\mathrm{in}}$ satisfies $\Vert u^{\mathrm{in}}-(U(y),0) \Vert_{H^5}\leq c \nu^{\frac{2}{3}}$, then inviscid damping and enhanced dissipation estimates also hold for the solution to the Navier-Stokes system.
\end{abstract}

\maketitle

\tableofcontents

\section{Introduction}

\renewcommand{\theequation}{\thesection.\arabic{equation}}
\newtheorem{Definition}{Definition}[section]
\newtheorem{Theorem}{Theorem}[section]
\newtheorem{Proposition}{Proposition}[section]
\newtheorem{Lemma}{Lemma}[section]
\newtheorem{Remark}{Remark}[section]
\newtheorem{Corollary}{Corollary}[section]
\numberwithin{equation}{section}

Hydrodynamic stability at high Reynolds numbers has been a fundamental problem in fluid mechanics since Reynolds' pioneering experiment \cite{r-prsl1883}. A central focus of this research has been understanding the laminar-to-turbulent transition mechanisms. To this end, Trefethen et al. \cite{ttrd-science1993} first proposed a crucial question regarding the transition threshold, which was formalized through a rigorous mathematical framework by Bedrossian, Germain, and Masmoudi \cite{bgm-bams2019}:

 {\it
 Given a norm $\n{\cdot}_X$, determine a $\beta=\beta(X)$ such that
	\begin{align*}
		\n{u^{\mathrm{in}}}_X&\leq Re^{-\beta}\Longrightarrow stability,\\
		\n{u^{\mathrm{in}}}_X&\gg Re^{-\beta}\Longrightarrow instability.
\end{align*}}
Here the exponent $\beta$ is referred to as the transition threshold.

 Notably, there has been significant progress in understanding the stability of various flows, see, for example, Couette flow \cite{bgm-bams2019,bgm-mams2020,wz-cpam2021,cwz-mams2024,bgm-mams2022, CDW}, Poiseuille flow \cite{cdlz-arxiv2023,cwz-cpam2023}, Kolmogorov flow \cite{lwz-cpam2020}, and many others. These results indicate that two important physical effects, namely inviscid damping and enhanced dissipation due to the shear-flow-induced mixing mechanism, play a crucial role in the early stage of the transition. For further details, see the review papers \cite{WZ-ICM, bgm-bams2019}.

The main goal of this paper is to establish the linear inviscid damping and enhanced dissipation estimates for the linearized Navier-Stokes system around the symmetric flow $(U(y),0)$ in a finite channel
 $\Omega=\{(x,y):x\in \T, y\in(-1,1)\}$, which takes the form of
 \begin{equation}\label{equ:per,INS}
	\left\{
	\begin{aligned}
		&\partial_t u-\nu\Delta u+U(y)\pa_x u+(U'(y)u^2,0)+\nabla P=0,\\
		&\nabla\cdot u=0,\\
		& u(t,x,\pm1)=0,\quad u(0,x,y)=u^{\mathrm{in}}(x,y),
	\end{aligned}\right.
\end{equation}
 where  $\nu=Re^{-1}\ll 1$ is the viscosity coefficient, $u=(u^1,u^2)$ is the velocity, $P$ is the pressure. To reduce technical difficulties, we only consider the following symmetric flows that fulfill the $\mathrm{S}$-condition
\begin{align}\label{equ:K condition,U}
	\mathrm{S}:=\Big\{U(y)\in C^{4}[-1,1]:\quad  U(y)=U(-y),\quad\mathrm{and}\quad \inf_{y\in[-1,1]}U''(y)>0\Big\}.
\end{align}
An important example is the plane Poiseuille flow $U(y)=y^{2}$. Let us point out that the symmetric flows in an infinite channel (namely, $x\in \R$) are linearly unstable for a certain region of the parameter space identified by Grenier, Guo and Nguyen \cite{GGN}. However, the question of linear stability outside this region  remained open. In a recent work \cite{AH}, Almog and Helffer fill this gap and prove that outside the region of instability in the parameter space, the symmetric flows are indeed linearly stable.  In  this paper, we prove the linear stability of  the symmetric flows in a finite channel in the high Reynolds number regime.

The second goal of this paper is to investigate the {\bf asymptotic stability threshold problem} for the symmetric flows, a problem first proposed by Masmoudi and Zhao in the context of the 2D Couette flow \cite{MZ-AIHP}.
To this end, we consider the 2D incompressible Navier-Stokes equation with a force in a finite channel $\Omega=\{(x,y):x\in \T, y\in(-1,1)\}$:
\begin{equation}\label{equ:INS}
	\left\{
	\begin{aligned}
		&\pa_t v-\nu\Delta v+v\cdot \nabla v+\nabla P=F,\\
		&\nabla\cdot v=0,\quad v(0,x,y)=v^{\mathrm{in}}(x,y).
	\end{aligned}\right.
\end{equation}
Then the shear flow $(U(y), 0)$ is a steady solution of \eqref{equ:INS} with $F=(-\nu\pa^2_y U,0 )$.  Let $u=v-U$ be the velocity perturbation. Thus,
\begin{equation}\label{equ:NS-pert}
	\left\{
	\begin{aligned}
		&\partial_t u-\nu \Delta u+U(y)\pa_x u+(U'(y)u^2,0)+u\cdot \nabla u+\nabla P=0,\\
		&\nabla\cdot u=0,\\
		& u(t,x,\pm1)=0,\quad u(0,x,y)=u^{\mathrm{in}}(x,y),
	\end{aligned}\right.
\end{equation}
The asymptotic stability threshold problem is to seek a $\beta>0$ as small as possible such that if $\|u_{\text{in}}\|_X\ll \nu^\beta$, then the inviscid damping and enhanced dissipation estimates hold for the solution $u$ to the system \eqref{equ:NS-pert}. Let us review some known results in this direction. For the Couette flow, the following asymptotic stability results have been established when $\Om=\mathbb{T}\times \R$:
\begin{itemize}

\item if $X$ is  Gevrey class $2+$, then $\beta=0$ \cite{BMV};

\item if $X$ is Sobolev space, then $\beta\le \f13$ \cite{MZ-AIHP, WZ-TJM};

\item if $X$ is Gevrey class $\f 1s$, $ s\in[0,\f12]$, then $\beta\le \f{1-2s}{3(1-s)}$ \cite{LMZ}.
\end{itemize}
For $\Om=\mathbb{T}\times [-1,1]$,  the following results have been proved under different boundary conditions:
\begin{itemize}
\item if $X$ is Sobolev space with non-slip BC, then $\beta\le \f12$ \cite{CLWZ-ARMA};

\item if $X$ is Sobolev space with Navier-slip BC, then $\beta\le \f13$ \cite{WZ-AIHP};

\item if $X$ is  Gevrey class $2+$ with Navier-slip BC, then $\beta=0$ \cite{BHIW}.
\end{itemize}
For the plane Poiseuille flow, the following asymptotic stability results have been established:
\begin{itemize}
\item if $X$ is Sobolev space in $\T\times \R$, then $\beta\le \f34+$  \cite{CEW};

\item if $X$ is Sobolev space in $\T\times \R$, then $\beta\le \f23+$ \cite{Del};

\item if $X$ is Sobolev space in $\T\times [-1,1]$ with Navier-slip BC, then $\beta\le \f23$ \cite{DL-SCM}.
\end{itemize}
Very recently, Chen, Jia, Wei and Zhang \cite{CJWZ} prove the asymptotic stability threshold $\beta\le \f13$ for the Kolmogorov flow in $\T_{2\kappa\pi}\times \T_{2\pi}, \kappa<1$.

Let us refer to \cite{BHIW-CMP, WZZ-Adv, CZ-Non, BG, LMZ-CPAM, MZ-CPDE, AHL, AB} and the references therein for more relevant results.

\subsection{Main results}

 Let $\phi$ be the stream function with the velocity $u=(-\pa_{y}\phi,\pa_x \phi)$ and the vorticity $\om=\pa_xu^2-\pa_yu^1$.
 The linearized NS system of \eqref{equ:omega,full,NS} then reduces to
\begin{equation}\label{equ:omega,linear,NS}
	\left\{
	\begin{aligned}
		&\pa_{t}\om-\nu\Delta\om+U(y)\pa_x\om-U''(y)\pa_x\Delta^{-1}\om=0,\\
		&\Delta \phi=\om,\quad \phi(\pm 1)=0,\quad u(t,x,\pm1)=0,\\ &\om(0,x,y)=\om^{\mathrm{in}}(x,y).
	\end{aligned}\right.
\end{equation}

We define the zero mode and non-zero mode as follows
\begin{align}
	P_{\al}f=:&{f_{\al}}(t,y)=\f{1}{\abs{\T}}\int_{\T}f(t,x,y)e^{-i\al x}dx,\quad P_{\neq 0}f=:f-f_0=\sum_{\al\neq0}f_{\al}(t,y)e^{i\al x},
\end{align}
and introduce the norm notation
\begin{align}
\n{\om_{\al}}_{H_{\al}^{k}}^{2}=:&\sum_{i=0}^{k}\n{(\pa_{y},|\al|)^{i}\om_{\al}}_{L_{y}^{2}}^{2}.
\end{align}

We now present the linear inviscid damping and enhanced dissipation estimates for the linearized Navier-Stokes system \eqref{equ:omega,linear,NS}.

\begin{theorem}\label{theorem:linear,enhan,invis,est}
Let $U(y)\in \mathrm{S}$ and $(\om,u)$ be the solution to the linearized  NS system \eqref{equ:omega,linear,NS} with $\int_{\T}\om^{\mathrm{in}}(x,y)dx=0$. There exist positive constants $\nu_0$, $\varepsilon_0$, such that for $\nu\in(0,\nu_0]$, $\ve\in[0,\ve_{0}]$, the following space-time estimates hold for each Fourier mode $\al$,
	\begin{align}\label{equ:linear,enhan,invis,est}
	&\abs{\al}\n{e^{\varepsilon \nu^{\f12}t}u_{\al}}_{L_{t}^{\infty}L_{y}^{2}}^{2}
		+\abs{\al}\n{e^{\varepsilon \nu^{\f12}t}u_{\al}}_{L_{t,y}^{2}}^{2}+\n{e^{\varepsilon \nu^{\f12}t}u_{\al}}_{L_{t}^{\infty}L_{y}^{\infty}}^{2}+\nu^{\f12}\abs{\al}^{\f12}\n{e^{\varepsilon \nu^{\f12}t}\om_{\al}}^2_{L_{t,y}^{2}}\\
		&\quad  + \nu^{\frac{1}{2}}\n{e^{\varepsilon \nu^{\f12}t}\om_{\al}}_{L_{t}^{\infty}L_{y}^{2}}^{2}+\nu^{\f32}\n{e^{\varepsilon \nu^{\f12}t}\pa_y\om_{\al}}^2_{L_{t,y}^{2}}+\n{e^{\varepsilon \nu^{\f12}t}\sqrt{1-y^{2}}\om_{\al}}^2_{L_{t}^{\infty}L_{y}^{2}}
		\lesssim \n{\om^{\mathrm{in}}_{\al}}_{H_{\al}^{4}}^{2}.\notag
	\end{align}
\end{theorem}

Building on space-time estimates in Theorem \ref{theorem:linear,enhan,invis,est}, we can prove the asymptotic stability of the symmetric flow. We consider the Navier-Stokes system in terms of  the vorticity $\om=\pa_xu^2-\pa_yu^1$:
\begin{equation}\label{equ:omega,full,NS}
	\left\{
	\begin{aligned}
		&	\pa_{t} \om-\nu\Delta\om+U(y)\pa_x\om -U''(y)\pa_x\Delta^{-1}\om +u\cdot\nabla\om=0,\\
		&\Delta \phi=\om,\quad \phi(\pm 1)=0,\\ &u(t,x,\pm1)=0,\quad  \om(0,x,y)=\om^{\mathrm{in}}(x,y).
	\end{aligned}\right.
\end{equation}

\begin{theorem}\label{theorem:nonlinear,stability}
Let $U(y)\in \mathrm{S}$ and $(\om,u)$ be the solution to the nonlinear system \eqref{equ:omega,full,NS}. There exist positive constants $\nu_0$, $\ve_{0}$, $c$, such that if the initial perturbation\footnote{The $H_{\alpha}^{4}$ regularity requirement could be reduced. For brevity, we do not pursue sharp regularity here.} satisfies
\begin{align}
\sum_{\al\in \Z}\n{\om^{\mathrm{in}}_{\al}}_{H_{\al}^{4}}\leq c\nu^{\f23},\quad \text{for $\nu\in(0, \nu_0]$,}
\end{align}
then the solution $(\om,u)$ is global and satisfies global stability estimate
	\begin{align}\label{equ:nonlinear,stability}
		\sum_{\al\in \Z}\mathcal{E}_{\al}
		\lesssim \sum_{\al\in \Z}\n{\om^{\mathrm{in}}_{\al}}_{H_{\al}^{4}},
	\end{align}
	where the stability norm $\mathcal{E}_{\al}$ is defined as
	\begin{equation}\label{euq,def,E,al}
\mathcal{E}_{\al}:=
		\left\{
		\begin{aligned}
		&\n{{\om}_0}_{L_{t}^{\infty}L_{y}^{2}},
		&   {\al}= 0,\\
		&\abs{\al}^{\f12}\n{ e^{\ve_{0} \nu^{\f12}t}u_{\al}}_{L_{t}^{\infty}L_{y}^{2}}+\abs{\al}^{\f12}\n{ e^{\ve_{0} \nu^{\f12}t}u_{\al}}_{L_{t,y}^{2}}+\n{e^{\ve_{0} \nu^{\f12}t} u_{\al}}_{L_{t}^{\infty}L_{y}^{\infty}}+\\
		&\nu^{\f14}\abs{\al}^{\f14}\n{e^{\ve_{0} \nu^{\f12}t}\om_{\al}}_{L_{t,y}^{2}}
		+\nu^{\f14}\n{e^{\ve_{0}\nu^{\f12}t}\om_{\al}}_{L_{t}^{\infty}L_{y}^{2}}+ \n{e^{\ve_{0} \nu^{\f12}t}\sqrt{1-y^{2}}\om_{\al}}_{L_{t}^{\infty}L_{y}^{2}},  &{\al}\neq 0.
		\end{aligned}\right.
	\end{equation}
In particular, there holds that
\begin{align}
\n{P_{\neq 0}u}_{L_{t,x,y}^{2}}+\nu^{\frac{1}{4}}\n{P_{\neq 0}\om}_{L_{t,x,y}^{2}}\lesssim& \sum_{\al\in \Z}\n{\om^{\mathrm{in}}_{\al}}_{H_{\al}^{4}},\label{equ:enhanced,inviscid,nonlinear}\\
\n{P_{\neq 0}u(t)}_{L_{x,y}^{\infty}}+\nu^{\frac{1}{4}}\n{P_{\neq 0}\om(t)}_{L_{x,y}^{2}}\lesssim& e^{-\ve_{0}\nu^{\frac{1}{2}}t}\sum_{\al\in \Z}\n{\om^{\mathrm{in}}_{\al}}_{H_{\al}^{4}}.
\end{align}
\end{theorem}
Let us give some remarks on Theorems \ref{theorem:linear,enhan,invis,est} and Theorem \ref{theorem:nonlinear,stability}.

\begin{itemize}
\item
Theorems \ref{theorem:linear,enhan,invis,est}--\ref{theorem:nonlinear,stability} establish the nonlinear stability of general symmetric flows under the non-slip boundary condition. The linear stability problem for such flows was initially studied in \cite{AH}. Our work employs a distinct framework to derive full resolvent estimates, inviscid damping, and enhanced dissipation estimates. These results serve as key ingredients for resolving the nonlinear stability problem.

\item
The asymptotic stability  threshold $\beta=\frac{2}{3}$ under the non-slip boundary condition is rather nontrivial, although it may not be optimal. For comparison, even in the special case of the plane Poiseuille flow with the Navier-slip boundary condition, the threshold in existing literature remains $\frac{2}{3}$, as seen in \cite{DL-SCM}.
The present work achieves this threshold for  general symmetric flows under the non-slip boundary condition, relying crucially on the sharp inviscid damping estimate $|\al|^{\frac{1}{2}}\n{u_{\al}}_{L_{t,y}^{2}}$  and the enhanced dissipation estimate $\nu^{\frac{1}{4}}|\al|^{\frac{1}{4}}\n{\om_{\al}}_{L_{t,y}^{2}}$. These estimates are novel and, to our knowledge, established here for the first time, even in the case of plane Poiseuille flow.

\item The $\n{u(t)}_{L_{x,y}^{\infty}}$ estimate  included in the stability norm means that the symmetric flows are absolutely stable. This pointwise  estimate is far from trivial, the proof of which is also based on sharp $L_{t,y}^{2}$ space-time estimates.

\item Unlike studies focusing on specific flows such as Couette or Poiseuille flows,
 the analysis of general symmetric flows presents a substantial challenge due to the nonlocal term $U''(y)\pa_x\Delta^{-1}$, which is widely recognized as a major difficulty.
 To tackle this, we develop a unified framework based on coercive estimates for the Rayleigh and Orr-Sommerfeld equations. This approach systematically addresses the inherent complexities of the nonlocal term, providing potential applications for broader stability analyses.

 \item The symmetric condition imposed on $U(y)$ greatly streamlines the proof. We note, however, that this condition may not be essential and could be relaxed.

 \item  The asymptotic stability threshold derived in this paper is not optimal. We conjecture that the asymptotic stability threshold is $\beta\le \f12$ for the non-slip boundary condition, and $\beta\le \f13$
 for the Navier-slip boundary condition.

\end{itemize}

\subsection{Key ideas and outline of the proof}
To investigate the asymptotic stability of the shear flow $(U(y),0)$, a crucial ingredient is to derive the space-time estimates for the linearized Navier-Stokes system around it
\begin{equation}\label{equ:linear,nonslip,NS,fj,intro}
	\left\{
	\begin{aligned}
		&\pa_{t}\om-\nu(\pa^2_y-\al^2)\om+i\al U\om-i\al U''\phi=i\al f_1+\pa_y f_2,\\
		&(\pa^2_y-\al^2)\phi=\om,\quad \phi(t,\pm 1)=\phi'( t,\pm 1)=0,\quad \om(0,y)=\om_{\al}^{\mathrm{in}}(y).
	\end{aligned}\right.
\end{equation}
Namely, there holds
\begin{align}\label{eq:ST}
		&\abs{\al}\n{e^{\varepsilon \nu^{\f12}t}u}_{L_{t}^{\infty}L_{y}^{2}}^{2}
		+\abs{\al}\n{e^{\varepsilon \nu^{\f12}t}u}_{L_{t,y}^{2}}^{2}+\n{e^{\varepsilon \nu^{\f12}t}u}_{L_{t}^{\infty}L_{y}^{\infty}}^{2}+\nu^{\f12}\abs{\al}^{\f12}\n{e^{\varepsilon \nu^{\f12}t}\om}^2_{L_{t,y}^{2}}\\
		&+ \nu^{\frac{1}{2}}\n{e^{\varepsilon \nu^{\f12}t}\om}_{L_{t}^{\infty}L_{y}^{2}}^{2}+\nu^{\f32}\n{e^{\varepsilon \nu^{\f12}t}\pa_y\om}^2_{L_{t,y}^{2}}+\n{e^{\varepsilon \nu^{\f12}t}\sqrt{1-y^{2}}\om}^2_{L_{t}^{\infty}L_{y}^{2}}\notag\\
		\lesssim& \n{\om^{\mathrm{in}}}_{H_{\al}^{4}}^{2}+\nu^{-1}\n{e^{\varepsilon \nu^{\f12}t}(f_1,f_2)}^2_{L_{t,y}^{2}}.\notag
	\end{align}
In the case of $\T\times \R$, both inviscid-damping and enhanced-dissipation estimates for the Couette flow can be obtained by Fourier-multiplier techniques \cite{bgm-am2017,BMV,WZ-TJM,MZ-AIHP}. In the channel domain $\T\times [-1,1]$, however, this approach ceases to apply. A systematic framework was introduced in \cite{CLWZ-ARMA} that reduces the derivation of space-time estimates to resolvent estimates via the following crucial decomposition.

The solution to the linearized equation \eqref{equ:linear,nonslip,NS,fj,intro} admits a decomposition:
\begin{align*}
\om=\om_{I}+\om_{H},
\end{align*}
where $\om_{I}$ satisfies the inhomogeneous equation
\begin{equation*}
	\left\{
	\begin{aligned}
		&\pa_{t} \om_{I}-\nu(\pa^2_y-\al^2) \om_{I}+i\al U \om_{I}-i\al U''\phi_{I}=i\al f_1+ \pa_y f_2,\\
		&(\pa^2_y-\al^2)\phi_{I}= \om_{I},\quad \phi_{I}(t,\pm 1)=\phi'_{I}( t,\pm 1)=0,\quad  \om_{I}(0,y)=0,
	\end{aligned}\right.
\end{equation*}
and $\om_{H}$ satisfies the homogeneous equation
\begin{equation*}
	\left\{
	\begin{aligned}
		&\pa_{t} \om_{H}-\nu(\pa^2_{y}-\al^2) \om_{H}+i\al U \om_{H}-i\al U''\phi_{H}=0,\\
		&(\pa^2_y-\al^2)\phi_{H}= \om_{H},\quad \phi_{H}(t,\pm 1)=\phi'_{H}( t,\pm 1)=0,\quad  \om_{H}(0,y)= \om_{\al}^{\mathrm{in}}(y).
	\end{aligned}\right.
\end{equation*}
Applying the Laplace transform, the inhomogeneous problem reduces to resolvent estimates for the {\bf Orr-Sommerfeld equation} with non-slip boundary condition:
\begin{equation*}
	\left\{
	\begin{aligned}
		&-\nu(\pa^2_y-\al^2)w+i\al(U-\la)w-i\al U''\psi=F,\\
		&(\pa^2_y-\al^2)\psi=w,\quad \psi(\pm 1)=\psi'( \pm 1)=0.
	\end{aligned}\right.
\end{equation*}
To decouple the boundary layer effect and the critical layer effect,  the solution can be further decomposed as
\begin{align*}
	w=w_{Na}+c_1w_{cor,1}+c_2w_{cor,2},
\end{align*}
where $w_{Na}$ solves the Orr-Sommerfeld equation with Navier-slip boundary condition
\begin{equation}\label{equ:os,navier-slip,intro}
\left\{
\begin{aligned}
&-\nu(\partial^2_y-\alpha^2)w_{Na}+i\alpha(U-\lambda)w_{Na}-i\alpha U''\psi_{Na}=F,\\
&(\partial^2_y-\alpha^2)\psi=w_{Na},\quad\psi_{Na}(\pm 1)= w_{Na}(\pm 1)=0,
\end{aligned}\right.
\end{equation}
and the boundary layer correctors $w_{cor,j}$ ($j=1,2$) satisfy the homogeneous equations
\begin{equation*}
	\left\{
	\begin{aligned}
		&-\nu(\pa^2_y-\al^2)w_{cor,j}+i\al(U-\la)w_{cor,j}-i\alpha U''\psi_{cor,j}+w_{cor,j}=0,\\
		&(\pa^2_y-\al^2)\psi_{cor,j}=w_{cor,j},\quad\psi_{cor,j}(\pm 1)=0,\\
        & \psi'_{cor,1}(1)=\psi'_{cor,2}(-1)=1,\quad\psi'_{cor,1}(-1)=\psi'_{cor,2}(1)=0.
	\end{aligned}\right.
\end{equation*}
The main advantage of this decomposition is that (1) due to the favorable boundary condition for $w_{Na}$, we can use the energy method to deal with it; (2) the boundary layer correctors $w_{cor,j}$ can be approximated via Airy functions; (3) the coefficients $c_1, c_2$ depend only on $w_{Na}$.

Having decomposed $w$ as above, we must derive sharp, comprehensive resolvent bounds under Navier-slip boundary conditions in order to handle the boundary-layer problem and ultimately obtain the corresponding bounds for the non-slip case.  The main difficulty in our setting stems from the non-local term $i\al U''\Delta^{-1}w$. If $U(y)=y$ or $U(y)=y^{2}$, then  $U''(y)$  would be identically zero or constant, and the Orr-Sommerfeld equation would simplify dramatically; for general shear profiles, however, one has to uncover the intrinsic structure and prove coercive estimates before the desired resolvent inequalities can be reached.  We now outline the argument, emphasizing the key ideas and novel ingredients.
\begin{itemize}
\item {\bf  Coercive estimates for the Rayleigh equation}. We begin by analyzing the Rayleigh equation:
\begin{equation*}
\left\{
\begin{aligned}
&i\alpha(U-\lambda)w-i\alpha U''\psi_{1}=F,\\
&(\partial^2_y-\alpha^2)\psi_{1}=w,\quad \psi_{1}(y_{1})=\psi_{1}(y_{2})=w(\pm 1)=0,
\end{aligned}\right.
\end{equation*}
where $-1\leq y_{1}\leq 0\leq y_{2}\leq 1, y_{1}=-y_{2}$ with $U(y_i)=\lambda\in [U(0),U(1)]$, $i=1,2$.
A key ingredient, stated in Lemma \ref{lemma:in,coercive estimate}, is the coercive estimate for the Rayleigh equation
 \begin{align}
\int_{y_1}^{y_2}(\la-U)^2\lrv{\lrs{\frac{\psi_{1}}{\la-U}}^{\prime}}^2 d y+|\al|^{2}\int_{y_1}^{y_2} |\psi_{1}|^2 d y \leq & \bblra{\frac{(\la-U)w}{U''},\chi w}+\lra{\psi_{1},\chi w} ,\label{equ:coercive,Ray,1,intro}\\
(y_{2}-y_{1})^{2}\lrs{\bblra{\frac{(\la-U)w}{U''},\chi w}-\lra{\psi_{1}, \chi w}}\lesssim&\bblra{\frac{(\la-U)w}{U''},\chi w}+\lra{\psi_{1},\chi w},\label{equ:coercive,Ray,2,intro}
\end{align}
where $\chi(y)=1_{(y_{1},y_{2})}(y)$.
These estimates are crucial for controlling the interior norm $\n{w_{Na}}_{L^{2}(y_{1},y_{2})}$. For example, testing the Orr-Sommerfeld equation \eqref{equ:os,navier-slip,intro} against $\frac{\chi w_{Na}}{U''}$, and applying coercive estimate \eqref{equ:coercive,Ray,2,intro}, we directly obtain
\begin{align*}
(y_{2}-y_{1})^{2}\n{\sqrt{\la-U}w_{Na}}_{L^{2}(y_{1},y_{2})}^{2}\lesssim |\al|^{-1}\babs{\operatorname{Im}\blra{F+\nu(\partial^2_y-\alpha^2)w_{Na},\frac{\chi w_{Na}}{U''}}}.
\end{align*}

To fully exploit the coercivity and obtain sharp bounds, we further derive two auxiliary estimates adapted to the shear flow profile, that is, a Hardy-type inequality and a pointwise bound for the stream function $\psi_{1}$:
\begin{align}\label{equ:hardy,est,psi1,U,intro}
\int^{y_{2}}_{y_{1}}\frac{|\psi_1|^2}{(\la-U)^2}dy
&\lesssim \frac{1}{(y_{2}-y_{1})^2}\int^{y_{2}}_{y_{1}}\big(\la-U\big) ^2\lrv{\lrs{\frac{\psi_1}{U-\lambda}}'}^{2}dy+\frac{|\psi_1(0)|^2}{(y_{2}-y_{1})^3},
\end{align}
\begin{align}\label{equ:single point estimate,in,intro}
\frac{|\psi_{1}(0)|^{2}}{(y_{2}-y_{1})^{3}}\lesssim & \frac{1}{(y_{2}-y_{1})^{2}}\lrs{\bblra{\frac{\big(\la-U\big)w}{U''},\chi w}+\lra{\psi_{1},\chi w}}
+\frac{1}{y_{2}-y_{1}}\bbabs{\int_{y_{1}+\theta}^{y_{2}-\theta}\frac{U''\psi_{1}}{\la-U}dy}^{2}.
\end{align}
Together, the coercive and auxiliary estimates furnish a complete toolkit for controlling the interior behavior of the solution, thereby bounding the relevant interior norms with explicit dependence on the spectrum.

\item {\bf Resolvent bounds with explicit spectral dependence.} While $L^2$-resolvent estimates suffice for linear stability, nonlinear stability demands full resolvent bounds with explicit dependence on the spectral parameter $\la$:
\begin{align*}
\nu^{\frac{1}{3}}|\al|^{\frac{2}{3}}\lrs{|\la-U(0)|^{\frac{1}{2}}+\nu^{\frac{1}{4}}|\al|^{-\frac{1}{4}}}^{\frac{2}{3}}\n{w_{Na}}_{L^{2}}\lesssim & \n{F}_{L^{2}},\\
\nu^{\f23}|\alpha|^{\f13}\lrs{|\lambda-U(0)|^{\f12}+\nu^{\f14}|\alpha|^{-\f14}}^{\f13}\n{(\pa_{y},|\al|)w_{Na}}_{L^{2}}\lesssim&
\n{F}_{L^{2}},\\
\nu^{\frac{2}{3}}|\al|^{\frac{1}{3}}\lrs{|\la-U(0)|^{\frac{1}{2}}
+\nu^{\frac{1}{4}}|\al|^{-\frac{1}{4}}}^{\frac{1}{3}}\n{w_{Na}}_{L^{2}}\lesssim & \n{F}_{H^{-1}}.
\end{align*}
Away from the critical layer $\la=U(0)$, we obtain the resolvent bound
$$\nu^{\frac{2}{3}}|\al|^{\frac{1}{3}}\n{w_{Na}}_{L^{2}}\lesssim \n{F}_{H^{-1}},$$
thereby recovering the same scaling previously established for monotone flows in \cite{CWZ-CMP}.

Derivation of $H^{-1}$  bounds is usually more delicate than that of  $L^{2}$ bounds: the rough cutoff  $\chi(y)=1_{(y_{1},y_{2})}(y)$ employed in the  $L^{2}$ argument is not admissible in the $H^{-1}$ setting. Thus, we localize with smooth cut-offs $\rho_{\delta}(y)$ and $\rho_{\delta}^{c}(y)$, splitting the problem into an exterior and an interior part. The starting point is the observation
\begin{align*}
  \lra{\psi,\rho_{\delta}^{c} w_{Na}}\geq& \|\psi_1'\|^2_{L^2((-1,1)\setminus (y_{1},y_{2}))}+|\alpha|^2\|\psi_1\|^2_{L^{2}((-1,1)\setminus (y_{1},y_{2}))}\\&\quad- \n{\psi}_{L^{\infty}(B_{1,2,\delta})}\n{\chi^{c}w_{Na}}_{L^{1}},
\end{align*}
The right-hand side is coercive once the error term $\n{\psi}_{L^{\infty}(B_{1,2,\delta})}\n{\chi^{c}w}_{L^{1}}$ is absorbed.  Together with the coercive estimate for the Rayleigh equation, this yields a unified framework fo $H^{-1}$ and $L^{2}$ resolvent bounds with explicit spectral dependence.

\item {\bf Weak-type resolvent estimates.} To estimate the velocity, we begin with the expression
\begin{align*}
|\lra{u_{Na},u_{Na}}|=\bbabs{\bblra{F+\nu(\partial^2_y-\alpha^2)w_{Na}+i\al U''\Delta^{-1}w_{Na},\frac{\psi_{Na}}{U-\la}}}.
\end{align*}
In the case of Couette flow, where $U''(y)=0$, this expression can be estimated directly. For general shear flows, however, the nonlocal operator
 $U''\Delta^{-1}$ cannot be treated as a negligible error.  To obtain sharper velocity estimates, we establish weak-type resolvent estimates for $\frac{w_{Na}}{U''}$ as follows
	\begin{align*}
		\babs{\blra{\frac{w_{Na}}{U''}, f }}
		\lesssim &|\al|^{-1}\sum_{j\in \lr{-1,1}}|f(j)| \n{F}_{H^{-1}}\big(|U(1)-\la|+\delta^{\frac43}\big)^{-\frac34}\delta^{-1}\\
		&+ |\al|^{-1} \n{F}_{H^{-1}} \lrs{\bn
		{\text{\bf Ray}_{\delta_1}^{-1} f}_{H^{1}}+ \delta^{-\frac{4}{3}}(|\la-U(0) |^{\frac{1}{2}} + \delta)^{\frac{1}{3}}\bn{\text{\bf Ray}_{\delta_1}^{-1} f}_{L^2}}.
	\end{align*}
Here, the factor $\frac{1}{U''}$ serves to make the nonlocal operator self-adjoint, which allows us to apply the limiting absorption principle for the Rayleigh equation.
A subsequent technical analysis based on these weak-type resolvent estimates leads to the following refined bounds for the velocity $u_{Na}$
   	\begin{align*}
    		\nu^{\frac{1}{4}}|\al|^{\frac{3}{4}}\n{u_{Na}}_{L^{2}}\lesssim & \n{F}_{L^2},\\
		\nu^{\f12}|\al|^{\f12}\n{u_{Na}}_{L^2} \lesssim&  \n{F}_{H^{-1}}.
	\end{align*}
\end{itemize}

Combining the weak-type resolvent estimates with a detailed analysis of boundary layer correctors and coefficients $c_1, c_2$, we can establish the following sharp $H^{-1}$ resolvent estimate under the non-slip boundary condition:
\begin{align} \nu^{\frac{3}{4}}\left|\al\right|^{\frac{1}{4}}\|w\|_{L^2}+\nu^{\frac{1}{2}}\left|\al\right|^{\frac{1}{2}}\|u\|_{L^2} \lesssim\|F\|_{H^{-1}}. \label{equ:reso,est,u,w,FH-1,nonslip,intro}
\end{align}
The bound \eqref{equ:reso,est,u,w,FH-1,nonslip,intro} is particularly important to prove the space-time estimates. The inhomogeneous part $\om_{I}$ is directly controlled via this estimate \eqref{equ:reso,est,u,w,FH-1,nonslip,intro}.

As for the homogeneous part $\om_{H}$, it admits a decomposition $\om_{H}=\om_{H}^{(1)}+\om_{H}^{(2)}+\om_{H}^{(3)}$ as in \cite{CLWZ-ARMA}.
 To derive sharp bounds, our argument relies on the vorticity depletion estimate from \cite{IIJ-VJM}:
\begin{align*}
|\om(t,y)|\lesssim \abs{U'(y)}^{\f74}\n{\om^{\mathrm{in}}}_{H_{\al}^{4}}+\lra{t}^{-\f78}
\n{\om^{\mathrm{in}}}_{H_{\al}^{4}},
\end{align*}
which enables effective control of the term $\om_{H}^{(1)}$. Then for the term $\om_{H}^{(2)}$,
we apply space-time estimates previously derived for the inhomogeneous problem. The term $\om_{H}^{(3)}$ is handled by combining boundary layer corrector estimates with bounds already obtained for $\om_{H}^{(1)}$.
As a consequence, we establish the inviscid damping and enhanced dissipation estimates as below
	\begin{align}\label{equ:linear,enhan,invis,est}
	&\abs{\al}\n{u_{\al}}_{L_{t,y}^{2}}^{2}+\nu^{\f12}\abs{\al}^{\f12}\n{\om_{\al}}^2_{L_{t,y}^{2}}
		\lesssim \n{\om^{\mathrm{in}}_{\al}}_{H_{\al}^{4}}^{2}.\notag
	\end{align}
	
 Building on the space-time estimates above, we further establish the weighted bound
$\n{\sqrt{1-y^{2}}\om_{\al}}_{L_{t}^{\infty}L_{y}^{2}}$, which in turn yields a pointwise estimate for $\n{u(t)}_{L_{x,y}^{\infty}}$. These estimates are novel even for the Poiseuille flow $U(y)=y^{2}$, and
serve as a key ingredient in establishing nonlinear stability. Notably, such bounds are already nontrivial in the simpler case of the Couette flow $U(y)=y$, as seen in \cite{CLWZ-ARMA}.

Finally, the asymptotic stability is a direct application of the space-time estimates \eqref{eq:ST}. See Section \ref{sec:Nonlinear Stability} for the details.

\section{Coercive estimates for Rayleigh equation}\label{sec:Rayleigh Equation}

In this section, we consider the Rayleigh equation
\begin{equation}\label{equ:Rayleigh,navier-slip}
\left\{
\begin{aligned}
&i\alpha(U-\lambda)w-i\alpha U''\psi_{1}=F,\\
&(\partial^2_y-\alpha^2)\psi_{1}=w,\quad \psi_{1}(y_{1})=\psi_{1}(y_{2})=w(\pm 1)=0.
\end{aligned}\right.
\end{equation}
Here, $-1\leq y_{1}\leq 0\leq y_{2}\leq 1, y_{1}=-y_{2}$ with $U(y_i)=\lambda\in [U(0),U(1)]$, $i=1,2$.  In subsequent analysis, we establish coercive estimates, Hardy-type inequalities, and single-point estimates specifically adapted for symmetric flows $U(y)$. These results serve as the foundation for deriving resolvent estimates for the Orr-Sommerfeld equation with the Navier-slip boundary condition.
This is the reason why we can impose the extra boundary condition $w(\pm 1)=0$ (Here we present only an a priori estimate rather than solving the Rayleigh equation).

We introduce the cutoff functions
\begin{align*}
\chi(y)=1_{(y_{1},y_{2})}(y),\quad \chi^{c}(y)=1_{(-1,1)\setminus (y_{1},y_{2})}(y).
\end{align*}
Applying integration by parts, we derive preliminary energy estimates
\begin{align}
 -\langle \psi_1,\chi w\rangle =&\|\psi_{1}'\|^2_{L^2(y_{1},y_{2})}+|\alpha|^2\|\psi_1\|^2_{L^2(y_{1},y_{2})},\label{estE2-varphi-1}\\
 -\langle \psi_1,\chi^{c}w\rangle =&\|\psi_1'\|^2_{L^2((-1,1)\setminus (y_{1},y_{2}))}+|\alpha|^2\|\psi_1\|^2_{L^{2}((-1,1)\setminus (y_{1},y_{2}))}.\label{estE1-varphi-1}
\end{align}

\begin{lemma}[Coercive estimates]\label{lemma:in,coercive estimate}
Let $(w,\psi_{1})$ be the solution to \eqref{equ:Rayleigh,navier-slip}. We have
 \begin{align}
\int_{y_1}^{y_2}(\la-U)^2\lrv{\lrs{\frac{\psi_{1}}{\la-U}}^{\prime}}^2 d y+|\al|^{2}\int_{y_1}^{y_2} |\psi_{1}|^2 d y \leq & \bblra{\frac{(\la-U)w}{U''},\chi w}+\lra{\psi_{1},\chi w} ,\label{equ:in,coercive,L2}\\
(y_{2}-y_{1})^{2}\lrs{\bblra{\frac{(\la-U)w}{U''},\chi w}-\lra{\psi_{1}, \chi w}}\lesssim&\bblra{\frac{(\la-U)w}{U''},\chi w}+\lra{\psi_{1},\chi w}. \label{equ:in,coercive,H1}
\end{align}
\end{lemma}
\begin{proof}
By integration by parts, we get
\begin{align}\label{equ:U,psi1,int,y1y2}
	\int^{y_2}_{y_1}\big(\la-U\big)^2 \lrv{\lrs{\frac{\psi_1}{\la-U}}'}^{2} dy
	&=\int_{y_1}^{y_2}| \psi_1'|^2+\frac{|U'|^2|\psi_1|^2}{(\la-U)^2}+\frac{U'\pa_y(|\psi_1|^2)}{\la-U} dy\\
	&=\int_{y_1}^{y_2}|\psi_1'|^2+\frac{U''|\psi_1|^2}{U-\la}dy.\nonumber
\end{align}
By Cauchy-Schwarz inequality, we have
\begin{align*}
	(\la-U)\frac{|w|^2}{U''}+2\psi_1 w+\frac{U''|\psi_1|^2}{\la-U}\geq 0,\quad \forall y\in(y_{1},y_{2}),
\end{align*}
which, together with \eqref{estE2-varphi-1} and \eqref{equ:U,psi1,int,y1y2}, implies that
\begin{align}\label{equ:est,Uwpsi1,int,y1y2,coercive}
&	\bblra{\frac{(\la-U)w}{U''},\chi w}+\lra{\psi_{1},\chi w}
	\geq -\lra{ \psi_1,\chi w}+\int_{y_1}^{y_2}\frac{U''|\psi_1|^2}{U-\la}dy\\
	=& \int^{y_2}_{y_1}(\la-U)^2 \lrv{\lrs{\frac{\psi_1}{U-\la}}'}^{2}dy+\int_{y_1}^{y_2}|\al|^2|\psi_1|^2 dy\geq 0.\notag
\end{align}
Hence, we complete the proof of \eqref{equ:in,coercive,L2}.

For \eqref{equ:in,coercive,H1}, using \eqref{equ:U,psi1,int,y1y2} and \eqref{equ:est,Uwpsi1,int,y1y2,coercive}, we obtain
 \begin{align*}
 \int_{y_1}^{y_2}|\psi_1'|^2dy&=\int^{y_2}_{y_1}(\la-U)^2 \lrv{\lrs{\frac{\psi_1}{\lambda-U}}'}^{2} dy+\int_{y_1}^{y_2}\frac{U''|\psi_1|^2}{\la-U}dy\nonumber\\
 	&\leq 	\bblra{\frac{(\la-U)w}{U''},\chi w}+\lra{\psi_{1},\chi w}-\int_{y_1}^{y_2}|\al|^2|\psi_1|^2 dy+\int_{y_1}^{y_2}\frac{U''|\psi_1|^2}{\la-U}dy,
 \end{align*}
 which, together with \eqref{estE2-varphi-1}, gives
 \begin{align}\label{equ:est,psi1,chiw,inty1y2}
 	-\lra{\psi_{1}, \chi w}\leq \bblra{\frac{(\la-U)w}{U''},\chi w}+\lra{\psi_{1},\chi w} +\int_{y_1}^{y_2}\frac{U''|\psi_1|^2}{\la-U}dy.
 \end{align}
  It remains to bound the last term $\int_{y_1}^{y_2}\frac{U''|\psi_1|^2}{\la-U}dy$ in \eqref{equ:est,psi1,chiw,inty1y2}. By Cauchy-Schwarz inequality and Hardy inequality \eqref{proper,f,h1,1} in Lemma \ref{lemma:hardy inequality},  we deduce
 \begin{align}\label{equ:est,psi1U,psi1chiL2,psi1'chiL2}
 	\int_{y_1}^{y_2}\frac{U''|\psi_1|^2}{\la-U}dy
        \leq&\lrc{\int_{y_1}^{y_2}U''|\psi_1|^2dy}^{\frac{1}{2}} \lrc{\int_{y_1}^{y_2}\frac{U''|\psi_1|^2}{(\la-U)^{2}}dy}^{\frac{1}{2}}\\
        \lesssim&\frac{1}{y_{2}-y_{1}}\lrc{\int_{y_1}^{y_2}|\psi_1|^2dy}^{\frac{1}{2}}\lrc{\int_{y_1}^{y_2}|\psi_{1}'|^{2}dy}^{\frac{1}{2}}\notag\\
        \leq& \frac{C_{\ve}}{(y_{2}-y_{1})^{2}}\int_{y_1}^{y_2}|\psi_1|^2dy+
        \ve\int_{y_1}^{y_2}|\psi_{1}'|^{2}dy.\notag
 \end{align}
Combining \eqref{equ:est,psi1,chiw,inty1y2} and \eqref{equ:est,psi1U,psi1chiL2,psi1'chiL2}, and using \eqref{equ:in,coercive,L2} with $|\al|\geq 1$, we have
 \begin{align*}
 	-\lra{\psi_{1}, \chi w}
 	&\leq \frac{C_{\ve}}{(y_{2}-y_{1})^{2}}\int_{y_1}^{y_2}|\psi_1|^2dy+\bblra{\frac{(\la-U)w}{U''},\chi w}+\lra{\psi_{1},\chi w}\\
 	&\lesssim \lrs{1+\frac{1}{(y_{2}-y_{1})^{2}|\al|^{2}}}\lrs{\bblra{\frac{(\la-U)w}{U''},\chi w}+\lra{\psi_{1},\chi w}}\\
 &\lesssim \frac{1}{(y_{2}-y_{1})^{2}}\lrs{\bblra{\frac{(\la-U)w}{U''},\chi w}+\lra{\psi_{1},\chi w}}.
 \end{align*}
This inequality, along with \eqref{equ:est,Uwpsi1,int,y1y2,coercive} which yields the non-negativity, gives the desired result \eqref{equ:in,coercive,H1}.
\end{proof}

\begin{lemma}[Hardy-type inequality]\label{lemma:hardy-type}
Let $(w,\psi_{1})$ be the solution to \eqref{equ:Rayleigh,navier-slip}. We have
\begin{align}\label{equ:hardy,est,psi1,U}
\int^{y_{2}}_{y_{1}}\frac{|\psi_1|^2}{(U-\la)^2}dy
&\lesssim \frac{1}{(y_{2}-y_{1})^2}\int^{y_{2}}_{y_{1}}\big(U-\la\big) ^2\lrv{\lrs{\frac{\psi_1}{U-\lambda}}'}^{2}dy+\frac{|\psi_1(0)|^2}{(y_{2}-y_{1})^3}.
\end{align}
\end{lemma}
\begin{proof}
Without loss of generality, we may assume that $\psi_1(0)=0$. Testing against a function $g$, we obtain
\begin{align}\label{equ:est,hardy,U,psi1,g,inty1y2}
		&\int_{y_1}^{y_2}\frac{\psi_1(y)}{\la-U(y)}g(y)dy=\int_{y_1}^{y_2}\int_0^y\lrs{\frac{\psi_1(z)}{\la-U(z)}}'dzg(y)dy\nonumber\\
		=&\int_{0}^{y_2}\lrs{\frac{\psi_1(z)}{\la-U(z)}}'\lrs{\int_z^{y_2}g(y)dy}dz+\int^{0}_{y_1}\lrs{\frac{\psi(z)}{\la-U(z)}}'
\lrs{\int_{y_1}^zg(y)dy}dz\nonumber\\
		\leq& \lrs{\int_0^{y_2}\lrv{\lrs{\frac{\psi_1(z)}{\la-U(z)}}'(z-y_2)}^2dz}^{\f12}\lrs{\int_0^{y_2}\lrv{\frac{\int_z^{y_2}g(y)dy}{y_2-z}}^2dz}^{\f12}\nonumber\\
		& +\lrs{\int_{y_1}^0\lrv{\lrs{\frac{\psi_1(z)}{\la-U(z)}}'(z-y_1)}^2dz}^{\f12}\lrs{\int_{y_1}^0\lrv{\frac{\int_{y_1}^zg(y)dy}{y_1-z}}^2dz}^{\f12}\nonumber\\
		\lesssim & \lrs{\int_{y_1}^{y_2}\lrv{\lrs{\frac{\psi_1(z)}{\la-U(z)}}'}^2\min\lr{(z-y_1)^2,(z-y_2)^2}dz}^{\f12}
\lrs{\int_{y_1}^{y_2}|g(y)|^2dy}^{\f12},
	\end{align}
where in the last inequality we used the Hardy inequality \eqref{lemma:hardy inequality}.
By Lemma \ref{lemma:proper,symme,flow}, we have
	\begin{align}\nonumber
		\min\lr{(z-y_1)^2,(z-y_2)^2}\sim \frac{\big(U(z)-\la\big)^2}{(y_2-y_1)^2}.
	\end{align}
	Taking $g(y)=\frac{\ol{\psi_1}(y)}{\la-U(y)}$ in \eqref{equ:est,hardy,U,psi1,g,inty1y2}, we obtain
	\begin{align*} \int_{y_1}^{y_2}\frac{|\psi_1(y)|^2}{(\la-U(y))^2}dy\lesssim\frac{1}{(y_2-y_1)^2}\int_{y_1}^{y_2}\big(U(z)-\la\big)^2
\lrv{\lrs{\frac{\psi_1(z)}{\la-U(z)}}'}^2 dz.
	\end{align*}
	
If $\psi_1(0)\neq 0$, we note that $\frac{\psi_1(y)}{\la-U(y)}=\int_0^y\lrs{\frac{\psi_1(z)}{\lambda-U(z)}}'dz+\frac{\psi_1(0)}{\la-U(0)}$. Following a similar procedure, we get
	\begin{align*}
		&\int_{y_1}^{y_2}\lrs{\int_0^y\lrs{\frac{\psi_1(z)}{\la-U(z)}}'dz+\frac{\psi_1(0)}{\la-U(0)}}g(y)dy\\
		\lesssim& \lrs{\frac{1}{(y_2-y_1)^2}\int_{y_1}^{y_2}(U(y)-\la)^2\lrv{\lrs{\frac{\psi_1(y)}{\la-U(y)}}'}^2dy}^{\f12}\lrs{\int_{y_1}^{y_2}|g(y)|^2dy}^{\f12}\nonumber\\
		& +\frac{|\psi_1(0)|}{\la-U(0)}(y_2-y_1)^{\f12}\lrs{\int_{y_1}^{y_2}|g(y)|^2dy}^{\f12},
	\end{align*}
	which, together with $\la-U(0)\sim (y_2-y_1)^2$, $g(y)=\frac{\ol{\psi_1}(y)}{\la-U(y)}$, implies \eqref{equ:hardy,est,psi1,U}.
\end{proof}

\begin{lemma}\label{lemma:single point estimate,in}
For $\theta\in(0,\frac{y_{2}-y_{1}}{4}]$, we have
\begin{align}\label{equ:single point estimate,in}
\frac{|\psi_{1}(0)|^{2}}{(y_{2}-y_{1})^{3}}\lesssim & \frac{1}{(y_{2}-y_{1})^{2}}\lrs{\bblra{\frac{\big(\la-U\big)w}{U''},\chi w}+\lra{\psi_{1},\chi w}}
+\frac{1}{y_{2}-y_{1}}\bbabs{\int_{y_{1}+\theta}^{y_{2}-\theta}\frac{U''\psi_{1}}{\la-U}dy}^{2}.
\end{align}

\end{lemma}
\begin{proof}
Due to that $U''(0)>0$, it suffices to estimate $U''(0)\psi_{1}(0)$.
By integration by parts, we have
\begin{align}\label{equ:single point estimate,in,equality,1}
&\frac{U''(0)\psi_{1}(0)(\theta-y_{2})}{\la-U(0)}
=\frac{U''(y)\psi_{1}(y)}{\la-U(y)}(y_{2}-\theta-y)\Big|^{y_{2}-\theta}_{0}\\
=&\int_{0}^{y_{2}-\theta}\lrs{\frac{\psi_{1}(y)}{\la-U(y)}}^{'}U''(y)(y_{2}-\theta-y)dy-
\int_{0}^{y_{2}-\theta}\frac{U''(y)\psi_{1}(y)}{\la-U(y)}dy\notag\\
&+\int_{0}^{y_{2}-\theta}\frac{\psi_{1}(y)}{\la-U(y)}U'''(y)(y_{2}-\theta-y)dy.\notag
\end{align}
In the same way, we have
\begin{align}\label{equ:single point estimate,in,equality,2}
&\frac{U''(0)\psi_{1}(0)(y_{1}+\theta)}{\la-U(0)}=\frac{U''(y)\psi_{1}(y)}{\la-U(y)}(y_{1}+\theta-y)\Big|_{y_{1}-\theta}^{0}\\
=&\int_{y_{1}-\theta}^{0}\lrs{\frac{\psi_{1}(y)}{\la-U(y)}}^{'}U''(y)(y_{1}+\theta-y)dy-
\int_{y_{1}-\theta}^{0}\frac{U''(y)\psi_{1}(y)}{\la-U(y)}dy\notag\\
&+\int_{y_{1}-\theta}^{0}\frac{\psi_{1}(y)}{\la-U(y)}U'''(y)(y_{1}+\theta-y)dy.\notag
\end{align}
Combining \eqref{equ:single point estimate,in,equality,1} and \eqref{equ:single point estimate,in,equality,2}, we obtain
\begin{align*}
&\frac{U''(0)\psi_{1}(0)(y_{2}-y_{1}-2\theta)}{\la-U(0)}\\
=&\int_{y_{1}+\theta}^{y_{2}-\theta}\frac{U''(y)\psi_{1}(y)}{\la-U(y)}dy-
\int_{0}^{y_{2}-\theta}\lrs{\frac{\psi_{1}(y)}{\la-U(y)}}^{'}U''(y)(y_{2}-\theta-y)dy\\
&-\int_{y_{1}+\theta}^{0}\lrs{\frac{\psi_{1}(y)}{\la-U(y)}}^{'}U''(y)(y_{1}+\theta-y)dy-\int_{0}^{y_{2}-\theta}\frac{\psi_{1}(y)}{\la-U(y)}U'''(y)(y_{2}-\theta-y)dy\\
&-\int_{y_{1}+\theta}^{0}\frac{\psi_{1}(y)}{\la-U(y)}U'''(y)(y_{1}
+\theta-y)dy.
\end{align*}
Using that $\la-U(0)\sim (y_{2}-y_{1})^{2}$, $\la-U(y)\sim y_{2}(y_{2}-y)\sim y_{2}(y-y_{1})$ as in Lemma \ref{lemma:proper,symme,flow}, we get
\begin{align*}
|U''(0)\psi_{1}(0)|^{2}\lesssim & (y_{2}-y_{1})^{2}\bbabs{\int_{y_{1}+\theta}^{y_{2}-\theta}\frac{U''(y)\psi_{1}(y)}{\la-U(y)}dy}^{2}+(y_{2}-y_{1})\int_{y_{1}+\theta}^{y_{2}-\theta}
\abs{\psi_{1}(y)}^{2}dy\\
&+(y_{2}-y_{1})\int_{y_{1}+\theta}^{y_{2}-\theta}
\bbabs{\lrs{\frac{\psi_{1}(y)}{\la-U(y)}}^{'}}^{2}(\la-U(y))^{2}dy,
\end{align*}
which, together with \eqref{equ:in,coercive,L2}, completes the proof.
\end{proof}

\section{Resolvent estimates with Navier-slip boundary condition}\label{sec:Orr-Sommerfeld Equation with the Navier-slip Boundary Condition}

In this section, we consider the Orr-Sommerfeld (OS) equation with the Navier-slip boundary condition
\begin{equation}\label{equ:OS,navier-slip}
\left\{
\begin{aligned}
&-\nu(\partial^2_y-\alpha^2)w+i\alpha(U-\lambda)w-i\alpha U''\psi=F,\\
&(\partial^2_y-\alpha^2)\psi=w,\quad\psi(\pm 1)= w(\pm 1)=0.
\end{aligned}\right.
\end{equation}
To facilitate the analysis, we decompose the stream function $\psi=\psi_1+\psi_2$, where
\begin{equation}\label{def-varphi1}
\left\{
\begin{aligned}
&(\partial_y^2-|\alpha|^2)\psi_1=w,\\
&\psi_1(\pm 1)=0, \psi_1(y_{1})=\psi_{1}(y_{2})=0,
\end{aligned}\right.
\end{equation}
and
\begin{equation}\label{def-varphi2}
    \left\{
    \begin{aligned}
         &(\partial_y^2-|\alpha|^2)\psi_2=0,\\
         &\psi_2(\pm 1)=0,\psi_2(y_i)=\psi(y_i),\ i=1,2.
  \end{aligned}\right.
\end{equation}
In fact, $\psi_{2}(y)$ is explicitly given by
\begin{equation}\label{form-varphi2}
    \psi_2(y)=
    \left\{
    \begin{aligned}
          &\frac{\operatorname{sinh}\big(|\alpha|(1+y)\big)}
          {\operatorname{\operatorname{sinh}}\big(|\alpha|(1+y_{1})\big)}\psi(y_{1}),\quad & y\in[-1,y_{1}],\\
          &\frac{\operatorname{\operatorname{sinh}}\big(|\alpha|(y-y_{1})\big)\psi(y_{2})
          +\operatorname{\operatorname{sinh}}
          \big(|\alpha|(y_{2}-y)\big)\psi(y_{1})}{\operatorname{\operatorname{sinh}}\big(|\alpha|(y_{2}-y_{1})\big)},\quad &y\in[y_{1},y_{2}],\\
          &\frac{\operatorname{\operatorname{sinh}}\big(|\alpha|(1-y)\big)}{\operatorname{\operatorname{sinh}}\big(|\alpha|(1-y_{2})\big)}\psi(y_{2}),\quad & y\in[y_{2},1].\\
    \end{aligned}\right.
\end{equation}

The main goal is to establish full resolvent estimates for the OS equation \eqref{equ:OS,navier-slip}. The proof is intricate and consists of several technical steps, such as exterior and interior estimates, weak-type resolvent estimates which are employed to enhance the velocity estimates. These steps are organized into Subsection \ref{sec:Resolvent Estimates for}--\ref{sec:Resolvent Estimates for F}.
Finally, we summarize the main results in Subsection \ref{sec:Summary of Results under the Navier-slip Boundary Condition}.

\subsection{Resolvent estimates for $\la\in \R\setminus (U(0),U(1))$}\label{sec:Resolvent Estimates for}
We provide energy estimates and resolvent estimates for $\la\in \R\setminus (U(0),U(1))$.
\begin{lemma}\label{lemma:energy estimate,real part}
For $\la\in \R$, we have
\begin{align}
    \nu\n{\pa_yw}_{L^{2}}^{2}+\nu\alpha^{2}\n{w}^2_{L^2}\lesssim& \nu\n{w}^2_{L^2}+ |\operatorname{Re}\blra{F,\frac{w}{U''}}|,
    \label{equ:energy estimate,real part,w,nu}\\
    |\al|\n{(\pa_{y},|\al|)w}_{L^{2}}   \lesssim& \n{w}_{L^2}+\nu^{-1}\n{F}_{L^{2}},\label{equ:energy estimate,real part,w,L2}\\
    \n{(\pa_{y},|\al|)w}_{L^{2}}   \lesssim& \n{w}_{L^2}+\nu^{-1}\n{F}_{H^{-1}}.\label{equ:energy estimate,real part,w,H-1}
\end{align}

For $\la\leq U(0)+\nu^{\frac{1}{2}}|\al|^{-\frac{1}{2}}$, we have
\begin{align}\label{equ:energy estimate,im,w,0}
|\al||\la-U(0)|\n{w}_{L^{2}}^{2}+2\nu^{\frac{1}{2}}|\al|^{\frac{1}{2}}\n{w}_{L^{2}}^{2}\lesssim \babs{\blra{F,\frac{w}{U''}}}.
\end{align}

For $\la\geq U(1)-\nu^{\frac{1}{3}}|\al|^{-\frac{1}{3}}$, we have
\begin{align}\label{equ:energy estimate,im,w,1}
|\al||\la-U(1)|\n{w}_{L^{2}}^{2}+2\nu^{\frac{1}{3}}|\al|^{\frac{2}{3}}\n{w}_{L^{2}}^{2}\lesssim \babs{\blra{F,\frac{w}{U''}}}.
\end{align}
\end{lemma}
\begin{proof}
We test equation $\eqref{equ:OS,navier-slip}$ by $\frac{w}{U''}$ to get
 \begin{align}\label{equ:F,test,w}
\blra{F,\frac{w}{U''}}=\blra{-\nu(\pa_{y}^{2}-\al^{2})w+i\al (U-\la) w-i\al U''\psi,  \frac{w}{U''}}.
\end{align}
 Taking the real part of \eqref{equ:F,test,w}, we use the integration by parts to obtain
\begin{align*}
    &\operatorname{Re}\lra{F,\frac{w}{U''}}=\int_{-1}^{1}\frac{\nu|\pa_yw|^{2}+\nu|\alpha|^{2}|w|^{2}}{U''}+\nu \operatorname{Re}\bblra{\pa_yw,w\lrs{\frac{1}{U''}}'}.
\end{align*}
Noticing that
\begin{align*}
\operatorname{Re}\bblra{\pa_yw,w\lrs{\frac{1}{U''}}'}=-\int_{-1}^{1}\frac{|w|^2}{2}\Big(\frac{1}{U''}\Big)''dy,
\end{align*}
by the condition that $U''(y)>0$, we have
\begin{align*}
\nu\n{\pa_yw}_{L^{2}}^{2}+\nu\alpha^{2}\n{w}_{L^{2}}^{2}\lesssim \nu\n{w}^2_{L^2}+\bbabs{\operatorname{Re}\lra{F,\frac{w}{U''}}},
\end{align*}
which completes the proof of \eqref{equ:energy estimate,real part,w,nu}. Using \eqref{equ:energy estimate,real part,w,nu} and Young's inequality,
we have
\begin{align*}
\nu\n{\pa_yw}_{L^{2}}^{2}+\nu\alpha^{2}\n{w}_{L^{2}}^{2}\lesssim \nu\n{w}^2_{L^2}+ o\lrs{\nu|\al|^{2}}\n{w}_{L^{2}}^{2}+ \frac{\n{F}_{L^{2}}^{2}}{\nu|\al|^{2}},
\end{align*}
which implies \eqref{equ:energy estimate,real part,w,L2}.
Moreover, we have
\begin{align*}
\babs{\blra{F,\frac{w}{U''}}}\leq \n{F}_{H^{-1}}\n{(\pa_{y},|\al|)w}_{L^{2}}\bbn{\frac{1}{U''}}_{W^{1,\infty}}\lesssim \n{F}_{H^{-1}}\n{(\pa_{y},|\al|)w}_{L^{2}},
\end{align*}
and, together with \eqref{equ:energy estimate,real part,w,nu}, arrive at \eqref{equ:energy estimate,real part,w,H-1}.

Taking the imaginary part of \eqref{equ:F,test,w}, we use the integration by parts to obtain
\begin{align}\label{equ:F,test,w,im}
\operatorname{Im}\blra{F,\frac{w}{U''}}
=&\al \lrs{\blra{w(U-\la),\frac{w}{U''}}+\al^{2}\lra{\psi,\psi}+\lra{\pa_{y}\psi,\pa_{y}\psi}} -\nu\operatorname{Im}\blra{w',w\frac{U'''}{(U'')^{2}}}.
\end{align}

For resolvent estimate \eqref{equ:energy estimate,im,w,0}, it suffices to consider $\la\leq U(0)$. By \eqref{equ:F,test,w,im}, we have
\begin{align*}
&|\al|\lrs{|\la-U(0)|\n{w}_{L^{2}}^{2}+\blra{(U-U(0))w,\frac{w}{U''}}+\n{(\pa_{y},|\al|)\psi}_{L^{2}}^{2}}\\
\lesssim& \babs{\operatorname{Im}\blra{F,\frac{w}{U''}}}+\nu\n{\pa_yw}_{L^{2}}\n{w}_{L^{2}},
\end{align*}
which together with \eqref{equ:energy estimate,real part,w,nu} implies that
\begin{align}\label{equ:energy estimate,all}
&|\al|\lrs{|\la-U(0)|\n{w}_{L^{2}}^{2}+\blra{(U-U(0))w,\frac{w}{U''}}+\n{(\pa_{y},|\al|)\psi}_{L^{2}}^{2}}+\nu\n{(\pa_{y},|\al|)w}_{L^{2}}^{2}\\
\lesssim&\babs{\blra{F,\frac{w}{U''}}}+\nu\n{\pa_yw}_{L^{2}}\n{w}_{L^{2}}+\nu \n{w}_{L^{2}}^{2}.\notag
\end{align}
Using that $y^{2}\lesssim |U(y)-U(0)|$ as in Lemma \ref{lemma:proper,symme,flow}, we have
\begin{align}\label{equ:L2,w,y,U}
 \|w\|^2_{L^2}=&\langle 1,w^2\rangle=\langle -y,\partial_y|w|^2\rangle
    \leq 2\|yw\|_{L^2}\|\partial_yw\|_{L^2}
    \lesssim  \blra{(U-U(0))w,w}^{\frac{1}{2}}\n{\pa_yw}_{L^{2}}.
\end{align}
Combining \eqref{equ:energy estimate,all} and \eqref{equ:L2,w,y,U}, we use Young's inequality to get
\begin{align}\label{equ:L2,w,resolvent,0,full,proof}
|\al|\lrs{|\la-U(0)|\n{w}_{L^{2}}^{2}+\blra{(U-U(0))w,\frac{w}{U''}}
+\n{(\pa_{y},|\al|)\psi}_{L^{2}}^{2}}+\nu\n{(\pa_{y},|\al|)w}_{L^{2}}^{2}
\lesssim&\babs{\blra{F,\frac{w}{U''}}},
\end{align}
which together with \eqref{equ:L2,w,y,U} implies that
\begin{align}\label{equ:L2,w,resolvent,0,proof}
|\al||\la-U(0)|\n{w}_{L^{2}}^{2}+2\nu^{\frac{1}{2}}|\al|^{\frac{1}{2}}\n{w}_{L^{2}}^{2}\lesssim \babs{\blra{F,\frac{w}{U''}}}.
\end{align}
Therefore, we complete the proof of \eqref{equ:energy estimate,im,w,0}.

For resolvent estimate \eqref{equ:energy estimate,im,w,1}, it suffices to consider $\la\geq U(1)$.
By interior coercive estimate \eqref{equ:in,coercive,H1} in Lemma \ref{lemma:in,coercive estimate}, we have
\begin{align*}
\blra{(U(1)-U)w,\frac{w}{U''}}-\n{(\pa_{y},|\al|)\psi}_{L^{2}}^{2}\gtrsim \blra{(U(1)-U)w,\frac{w}{U''}}+\n{(\pa_{y},|\al|)\psi}_{L^{2}}^{2}.
\end{align*}
 Repeating the proof of \eqref{equ:energy estimate,all}, we get
\begin{align*}
&|\al|\lrs{|\la-U(1)|\n{w}_{L^{2}}^{2}+\blra{(U(1)-U)w,\frac{w}{U''}}-\n{(\pa_{y},|\al|)\psi}_{L^{2}}^{2}}
+\nu\n{(\pa_{y},|\al|)w}_{L^{2}}^{2}\\
\lesssim& \babs{\blra{F,\frac{w}{U''}}}+\nu\n{\pa_yw}_{L^{2}}\n{w}_{L^{2}}+\nu \n{w}_{L^{2}}^{2},
\end{align*}
and hence,
\begin{align}\label{equ:energy estimate,all,la=1}
&|\al|\lrs{|\la-U(1)|\n{w}_{L^{2}}^{2}+\blra{(U(1)-U)w,\frac{w}{U''}}+\n{(\pa_{y},|\al|)\psi}_{L^{2}}^{2}}
+\nu\n{(\pa_{y},|\al|)w}_{L^{2}}^{2}\\
\lesssim& \babs{\blra{F,\frac{w}{U''}}}+\nu\n{\pa_yw}_{L^{2}}\n{w}_{L^{2}}+\nu \n{w}_{L^{2}}^{2}.\notag
\end{align}
Due to that $1-y^{2}\lesssim U(1)-U(y)$ as in Lemma \ref{lemma:proper,symme,flow}, we use Hardy-type inequality \eqref{proper,f,h1,3} in Lemma \ref{lemma:hardy inequality} to obtain
\begin{align}\label{equ:L2,w,y,hardy}
 \|w\|^2_{L^2}
    \lesssim \n{\sqrt{1-y^{2}}w}_{L^{2}}^{\frac{4}{3}}\n{\pa_yw}_{L^{2}}^{\frac{2}{3}}\lesssim \blra{(U(1)-U)w,w}^{\frac{2}{3}}\n{\pa_yw}_{L^{2}}^{\frac{2}{3}}.
\end{align}
Combining \eqref{equ:energy estimate,all,la=1} and \eqref{equ:L2,w,y,hardy}, we use Young's inequality to get
\begin{align}\label{equ:L2,w,resolvent,1,full,proof}
&|\al|\lrs{|\la-U(1)|\n{w}_{L^{2}}^{2}+\blra{(U(1)-U)w,\frac{w}{U''}}+\n{(\pa_{y},|\al|)\psi}_{L^{2}}^{2}}
+\nu\n{(\pa_{y},|\al|)w}_{L^{2}}^{2}\\
\lesssim& \bbabs{\blra{F,\frac{w}{U''}}},\notag
\end{align}
which together with \eqref{equ:L2,w,y,hardy} implies the desired estimate
\begin{align*}
|\al||\la-U(1)|\n{w}_{L^{2}}^{2}+2\nu^{\frac{1}{3}}|\al|^{\frac{2}{3}}\n{w}_{L^{2}}\lesssim \bbabs{\blra{F,\frac{w}{U''}}}.
\end{align*}
Hence, we complete the proof of \eqref{equ:energy estimate,im,w,1}.
\end{proof}

\subsection{Resolvent estimates for $\lambda\in [U(0), U(1)]$}

\subsubsection{Exterior energy estimates}

In this subsection, we focus on the estimates on the exterior domain $(-1,1)\setminus (y_{1},y_{2})$ with $\la\in [U(0),U(1)]$. We define the smooth cutoff function
\begin{equation}\label{equ:smooth cutoff,ex,delta}
\rho_{\delta}^{c}(y)= \begin{cases}0, & y \in\left(y_{1}-\frac{\delta}{2}, y_{2}+\frac{\delta}{2}\right), \\ 1, & y \in(-1,1) \backslash\left(y_{1}-\delta, y_{2}+\delta\right), \\ \text { smooth, } & \text {others, }\end{cases}
\end{equation}
where $0<\delta\leq y_{2}\leq 1$ and $ 0\leq \rho_{\delta}^{c}(y)\leq 1$.
For simplicity, we set
\begin{itemize}
\item $\chi(y)=1_{(y_{1},y_{2})}(y),\quad \chi^{c}(y)=1_{(-1,1)\setminus (y_{1},y_{2})}(y).$
\item $B_{1,2,\delta}=B(y_{1},\delta)\cup B(y_{2},\delta)$.
\end{itemize}

From the expression \eqref{form-varphi2}, we have that $\n{\psi_{2}}_{L^{\infty}}\leq \n{\psi}_{L^{\infty}(B_{1,2,\delta})}$ and the local $L^{\infty}$ estimate
\begin{align}\label{equ:local Linfity estimate}
\n{\psi_{1}}_{L^{\infty}(B_{1,2,\delta})}+\n{\psi_{2}}_{L^{\infty}}\leq 2\n{\psi}_{L^{\infty}(B_{1,2,\delta})}.
\end{align}

The following lemma provides exterior energy estimates.
\begin{lemma}[Exterior energy estimates]\label{lemma:exter,estimate,w,psi}
Let $\delta\in (0, y_{2}]$ and $(w,\psi)$ be the solution to \eqref{equ:OS,navier-slip}.
We have
\begin{align}\label{equ:ex,estimate,w,L2}
        \int_{-1}^{1}\rho_{\delta}^{c}|w|^2dy\leq &\frac{|\operatorname{Im}\lra{F,\frac{\rho_{\delta}^{c} w}{U-\la}}|}{|\al|}+
        \frac{\nu\|\partial_y w\|_{L^2}\n{w}
        _{L^\infty}}{|\alpha|\delta^{\f32}(y_{2}-y_{1})} +\int_{-1}^{1}\frac{|\psi|^2\rho_{\delta}^{c}}{(U-\lambda)^2}dy,
\end{align}
\begin{align}\label{equ:ex,estimate,w}
  \int_{-1}^{1}\frac{\rho_{\delta}^{c}(U-\lambda)|w|^2}{U''}dy-\lra{\psi_{1},\chi^{c}w}
 \lesssim &\frac{|\lra{F,\frac{\rho_{\delta}^{c}w}{U''}}|}{|\alpha|}+\frac{\nu\|\partial_yw\|_{L^2}
 \n{w}_{L^\infty}}{|\alpha|\delta^{\frac{1}{2}}}
   +\n{\psi}_{L^{\infty}(B_{1,2,\delta})}\n{\chi^{c}w}_{L^{1}}.
\end{align}
In particular, we have
\begin{align}\label{equ:ex,estimate,w,psi}
&\frac{1}{(y_{2}-y_{1})^{2}}\int_{(-1,1)\setminus (y_{1},y_{2})} |\al|^{2}|\psi_{1}|^2+|\pa_y\psi_{1}|^{2} d y\\
\lesssim& \frac{|\lra{F,\frac{\rho_{\delta}^{c}w}{U''}}|}{|\alpha|(y_{2}-y_{1})^{2}}+\mathcal{E}(w)
 +\frac{\n{\chi^{c}w}_{L^{1}}^{2}}{y_{2}-y_{1}}+
        \frac{\n{\psi}_{L^{\infty}(B_{1,2,\delta})}^{2}}{(y_{2}-y_{1})^{2}\delta}.\notag
\end{align}
where
\begin{align}\label{equ:E,w,effective}
\mathcal{E}(w)=\frac{\nu^{2}\n{\pa_yw}_{L^{2}}^{2}}{|\al|^{2}(y_{2}-y_{1})^{2}\delta^{4}}
        +\delta\n{w}_{L^{\infty}}^{2}.
\end{align}
\end{lemma}
\begin{proof}
Testing equation \eqref{equ:OS,navier-slip} by $\frac{\rho_{\delta}^{c} w}{U-\la}$, we have
\begin{align*}
\lra{F,\frac{\rho_{\delta}^{c} w}{U-\la}}=&\lra{-\nu(\pa_{y}^{2}-\al^{2})w+i\al ((U-\la) w-U''\psi),\frac{\rho_{\delta}^{c} w}{U-\la}}
\end{align*}
and take the imaginary part to obtain
\begin{align*}
\operatorname{Im}\lra{F,\frac{\rho_{\delta}^{c} w}{U-\la}}=&\al \lra{w,\rho_{\delta}^{c} w}+\nu \operatorname{Im}\lra{\pa_yw,w\lrs{\frac{\rho_{\delta}^{c}}{U-\la}}'}-\al
\operatorname{Re}\lra{U''\psi,\frac{\rho_{\delta}^{c} w}{U-\la}}.
\end{align*}
By H\"{o}lder inequality and estimate \eqref{equ:basic estimate,U,rho,H1} in Lemma \ref{lemma:basic Lp estimate,U}, we get
\begin{align*}
        \int_{-1}^{1}\rho_{\delta}^{c}|w|^2dy\lesssim &\frac{|\operatorname{Im}\lra{F,\frac{\rho_{\delta}^{c} w}{U-\la}}|}{|\al|}+
        \frac{\nu\|\partial_y w\|_{L^2}\n{w}
        _{L^\infty}}{|\alpha|\delta^{\f32}(y_{2}+\delta)} +\int_{-1}^{1}\frac{|\psi|^2\rho_{\delta}^{c}}{(U-\lambda)^2}dy,
\end{align*}
which completes the proof of \eqref{equ:ex,estimate,w,L2}.

In a similar way, testing equation \eqref{equ:OS,navier-slip} by $\frac{\rho_{\delta}^{c} w}{U''}$, we take the imaginary part to obtain
\begin{align}\label{equ:F,test,w,ex}
\operatorname{Im}\lra{F,\frac{\rho_{\delta}^{c} w}{U''}}=&\nu\operatorname{Im} \lra{\pa_{y}w,w\lrs{\frac{\rho_{\delta}^{c}}{U''}}'}+\al\operatorname{Re}\lra{(U-\la)w-U''\psi,\frac{\rho_{\delta}^{c} w}{U''}}\\
=&\nu \operatorname{Im}\lra{\pa_{y}w,w\lrs{\frac{\rho_{\delta}^{c}}{U''}}'}+\al\operatorname{Re}\lra{\frac{(U-\la)w}{U''},\rho_{\delta}^{c} w}-\al \operatorname{Re}\lra{\psi,\rho_{\delta}^{c} w}.\notag
\end{align}
Noticing that
\begin{align*}
    \lra{\psi,\rho_{\delta}^{c} w}=\lra{\psi_{1},\chi^{c}w}+\lra{\psi_{1},(\rho_{\delta}^{c}-\chi^{c})w}+\lra{\psi_{2},\rho_{\delta}^{c}w},
\end{align*}
 we use H\"{o}lder inequality and the local $L^{\infty}$ estimate \eqref{equ:local Linfity estimate} to get
\begin{align}\label{equ:F,test,w,ex,main part}
  |\lra{\psi,\rho_{\delta}^{c} w}- \lra{\psi_{1},\chi^{c}w}|\lesssim \n{\psi}_{L^{\infty}(B_{1,2,\delta})}\n{\chi^{c}w}_{L^{1}}.
\end{align}
Combining \eqref{equ:F,test,w,ex} and \eqref{equ:F,test,w,ex,main part}, we use H\"{o}lder inequality and estimate \eqref{equ:basic estimate,U,rho,L2} in Lemma \ref{lemma:basic Lp estimate,U} to get
\begin{align*}
\int_{-1}^{1}\frac{\rho_{\delta}^{c}(U-\lambda)|w|^2}{U''}dy-\lra{\psi_{1},\chi^{c}w}
\lesssim&\frac{|\lra{F,\frac{\rho_{\delta}^{c}w}{U''}}|}{|\alpha|}+\frac{\nu\|\partial_yw\|_{L^2}\n{w}_{L^\infty}}{|\al|\delta^{\frac{1}{2}}}
+\n{\psi}_{L^{\infty}(B_{1,2,\delta})}\n{\chi^{c}w}_{L^{1}},
\end{align*}
which completes the proof of \eqref{equ:ex,estimate,w}.
Applying Cauchy-Schwarz inequality to \eqref{equ:ex,estimate,w} and using that $\delta\leq y_{2}$, we obtain \eqref{equ:ex,estimate,w,psi}.
\end{proof}

From \eqref{equ:ex,estimate,w,L2}, to provide exterior $L^{2}$ estimate on $w$, we need to estimate
$\int_{-1}^{1}\frac{|\psi|^2\rho_{\delta}^{c}}{(U-\lambda)^2}dy$.

\begin{lemma}[Exterior $L^{2}$ weighted estimate on $\psi$]\label{lemma:ex,l2 weighted,psi}
Let $\delta\in (0, y_{2}]$. It holds that
\begin{align}\label{equ:ex,l2 weighted,psi}
\int_{-1}^{1}\frac{|\psi|^2\rho_{\delta}^{c}}{(U-\lambda)^2}dy\lesssim
\frac{|\lra{F,\frac{\rho_{\delta}^{c}w}{U''}}|}{|\alpha|(y_{2}-y_{1})\delta}+
\mathcal{E}(w)
        +\frac{\n{\chi^{c}w}_{L^{1}}^{2}}{y_{2}-y_{1}}+
        \frac{\n{\psi}_{L^{\infty}(B_{1,2,\delta})}^{2}}{(y_{2}-y_{1})^{2}\delta}.
\end{align}

\end{lemma}
\begin{proof}
Since $\psi=\psi_{1}+\psi_{2}$, we have
\begin{align}\label{equ:ex,l2 weighted,psi,1+2}
\int_{-1}^{1}\frac{|\psi|^2\rho_{\delta}^{c}}{(U-\lambda)^2}dy\leq
\int_{-1}^{1}\frac{|\psi_{1}|^2\rho_{\delta}^{c}}{(U-\lambda)^2}dy+
\int_{-1}^{1}\frac{|\psi_{2}|^2\rho_{\delta}^{c}}{(U-\lambda)^2}dy.
\end{align}
By the local $L^{\infty}$ estimate \eqref{equ:local Linfity estimate} and estimate \eqref{equ:basic estimate,U,L2} in Lemma \ref{lemma:basic Lp estimate,U}, we get
\begin{align}\label{equ:ex,l2 weighted,psi2,delta}
\bbabs{\int_{-1}^{1}\frac{|\psi_{2}|^2\rho_{\delta}^{c}}{(U-\lambda)^2}dy}\lesssim \frac{\n{\psi}_{L^{\infty}(B_{1,2,\delta})}^{2}}{(y_{2}-y_{1})^{2}\delta}.
\end{align}

On the other hand, using \eqref{u-y2} in Lemma \ref{lemma:proper,symme,flow} that
$$|U(y)-\lambda|\sim |(y+y_{2})(y-y_{2})|\sim |(y+y_{1})(y-y_{1})|,$$
we obtain
\begin{align*}
    \int_{-1}^{1}\frac{|\psi_{1}|^2\rho_{\delta}^{c}}{(U-\lambda)^2}dy&\leq
   \int_{-1}^{y_{1}-\frac{\delta}{2}}\frac{|\psi_1|^2}{(U-\lambda)^2}dy+\int_{y_{2}+\frac{\delta}{2}}^1\frac{|\psi_1|^2}{(U-\lambda)^2}dy\nonumber\\
    &\lesssim \int_{-1}^{y_{1}-\frac{\delta}{2}}\frac{|\psi_1|^2}{(y-y_{1})^2(y+y_{1})^2}dy+\int_{y_{2}+\frac{\delta}{2}}^1\frac{|\psi_1|^2}{(y-y_{2})^2(y+y_{2})^2}dy\nonumber\\
    &\lesssim \frac{1}{(y_{2}+\frac{\delta}{2})^2}\int_{-1}^{y_{1}}\frac{|\psi_1|^2}{(y-y_{1})^2}dy+\frac{1}{(y_{2}+\frac{\delta}{2})^2}
    \int_{y_{2}}^1\frac{|\psi_1|^2}{(y-y_{2})^2}dy\nonumber\\
    &\lesssim \frac{1}{(y_{2}+\frac{\delta}{2})^2}\int_{(-1,1)\setminus(y_{1},y_{2})}|\pa_y\psi_{1}|^2dy,
\end{align*}
where in the last inequality we have used the boundary condition $\psi_{1}(\pm 1)=\psi_{1}(y_{1})=\psi_{1}(y_{2})=0$ and Hardy inequality. Then by exterior energy estimate \eqref{equ:ex,estimate,w,psi}, we arrive at
\begin{align*}
\int_{-1}^{1}\frac{|\psi_{1}|^2\rho_{\delta}^{c}}{(U-\lambda)^2}dy\lesssim \frac{|\lra{F,\frac{\rho_{\delta}^{c}w}{U''}}|}{|\alpha|(y_{2}-y_{1})\delta}+
\mathcal{E}(w)
        +\frac{\n{\chi^{c}w}_{L^{1}}^{2}}{y_{2}-y_{1}}+
        \frac{\n{\psi}_{L^{\infty}(B_{1,2,\delta})}^{2}}{(y_{2}-y_{1})^{2}\delta},
\end{align*}
which completes the proof of \eqref{equ:ex,l2 weighted,psi}.
\end{proof}

We are left to estimate $\n{\chi^{c}w}_{L^{1}}$ and $\n{\psi}_{L^{\infty}(B_{1,2,\delta})}$.
\begin{lemma}[Exterior $L^{1}$ estimate on $w$]
Let $\delta\in (0, y_{2}]$. It holds that
\begin{align}\label{equ:estimate,w,L1,ex}
\frac{\n{\chi^{c}w}_{L^{1}}^{2}}{y_{2}-y_{1}}\lesssim \frac{|\lra{F,\frac{\rho_{\delta}^{c}w}{U''}}|}{|\alpha|(y_{2}-y_{1})\delta}+\mathcal{E}(w)
+\frac{\n{\psi}_{L^{\infty}(B_{1,2,\delta})}^{2}}{(y_{2}-y_{1})^{2}\delta}.
\end{align}
\end{lemma}
\begin{proof}
First, we have
\begin{align*}
\frac{\n{\chi^{c}w}_{L^{1}}^{2}}{y_{2}-y_{1}}
\leq& \frac{\n{\rho_{\delta}^{c}w}^2_{L^{1}}}{y_{2}-y_{1}}+ \frac{\delta^{2}\n{w}_{L^{\infty}}^{2}}{y_{2}-y_{1}}\leq
\frac{\n{\rho_{\delta}^{c}w}^2_{L^{1}}}{y_{2}-y_{1}}+\mathcal{E}(w).
\end{align*}
By exterior energy estimate \eqref{equ:ex,estimate,w} and \eqref{equ:basic estimate,U,L1} in Lemma \ref{lemma:basic Lp estimate,U}, we get
\begin{align*}
 \n{\rho_{\delta}^{c}w}^2_{L^{1}}
   \leq& \lrc{\int\frac{U''\rho_{\delta}^{c}}{U-\lambda}dy}\lrc{\int \frac{\rho_{\delta}^{c}\lrs{U-\lambda}|w|^{2}}{U''}dy}\\
   \lesssim& \frac{\mathrm{ln}(1+\frac{2y_{2}}{\delta})}{y_{2}-y_{1}} \lrs{ \frac{|\lra{F,\frac{\rho_{\delta}^{c}w}{U''}}|}{|\alpha|}+\frac{\nu\|\partial_yw\|_{L^2}\n{w}_{L^\infty}}{|\alpha|\delta^{\frac{1}{2}}}
        +\n{\psi}_{L^{\infty}(B_{1,2,\delta})}\n{\chi^{c}w}_{L^{1}}},
\end{align*}
which yields that
\begin{align*}
\frac{\n{\chi^{c}w}^2_{L^{1}}}{y_{2}-y_{1}}
\lesssim&
 \frac{\mathrm{ln}(1+\frac{2y_{2}}{\delta})}{(y_{2}-y_{1})^{2}} \lrs{ \frac{|\lra{F,\frac{\rho_{\delta}^{c}w}{U''}}|}{|\alpha|}+\frac{\nu\|\partial_yw\|_{L^2}\n{w}_{L^\infty}}{|\alpha|\delta^{\frac{1}{2}}}
        +\frac{\mathrm{ln}(1+\frac{2y_{2}}{\delta})}{y_{2}-y_{1}}\n{\psi}_{L^{\infty}(B_{1,2,\delta})}^{2}}+
        \mathcal{E}(w)\\
        \lesssim &
        \frac{|\lra{F,\frac{\rho_{\delta}^{c}w}{U''}}|}{|\alpha|(y_{2}-y_{1})\delta}+\frac{\nu\n{\pa_{y}w}_{L^2}
        \n{w}_{L^\infty}}{|\alpha|(y_{2}-y_{1})\delta^{\frac{3}{2}}}
        +\frac{\n{\psi}_{L^{\infty}(B_{1,2,\delta})}^{2}}{(y_{2}-y_{1})^{2}\delta}+
        \mathcal{E}(w)\\
        \lesssim &\frac{|\lra{F,\frac{\rho_{\delta}^{c}w}{U''}}|}{|\alpha|(y_{2}-y_{1})\delta}
        +
        \mathcal{E}(w)+\frac{\n{\psi}_{L^{\infty}(B_{1,2,\delta})}^{2}}{(y_{2}-y_{1})^{2}\delta}.
\end{align*}
Hence, we complete the proof of \eqref{equ:estimate,w,L1,ex},
\end{proof}

\begin{lemma}[Local $L^{\infty}$ estimate on $\psi$]\label{Local-L-inf-est-psi}
Let $\delta\in (0, y_{2}]$. We have
\begin{align}\label{equ:estimate,vp,B,infty}
\frac{\n{\psi}_{L^{\infty}(B_{1,2,\delta})}^{2}}{(y_{2}-y_{1})^{2}\delta}\leq \frac{\min\lr{\n{F}_{L^{2}}^{2},\delta^{-2}\n{F}_{H^{-1}}^{2}}}{|\alpha|^2(y_{2}-y_{1})^2\delta^2}+
        \frac{\nu^{2}|\al|^{2}\n{w}_{L^{\infty}}^{2}}{(y_{2}-y_{1})^{2}\delta}+\mathcal{E}(w).
\end{align}
Moreover, for $F\in L^{2}$, we have a more stronger estimate
\begin{align}\label{equ:estimate,vp,B,infty,L2}
\frac{\n{\psi}_{L^{\infty}(B_{1,2,\delta})}^{2}}{(y_{2}-y_{1})^{2}\delta}\lesssim
\frac{\n{F}_{L^{2}}^{2}}{|\alpha|^2(y_{2}-y_{1})^2\delta^2}
        +\frac{\nu^{2}\n{w}_{L^{2}}^{2}}{|\alpha|^2(y_{2}-y_{1})^2\delta^2}+\mathcal{E}(w).
\end{align}
\end{lemma}
\begin{proof}

For any $a\in B(y_{1},\delta)$, let us introduce a smooth cut-off function $\rho_{a,\delta}(y)\in [0,1]$ with $0\le \rho_{a,\delta}(y)\leq 1$, $\rho_{a,\delta}(y+a)=\rho_{a,\delta}(a-y)$ and denoted by
\begin{equation}\label{def-rho3}
    \rho_{a,\delta}(y)=
    \left\{
    \begin{aligned}
          &1,\quad & y\in [a-\frac{\delta}{2}, a+\frac{\delta}{2}],\\
          &0,\quad &y\in[-1,1] \setminus (a-\delta,a+\delta),\\
          &\mathrm{smooth},\quad & \mathrm{others}.
    \end{aligned}\right.
\end{equation}
It is easy to see that
\begin{align}\label{equ:estimate,rho,a,delta}
   \n{\rho_{a,\delta}}_{L^{2}}\lesssim \delta^{\frac{1}{2}},\quad \|\rho_{a,\delta}\|_{L^\infty}\leq 1,\qquad \|\rho_{a,\delta}'\|_{L^1}\lesssim 1,\qquad \|\rho_{a,\delta}'\|_{L^2}\lesssim \delta^{-\f12},
\end{align}
We assume $g(a+y)=g(a-y)$, and set
\begin{align*}
G_{1}(y)=\int_{y}^{a+b}\int_{z_{1}}^{a+b}g(z_{2})dz_{2}dz_{1},\quad G_{2}(y)=\int_{a-b}^{y}\int_{a-b}^{z_{1}}g(z_{2})dz_{2}dz_{1}.
\end{align*}
By integration by parts, we have
\begin{align*}
\int_{a-b}^{a+b}f(y)g(y)dy =f(a)\int_{a-b}^{a+b}g(y)dy+\int_{a}^{a+b}f''(y)G_{1}(y)dy+\int_{a-b}^{a}f''(y)G_{2}(y)dy.
\end{align*}
By taking $b=\delta$, $g(y)=\rho_{a,\delta}(y)$, we obtain
\begin{align}\label{equ:estimate,vp,B,infty,basic inequality}
|f(a)|\lesssim \frac{1}{b}\bbabs{\int_{a-\delta}^{a+\delta}f(y)\rho_{a,\delta}(y)dy}+\delta^{2}\n{f''}_{L^{\infty}}.
\end{align}
Then, recalling that
$$-\nu(\partial^2_y-\alpha^2)w+i\alpha(U-\lambda)w-i\alpha U''\psi=F,$$ we take $f(y)=U''(y)\psi(y)$ in \eqref{equ:estimate,vp,B,infty,basic inequality} to obtain
\begin{align}\label{equ:estimate,vp,B,infty,single}
         |U''(a)\psi(a)|
         \leq&\frac{1}{\delta}\Bigg|\int_{a-\delta}^{a+\delta}U''\psi\rho_{a,\delta}(y)dy\Bigg|+
         \delta^2\|(U''\psi)''\|_{L^\infty}\notag\\
\leq&\frac{1}{|\al|\delta}\bbabs{\lra{F+\nu(\pa_{y}^{2}-\al^{2})w,\rho_{a,\delta}}}+
\frac{1}{\delta}\babs{\lra{(U-\la)w,\rho_{a,\delta}}}+\delta^2\n{w}_{L^\infty},
\end{align}
where in the last inequality we have used that $\delta^2\|(U''\psi)''\|_{L^\infty}\lesssim \delta^2\n{w}_{L^\infty}$.
Next, we deal with other terms on the right hand side of \eqref{equ:estimate,vp,B,infty,single}. By \eqref{equ:estimate,rho,a,delta} and energy estimate \eqref{equ:energy estimate,real part,w,L2}, we have
\begin{align*}
|\lra{F,\rho_{a,\delta}}|\lesssim& \delta^{\frac{1}{2}}\min\lr{\n{F}_{L^{2}},\delta^{-1}\n{F}_{H^{-1}}},\\
\nu|\lra{\pa_{y}^{2}w,\rho_{a,\delta}}|\lesssim& \nu\delta^{-\frac{1}{2}}\n{\pa_yw}_{L^{2}},\\
\nu|\al|^{2}|\lra{w,\rho_{a,\delta}}|\lesssim &\delta^{\frac{1}{2}}\nu|\al|^{2}\n{w}_{L^{2}}\lesssim \delta^{\frac{1}{2}}\nu \n{w}_{L^{2}}+\delta^{\frac{1}{2}}\n{F}_{L^{2}},\\
\nu|\al|^{2}|\lra{w,\rho_{a,\delta}}|\lesssim& \delta \nu|\al|^{2}\n{w}_{L^{\infty}},\\
|\lra{(U-\la)w,\rho_{a,\delta}}|\lesssim& (y_{2}-y_{1})\delta^{2}\n{w}_{L^{\infty}},
\end{align*}
where in the last inequality we have used that
 $$1_{B(a,\delta)}(y)|U(y)-\la|\lesssim (y_{2}-y_{1})\delta.$$

To sum up, using that $U''>0$ and $\delta\leq y_{2}$, we have
\begin{align*}
\frac{\n{\psi}_{L^{\infty}(B_{1,2,\delta})}^{2}}{(y_{2}-y_{1})^{2}\delta}\leq \frac{\min\lr{\n{F}_{L^{2}}^{2},\delta^{-2}\n{F}_{H^{-1}}^{2}}}{|\alpha|^2(y_{2}-y_{1})^2\delta^2}+
        \frac{\nu^{2}|\al|^{2}\n{w}_{L^{\infty}}^{2}}{(y_{2}-y_{1})^{2}\delta}+\mathcal{E}(w),
\end{align*}
and
\begin{align*}
\frac{\n{\psi}_{L^{\infty}(B_{1,2,\delta})}^{2}}{(y_{2}-y_{1})^{2}\delta}\lesssim
\frac{\n{F}_{L^{2}}^{2}}{|\alpha|^2(y_{2}-y_{1})^2\delta^2}
        +\frac{\nu^{2}\n{w}_{L^{2}}^{2}}{|\alpha|^2(y_{2}-y_{1})^2\delta^2}+\mathcal{E}(w).
\end{align*}
Hence, we complete the proof of \eqref{equ:estimate,vp,B,infty} and \eqref{equ:estimate,vp,B,infty,L2}.
\end{proof}

Now, we are able to provide the exterior $L^{2}$ estimate on $w$.
\begin{lemma}[Exterior $L^{2}$ estimate on $w$]\label{lemma:ex,l2,w}
Let $\delta\in (0, y_{2}]$.
We have
\begin{equation}\label{equ:ex,l2,w}
\n{w}_{L^{2}\lrs{(-1,1)\setminus(y_{1}, y_{2})}}^{2} \lesssim \mathcal{E}^{ex}(w),
\end{equation}
where
\begin{align}\label{equ:E,ex,full}
\mathcal{E}^{ex}(w)=&\frac{\operatorname{Im}|\lra{F,\frac{\rho_{\delta}^{c} w}{U-\la}}|}{|\al|}+ \frac{|\lra{F,\frac{\rho_{\delta}^{c}w}{U''}}|}{|\alpha|(y_{2}-y_{1})\delta}
+\frac{\min\lr{\n{F}_{L^{2}}^{2},\delta^{-2}\n{F}_{H^{-1}}^{2}}}{|\alpha|^2(y_{2}-y_{1})^2\delta^2}\\
&+\frac{\nu^{2}|\al|^{2}\n{w}_{L^{\infty}}^{2}}{(y_{2}-y_{1})^{2}\delta}+\frac{\nu^{2}\n{w}_{L^{2}}^{2}}{|\alpha|^2(y_{2}-y_{1})^2\delta^2}+\mathcal{E}(w).\notag
\end{align}

\end{lemma}
\begin{proof}
First, we have
\begin{align*}
\n{w}_{L^{2}\lrs{(-1,1)\setminus(y_{1}, y_{2})}}^{2}\leq& \int_{-1}^{1}\rho_{\delta}^{c}|w|^2dy+\delta\n{w}_{L^{\infty}}^{2}.
\end{align*}
Then by \eqref{equ:ex,estimate,w,L2} and \eqref{equ:ex,l2 weighted,psi}, we get
\begin{align*}
 \int_{-1}^{1}\rho_{\delta}^{c}|w|^{2}dy
\leq &\frac{|\operatorname{Im}\lra{F,\frac{\rho_{\delta}^{c} w}{U-\la}}|}{|\al|}+
        \frac{\nu\|\partial_y w\|_{L^2}\n{w}_{L^\infty}}{|\alpha|\delta^{\f32}(y_{2}-y_{1})} +\int_{-1}^{1}\frac{|\psi|^2\rho_{\delta}^{c}}{(U-\lambda)^2}dy\\
        \lesssim&\frac{|\operatorname{Im}\lra{F,\frac{\rho_{\delta}^{c} w}{U-\la}}|}{|\al|}+\mathcal{E}(w)+\int_{-1}^{1}\frac{|\psi|^2\rho_{\delta}^{c}}{(U-\lambda)^2}dy\\
\lesssim& \frac{|\operatorname{Im}\lra{F,\frac{\rho_{\delta}^{c} w}{U-\la}}|}{|\al|}+ \frac{|\lra{F,\frac{\rho_{\delta}^{c}w}{U''}}|}{|\alpha|(y_{2}-y_{1})\delta}+
\mathcal{E}(w)
        +\frac{\n{\psi}_{L^{\infty}(B_{1,2,\delta})}^{2}}{(y_{2}-y_{1})^{2}\delta}+\frac{\n{\chi^{c}w}_{L^{1}}^{2}}{y_{2}-y_{1}}.
\end{align*}
Using estimates \eqref{equ:estimate,vp,B,infty} on $\n{\psi}_{L^{\infty}(B_{1,2,\delta})}$ and \eqref{equ:estimate,w,L1,ex} on $\n{\chi^{c}w}_{L^{1}}$, we arrive at
\begin{align*}
 \int_{-1}^{1}\rho_{\delta}^{c}|w|^{2}dy\lesssim& \frac{|\operatorname{Im}\lra{F,\frac{\rho_{\delta}^{c} w}{U-\la}}|}{|\al|}+ \frac{|\lra{F,\frac{\rho_{\delta}^{c}w}{U''}}|}{|\alpha|(y_{2}-y_{1})\delta}
+\frac{\min\lr{\n{F}_{L^{2}}^{2},\delta^{-2}\n{F}_{H^{-1}}^{2}}}{|\alpha|^2(y_{2}-y_{1})^2\delta^2}\\
&+\frac{\nu^{2}|\al|^{2}\n{w}_{L^{\infty}}^{2}}{(y_{2}-y_{1})^{2}\delta}+\mathcal{E}(w).
\end{align*}
Hence, the proof of \eqref{equ:ex,l2,w} is done.
\end{proof}

\subsubsection{Interior energy estimates}

In this subsection, we focus on estimates on the interior domain $(y_{1},y_{2})$ with $\la\in [U(0),U(1)]$. First, let us introduce the smooth cutoff function
\begin{equation}\label{equ:smooth cutoff,in,delta}
    \rho_{\delta}(y)=
    \left\{
    \begin{aligned}
          &1,\quad & y\in(y_{1}-\delta, y_{2}-\delta),\\
          &0,\quad &y\in[-1,1] \setminus (y_{1}-\frac{\delta}{2},y_{2}-\frac{\delta}{2}),\\
          &\mathrm{smooth},\quad & \mathrm{others},
    \end{aligned}\right.
\end{equation}
where $0<\delta\leq y_{2}\leq 1$ and $0\leq \rho_{\delta}(y)\leq 1. $

The following lemma provides interior energy estimates.
\begin{lemma}[Interior energy estimates]\label{lemma:in,estimate,w,psi}
Let $\delta\in (0, y_{2}]$ and $(w,\psi)$ be the solution to \eqref{equ:OS,navier-slip}.
We have
\begin{align}\label{equ:in,estimate,w,L2}
        \int_{-1}^{1}\rho_{\delta}|w|^2dy\leq &\frac{|\operatorname{Im}\lra{F,\frac{\rho_{\delta} w}{U-\la}}|}{|\al|}+
        \frac{\nu\|\partial_y w\|_{L^2}\n{w}
        _{L^\infty}}{|\alpha|\delta^{\f32}(y_{2}-y_{1})} +\int_{-1}^{1}\frac{|\psi|^2\rho_{\delta}}{(U-\lambda)^2}dy
\end{align}
and
\begin{align}\label{equ:in,estimate,w}
        &\int_{y_1}^{y_2}(\la-U)^2\left|\left(\frac{\psi_{1}}{\la-U}\right)^{\prime}\right|^2 d y+|\al|^{2}\int_{y_1}^{y_2} |\psi_{1}|^2 d y+(y_{2}-y_{1})^{2}\int_{y_{1}}^{y_{2}}|\pa_y\psi_{1}|^{2}dy\\
        \lesssim & \frac{|\lra{F,\frac{\rho_{\delta}w}{U''}}|}{|\alpha|}+\frac{\nu\|\partial_yw\|_{L^2}\n{w}_{L^\infty}}{|\alpha|\delta^{\frac{1}{2}}}
        +(y_{2}-y_{1})\n{\psi}_{L^{\infty}(B_{1,2,\delta})}\n{w}_{L^{\infty}}+(y_{2}-y_{1})\delta^{2}\n{w}_{L^{\infty}}^{2}.\notag
\end{align}
In particular, we have
\begin{align}\label{equ:in,estimate,w,psi}
&\frac{1}{(y_{2}-y_{1})^{2}}\int_{y_1}^{y_2}(\la-U)^2\left|\left(\frac{\psi_{1}}{\la-U}\right)^{\prime}\right|^2 d y+\frac{|\al|^{2}}{(y_{2}-y_{1})^{2}}\int_{y_1}^{y_2} |\psi_{1}|^2 d y\\
\lesssim& \frac{|\lra{F,\frac{\rho_{\delta}w}{U''}}|}{|\alpha|(y_{2}-y_{1})^{2}}+\mathcal{E}(w)
+\frac{\|\psi\|_{L^\infty(B_{1,2,\delta})}^{2}}{(y_{2}-y_{1})^2\delta},\notag
\end{align}
where $\mathcal{E}(w)$ is defined in \eqref{equ:E,w,effective}.
\end{lemma}
\begin{proof}
The proof of the interior energy estimate \eqref{equ:in,estimate,w,L2} follows the same procedure as that of the exterior energy estimate \eqref{equ:ex,estimate,w,L2}, with the smooth cutoff function $\rho_{\delta}^{c}$ replaced by $\rho_{\delta}$. Similarly, following the proof of exterior energy estimate \eqref{equ:ex,estimate,w}, we also have
\begin{align}\label{equ:in,estimate,w,psi,proof}
        &\bblra{\frac{(\la-U)w}{U''},\chi w}+\lra{\psi_{1},\chi w} \\
        \leq   &\bblra{\frac{(\la-U)w}{U''},\rho_{\delta} w}+\lra{\psi_{1},\chi w} +(y_{2}-y_{1})\delta^{2}\n{w}_{L^{\infty}}^{2}\notag\\
        \lesssim & \frac{|\lra{F,\frac{\rho_{\delta}w}{U''}}|}{|\alpha|}+\frac{\nu\|\partial_yw\|_{L^2}\n{w}_{L^\infty}}{|\alpha|\delta^{\frac{1}{2}}}
        +(y_{2}-y_{1})\n{\psi}_{L^{\infty}(B_{1,2,\delta})}\n{\rho_{\delta}w}_{L^{\infty}}
        +(y_{2}-y_{1})\delta^{2}\n{w}_{L^{\infty}}^{2},\notag
\end{align}
which together with interior coercive estimates \eqref{equ:in,coercive,L2}--\eqref{equ:in,coercive,H1} implies the desired estimate \eqref{equ:in,estimate,w}.
Applying Cauchy-Schwarz inequality to \eqref{equ:in,estimate,w} and using that $\delta\leq y_{2}$, we obtain \eqref{equ:in,estimate,w,psi}.
\end{proof}

From \eqref{equ:in,estimate,w,L2}, to provide interior $L^{2}$ estimate on $w$, we need to estimate
 $\int_{-1}^{1}\frac{|\psi|^2\rho_{\delta}}{(U-\lambda)^2}dy$.

\begin{lemma}[Interior $L^{2}$ weighted estimate on $\psi$]\label{lemma:in,l2 weighted,psi}
Let $\delta\leq 4\theta \leq y_{2}$.
We have
\begin{align}\label{equ:in,l2 weighted,psi}
\int_{-1}^{1}\frac{|\psi|^2\rho_{\delta}}{(U-\lambda)^2}dy\lesssim \frac{|\lra{F,\frac{\rho_{\delta}w}{U''}}|}{|\alpha|(y_{2}-y_{1})^{2}}+\mathcal{E}(w)+ \frac{\|\psi\|_{L^\infty(B_{1,2,\delta})}^{2}}{(y_{2}-y_{1})^2\delta}+\frac{1}{y_{2}-y_{1}}\bbabs{\int_{y_{1}+\theta}^{y_{2}-\theta}\frac{U''\psi_{1}}{\la-U}dy}^{2}.
\end{align}
\end{lemma}
\begin{proof}
Due to $\psi=\psi_{1}+\psi_{2}$, we have
\begin{align*}
\int_{-1}^{1}\frac{|\psi|^2\rho_{\delta}}{(U-\lambda)^2} dy\lesssim & \int_{-1}^{1}\frac{|\psi_{1}|^2\rho_{\delta}}{(U-\lambda)^2} dy+
\int_{-1}^{1}\frac{|\psi_{2}|^2\rho_{\delta}}{(U-\lambda)^2} dy\\
\lesssim&\int_{y_{1}}^{y_{2}}\frac{|\psi_{1}|^2}{(U-\lambda)^2} dy+\frac{\|\psi\|_{L^\infty(B_{1,2,\delta})}^{2}}{(y_{2}-y_{1})^2\delta},
\end{align*}
where in the last inequality we have used the local $L^{\infty}$ estimate \eqref{equ:local Linfity estimate} on $\psi_{2}$ and \eqref{equ:basic estimate,U,L2} in Lemma \ref{lemma:basic Lp estimate,U}.

By Hardy-type inequality \eqref{equ:hardy,est,psi1,U} in Lemma \ref{lemma:hardy-type}, single point estimate \eqref{equ:single point estimate,in} in Lemma \ref{lemma:single point estimate,in}, and coercive estimate \eqref{equ:in,coercive,L2}, we derive
\begin{align*}
\int_{y_{1}}^{y_{2}}\frac{|\psi_{1}|^{2}}{|U-\la|^{2}}dy
\lesssim& \frac{1}{\lrs{y_{2}-y_{1}}^{2}}\int_{y_{1}}^{y_{2}}
\bbabs{\lrs{\frac{\psi_{1}}{\la-U}}^{'}}^{2}\lrs{\la-U}^{2}dy
+\frac{|\psi_{1}(0)|^{2}}{(y_{2}-y_{1})^{3}}\\
\lesssim& \frac{1}{(y_{2}-y_{1})^{2}}\lrs{\bblra{\frac{\big(\la-U\big)w}{U''},\chi w}+\lra{\psi_{1},\chi w}}
+\frac{1}{y_{2}-y_{1}}\bbabs{\int_{y_{1}+\theta}^{y_{2}-\theta}\frac{U''\psi_{1}}{\la-U}dy}^{2}.
\end{align*}
which, together with interior energy estimates \eqref{equ:in,estimate,w}, \eqref{equ:in,estimate,w,psi}, and \eqref{equ:in,estimate,w,psi,proof}, gives the desired estimate \eqref{equ:in,l2 weighted,psi}.
\end{proof}

We are left to deal with the weighted integration term $\int_{y_{1}+\theta}^{y_{2}-\theta}\frac{U''\psi_{1}}{\la-U}dy$.
\begin{lemma}[Interior weighted integration estimate on $\psi_{1}$]\label{lemma:in,weighted integration,psi1}
Let $\delta\leq 4\theta\leq y_{2}$. It holds that
\begin{align}\label{equ:in,weighted integration,psi1}
\frac{1}{y_{2}-y_{1}}\bbabs{\int_{y_{1}+\theta}^{y_{2}-\theta}\frac{U''\psi_{1}}{\la-U}dy}^{2}\lesssim \frac{|\lra{F,\frac{\rho_{\delta}w}{U''}}|}{|\alpha|(y_{2}-y_{1})\delta}+\mathcal{E}^{ex}(w),
\end{align}
where $\mathcal{E}^{ex}(w)$ is given by \eqref{equ:E,ex,full}.
\end{lemma}
\begin{proof}
We divide the proof into the following two cases.\smallskip

{\it Case 1.} $|\alpha|^2(y_{2}-y_{1})^2\geq 1.$

In this case, taking $\theta=\frac{y_{2}}{2}$, we use \eqref{equ:basic estimate,U,L2} and interior energy estimate \eqref{equ:in,estimate,w,psi} to get
\begin{align*}
         \frac{1}{y_{2}-y_{1}}\bbabs{\int_{y_{1}+\theta}^{y_{2}-\theta}\frac{U''\psi_{1}}{\la-U}dy}^{2}
         \leq & \frac{1}{y_{2}-y_{1}}\lrs{\int_{y_{1}+\theta}^{y_{2}-\theta}|\psi_{1}|^{2}dy}
         \lrs{\int_{y_{1}+\theta}^{y_{2}-\theta}\frac{1}{(\lambda-U)^2}dy}\\
         \lesssim& \frac{|\al|^{2}}{(y_{2}-y_{1})^{2}}\int_{y_{1}}^{y_{2}}|\psi_{1}|^{2}dy\\
         \lesssim& \frac{|\lra{F,\frac{\rho_{\delta}w}{U''}}|}{|\alpha|(y_{2}-y_{1})^{2}}+\mathcal{E}(w)
+\frac{\|\psi\|_{L^\infty(B_{1,2,\delta})}^{2}}{(y_{2}-y_{1})^2\delta}.
\end{align*}
Hence, together with local $L^{\infty}$ estimate \eqref{equ:estimate,vp,B,infty} on $\psi$, we arrive at \eqref{equ:in,weighted integration,psi1}.\smallskip

{\it Case 2.} $|\alpha|^2(y_{2}-y_{1})^2\leq 1.$

In this case, taking $\theta=\delta$,  we use \eqref{equ:basic estimate,U,L2} to obtain
\begin{align*}
\bbabs{\int_{y_{1}+\theta}^{y_{2}-\theta}\frac{U''(y)\psi_{1}(y)}{\la-U(y)}dy}\leq& \bbabs{\int_{-1}^{1}\frac{U''\psi_{1}\rho_{\delta}}{\lambda-U}dy}+\delta^{\frac{1}{2}}\n{\psi_{1}}_{L^{\infty}(B_{1,2,\delta})}\bbn{\frac{1}{U-\lambda}}_{L^{2}((-1,1)\setminus B_{1,2,\delta})}\\
\lesssim&\bbabs{\int_{-1}^{1}\frac{U''\psi_{1}\rho_{\delta}}{\lambda-U}dy}+\frac{\n{\psi_{1}}_{L^{\infty}(B_{1,2,\delta})}}{y_{2}-y_{1}}.
\end{align*}
Recalling that
$$
F=-\nu(\partial_y^2-|\alpha|^2)w+i\alpha(U-\lambda)w-i\alpha U''\psi,
$$
we have
\begin{align*}
\bbabs{\int_{-1}^{1}\frac{U''\psi_{1}\rho_{\delta}}{\lambda-U}dy}\leq&
\bbabs{\int_{-1}^{1}\frac{U''\psi_{2}\rho_{\delta}}{\lambda-U}dy}+\frac{1}{|\al|} \bbabs{\int_{-1}^{1} \frac{F\rho_{\delta}+\nu(\pa_{y}^{2}w-\al^{2}w)\rho_{\delta}}{\la-U}dy}+\abs{\lra{w,\rho_{\delta}}},
\end{align*}
and hence get
\begin{align*}
\frac{1}{y_{2}-y_{1}}\bbabs{\int_{-1}^{1}\frac{U''\psi_{1}\rho_{\delta}}{\lambda-U}dy}^{2}\leq I_{1}+I_{2}+I_{3},
\end{align*}
where
\begin{align*}
I_{1}=&\frac{1}{y_{2}-y_{1}}\bbabs{\int_{-1}^{1}\frac{U''\psi_{2}\rho_{\delta}}{\lambda-U}dy}^{2},\\
I_{2}=&\frac{1}{|\al|^{2}(y_{2}-y_{1})}\bbabs{\int_{-1}^{1} \frac{F\rho_{\delta}+\nu(\pa_{y}^{2}w-\al^{2}w)\rho_{\delta}}{\la-U}dy}^{2},\\
I_{3}=&\frac{1}{y_{2}-y_{1}}\bbabs{\int_{-1}^{1} w\rho_{\delta}dy }^{2}.
\end{align*}

For $I_{1}$, by the local $L^{\infty}$ estimate \eqref{equ:local Linfity estimate} and estimate \eqref{equ:basic estimate,U,L2} in Lemma \ref{lemma:basic Lp estimate,U}, we get
\begin{align}
I_{1}\lesssim \bbabs{\int_{-1}^{1}\frac{|\psi_{2}|^2}{(U-\lambda)^2}dy}\lesssim \frac{\n{\psi}_{L^{\infty}(B_{1,2,\delta})}^{2}}{(y_{2}-y_{1})^{2}\delta}\lesssim \mathcal{E}^{ex}(w).
\end{align}

For $I_{2}$, by Lemma \ref{lemma:basic Lp estimate,U}, we have
\begin{align*}
\babs{\blra{F,\frac{\rho_{\delta}}{\la-U}}}\lesssim& \n{F}_{L^{2}}\bbn{\frac{\rho_{\delta}}{\la-U}}_{L^{2}}\lesssim \frac{\n{F}_{L^{2}}}{(y_{2}-y_{1})\delta^{\frac{1}{2}}},\\
\babs{\blra{F,\frac{\rho_{\delta}}{\la-U}}}\lesssim& \n{F}_{H^{-1}}\bbn{\lrs{\frac{\rho_{\delta}}{\la-U}}'}_{L^{2}}\lesssim \frac{\n{F}_{H^{-1}}}{(y_{2}-y_{1})\delta^{\frac{3}{2}}},\\
\nu\babs{\blra{\pa_{y}^{2}w,\frac{\rho_{\delta}}{\la-U}}}\lesssim & \nu\n{\pa_yw}_{L^{2}}\bbn{\lrs{\frac{\rho_{\delta}}{\la-U}}'}_{L^{2}}\lesssim
\frac{\nu\n{\pa_yw}_{L^{2}}}{(y_{2}-y_{1})\delta^{\frac{3}{2}}},\\
\nu\babs{\blra{w,\frac{\rho_{\delta}}{\la-U}}}\lesssim & \nu\n{w}_{L^{2}}\bbn{\frac{\rho_{\delta}}{\la-U}}_{L^{2}}\lesssim
\frac{\nu\n{w}_{L^{2}}}{(y_{2}-y_{1})\delta^{\frac{1}{2}}}.
\end{align*}
Using that $\delta\leq y_{2}$, we obtain
\begin{align*}
I_{2}\lesssim \mathcal{E}^{ex}(w).
\end{align*}

For $I_{3}$, we set $\eta(y)$ to be a smooth cutoff function satisfying $\eta(y)=1$ for $|y|\leq 1$. Then, we have
\begin{align}\label{equ:in,weighted integration,psi1,I3}
I_{3}\lesssim& \frac{1}{y_{2}-y_{1}}\lrs{ \delta\n{w}_{L^{\infty}}+\n{w}_{L^{1}((-1,1)\setminus (y_{1},y_{2}))}+
\bbabs{\blra{w,\eta(\frac{y}{y_{2}-y_{1}})}}}^{2}\\
\lesssim &\delta\n{w}_{L^{\infty}}^{2}+\frac{1}{y_{2}-y_{1}}\n{w}_{L^{1}((-1,1)\setminus (y_{1},y_{2}))}^{2}+\frac{1}{y_{2}-y_{1}}
\bbabs{\blra{w,\eta(\frac{y}{y_{2}-y_{1}})}}^{2}\notag\\
\lesssim& \mathcal{E}^{ex}(w)+\frac{1}{y_{2}-y_{1}}
\bbabs{\blra{w,\eta(\frac{y}{y_{2}-y_{1}})}}^{2},\notag
\end{align}
where in the last inequality we have used Exterior $L^{1}$ estimate \eqref{equ:estimate,w,L1,ex} on $w$. By integration by parts, we get
\begin{align*}
&\frac{1}{y_{2}-y_{1}}\babs{\blra{w,\eta(\frac{y}{y_{2}-y_{1}})}}^{2}
=\frac{1}{y_{2}-y_{1}}\babs{\blra{\pa_{y}^{2}\psi_{1}-\al^{2}\psi_{1},\eta(\frac{y}{y_{2}-y_{1}})}}^{2}\\
\lesssim& \frac{1}{y_{2}-y_{1}} \n{\psi_{1}'}_{L^{2}((-1,1)\setminus(y_{1},y_{2}))}^{2}\bbn{\lrs{\eta(\frac{y}{y_{2}-y_{1}})}'}_{L^{2}}^{2}+|\al|^{4}\n{\psi_{1}}_{L^{2}}^{2}
\bbn{\eta(\frac{y}{y_{2}-y_{1}})}_{L^{2}}^{2}\\
\lesssim & \frac{\n{\psi_{1}'}_{L^{2}((-1,1)\setminus(y_{1},y_{2}))}^{2}}{(y_{2}-y_{1})^{2}}+|\al|^{4}\n{\psi_{1}}_{L^{2}}^{2}.
\end{align*}
Due to that $|\al|^{2}\lesssim \frac{1}{(y_{2}-y_{1})^{2}}$, we use exterior and interior energy estimates \eqref{equ:ex,estimate,w,psi} and \eqref{equ:in,estimate,w,psi} to get
\begin{align*}
\frac{1}{y_{2}-y_{1}}\babs{\blra{w,\eta(\frac{y}{y_{2}-y_{1}})}}^{2}\lesssim \mathcal{E}^{ex}(w)+\frac{|\lra{F,\frac{\rho_{\delta}w}{U''}}|}{|\alpha|(y_{2}-y_{1})^{2}},
\end{align*}
which together with \eqref{equ:in,weighted integration,psi1,I3} gives that
\begin{align*}
I_{3}\lesssim \mathcal{E}^{ex}(w)+\frac{|\lra{F,\frac{\rho_{\delta}w}{U''}}|}{|\alpha|(y_{2}-y_{1})^{2}}.
\end{align*}

Putting together the estimates for $I_{1}$, $I_{2}$, $I_{3}$, we complete the proof of \eqref{equ:in,weighted integration,psi1}.
\end{proof}

To sum up, we arrive at the following interior $L^{2}$ estimate.
\begin{lemma}[Interior $L^{2}$ estimate on $w$]\label{lemma:in,l2,w}
It holds that
\begin{equation}\label{equ:in,l2,w}
\n{w}_{L^{2}(y_{1}, y_{2})}^2 \lesssim  \mathcal{E}^{in}(w)+\mathcal{E}^{ex}(w),
\end{equation}
where
\begin{align*}
\mathcal{E}^{in}(w)=\frac{|\operatorname{Im}\lra{F,\frac{\rho_{\delta} w}{U-\la}}|}{|\al|}+ \frac{|\lra{F,\frac{\rho_{\delta}w}{U''}}|}{|\alpha|(y_{2}-y_{1})\delta}
+\mathcal{E}(w).
\end{align*}
\end{lemma}
\begin{proof}
By interior energy estimate \eqref{equ:in,estimate,w,L2}, we have
\begin{align*}
\n{w}_{L^{2}(y_{1}, y_{2})}^{2} \leq& \int_{-1}^{1}\rho_{\delta}|w|^2dy+\delta\n{w}_{L^{\infty}}^{2}\\
\leq &\frac{|\operatorname{Im}\lra{F,\frac{\rho_{\delta} w}{U-\la}}|}{|\al|}+
        \frac{\nu\|\partial_y w\|_{L^2}\n{w}_{L^\infty}}{|\alpha|\delta^{\f32}(y_{2}-y_{1})} +\int_{-1}^{1}\frac{|\psi|^2\rho_{\delta}}{(U-\lambda)^2}dy+\delta\n{w}_{L^{\infty}}^{2}\\
        \lesssim&\frac{|\operatorname{Im}\lra{F,\frac{\rho_{\delta} w}{U-\la}}|}{|\al|}+\mathcal{E}(w)+\int_{-1}^{1}\frac{|\psi|^2\rho_{\delta}}{(U-\lambda)^2}dy.
\end{align*}
Then by Lemmas \ref{lemma:in,l2 weighted,psi}--\ref{lemma:in,weighted integration,psi1}, we get
\begin{align*}
\n{w}_{L^{2}(y_{1}, y_{2})}^{2}\lesssim& \frac{|\operatorname{Im}\lra{F,\frac{\rho_{\delta} w}{U-\la}}|}{|\al|}+\frac{|\lra{F,\frac{\rho_{\delta}w}{U''}}|}{|\alpha|(y_{2}-y_{1})\delta}+\mathcal{E}^{ex}(w)\\
\lesssim& \mathcal{E}^{in}(w)+\mathcal{E}^{ex}(w),
\end{align*}
which completes the proof of \eqref{equ:in,l2,w}.
\end{proof}

\subsection{Resolvent estimates for $F\in L^{2}$ and $F\in H^{-1}$}
In the subsection, we prove resolvent estimates for $F\in L^{2}$ and $F\in H^{-1}$.
\begin{lemma}\label{lemma:resol,w,L2,H-1,delta}
Let $\lambda \in (U(0)+\nu^{\frac{1}{2}}|\al|^{-\frac{1}{2}},U(1))$, and set
\begin{align}
&\delta=\nu^{\frac{1}{3}}|\al|^{-\frac{1}{3}}(y_{2}-y_{1})^{-\frac{1}{3}},
\quad \nu^{\frac{1}{4}}|\al|^{-\frac{1}{4}}\leq y_{2}\leq 1,\\
&\nu|\al|^{2}\lesssim |\la-U(0)|^{\frac{1}{2}}+\nu^{\frac{1}{4}}|\al|^{-\frac{1}{4}}.\label{equ:restriction,v,alpha}
\end{align}
Then we have
\begin{align}
\frac{\nu}{\delta^{2}}\n{w}_{L^{2}}+\frac{\nu(y_{2}-y_{1})^{\frac{1}{2}}}{\delta^{2}}\n{w}_{L^{1}}
+\frac{\nu}{\delta}\n{(\pa_{y},|\al|)w}_{L^{2}}\lesssim& \n{F}_{L^{2}},\label{equ:re,estimate,L2,lemma}\\
\frac{\nu}{\delta}\n{w}_{L^{2}}+\nu\n{(\pa_{y},|\al|)w}_{L^{2}}\lesssim& \n{F}_{H^{-1}}.\label{equ:re,estimate,H-1,lemma}
\end{align}
\end{lemma}
\begin{proof}
By $L^{2}$ exterior and interior estimates on $w$ in Lemmas \ref{lemma:ex,l2,w} and \ref{lemma:in,l2,w}, we have
\begin{align*}
\n{w}_{L^{2}}^{2}\lesssim \mathcal{E}^{in}(w)+\mathcal{E}^{ex}(w),
\end{align*}
where
\begin{align*}
\mathcal{E}^{ex}(w)= &\frac{\operatorname{Im}\lra{F,\frac{\rho_{\delta}^{c} w}{U-\la}}|}{|\al|}+ \frac{|\lra{F,\frac{\rho_{\delta}^{c}w}{U''}}|}{|\alpha|(y_{2}-y_{1})\delta}
+\frac{\min\lr{\n{F}_{L^{2}}^{2},\delta^{-2}\n{F}_{H^{-1}}^{2}}}{|\alpha|^2(y_{2}-y_{1})^2\delta^2}\\
&+\frac{\nu^{2}|\al|^{2}\n{w}_{L^{\infty}}^{2}}{(y_{2}-y_{1})^{2}\delta}+
\frac{\nu^{2}\n{w}_{L^{2}}^{2}}{|\alpha|^2(y_{2}-y_{1})^2\delta^2}+\mathcal{E}(w),\\
\mathcal{E}^{in}(w)=&\frac{|\operatorname{Im}\lra{F,\frac{\rho_{\delta} w}{U-\la}}|}{|\al|}+ \frac{|\lra{F,\frac{\rho_{\delta}w}{U''}}|}{|\alpha|(y_{2}-y_{1})\delta}
+\mathcal{E}(w),\\
\mathcal{E}(w)=&\frac{\nu^{2}\n{\pa_yw}_{L^{2}}^{2}}{|\al|^{2}(y_{2}-y_{1})^{2}\delta^{4}}
        +\delta\n{w}_{L^{\infty}}^{2}.
\end{align*}

Now, we need to estimate $\mathcal{E}^{in}(w)+\mathcal{E}^{ex}(w)$. First,
 with $\nu=\delta^{3}|\al|(y_{2}-y_{1})$, we have
\begin{align*}
\frac{\nu^{2}\n{w}_{L^{2}}^{2}}{|\alpha|^2(y_{2}-y_{1})^2\delta^2}
=\frac{\delta^{6}|\al|^{2}(y_{2}-y_{1})^{2}\n{w}_{L^{2}}^{2}}{|\alpha|^{2}(y_{2}-y_{1})^2\delta^2}=\delta^{4}\n{w}_{L^{2}}^{2}\ll \n{w}_{L^{2}}^{2}.
\end{align*}
Due to the restriction \eqref{equ:restriction,v,alpha} which gives that $|\al|\delta\lesssim 1$, with $\nu=\delta^{3}|\al|(y_{2}-y_{1})$, we get
\begin{align*}
\frac{\nu^{2}|\al|^{2}\n{w}_{L^{\infty}}^{2}}{(y_{2}-y_{1})^{2}\delta}\lesssim \delta\n{w}_{L^{\infty}}^{2}.
\end{align*}
We are left to control the other terms in $\mathcal{E}^{ex}(w)$ and $\mathcal{E}^{in}(w)$.
Noting that
\begin{align*}
\bbabs{\operatorname{Im}\lra{F,\frac{\rho_{\delta}^{c} w}{U-\la}}}\leq & \n{F}_{H^{-1}}\n{(\pa_{y},|\al|)w}_{L^{2}}\bbn{\frac{\rho_{\delta}^{c}}{U-\la}}_{L^{\infty}}+\n{F}_{H^{-1}}\n{w}_{L^{\infty}}
\bbn{\lrs{\frac{\rho_{\delta}^{c}}{U-\la}}'}_{L^{2}}\\
\bbabs{\operatorname{Im}\lra{F,\frac{\rho_{\delta}^{c} w}{U-\la}}}\leq& \n{F}_{L^{2}}\n{w}_{L^{2}}\bbn{\frac{\rho_{\delta}^{c}}{U-\la}}_{L^{\infty}},\\
\bbabs{\lra{F,\frac{\rho_{\delta}^{c}w}{U''}}}\leq&\n{F}_{H^{-1}}\n{(\pa_{y},|\al|)w}_{L^{2}}\bbn{\frac{\rho_{\delta}^{c}}{U''}}_{L^{\infty}}
+\n{F}_{H^{-1}}\n{w}_{L^{\infty}}
\bbn{\lrs{\frac{\rho_{\delta}^{c}}{U''}}'}_{L^{2}},\\
\bbabs{\lra{F,\frac{\rho_{\delta}^{c}w}{U''}}}\leq&\n{F}_{L^{2}}\n{w}_{L^{2}}\bbn{\frac{\rho_{\delta}^{c}}{U''}}_{L^{\infty}},
\end{align*}
by estimates in Lemma \ref{lemma:basic Lp estimate,U}, we have
\begin{align*}
\frac{|\operatorname{Im}\lra{F,\frac{\rho_{\delta}^{c} w}{U-\la}}|}{|\al|}\lesssim&
\min\lr{\frac{\n{F}_{H^{-1}}\n{(\pa_{y},|\al|)w}_{L^{2}}}{|\alpha| (y_{2}-y_{1})\delta}
        +\frac{\n{F}_{H^{-1}}\n{w}_{L^\infty}}{|\alpha| (y_{2}-y_{1})\delta^{\f32}},\frac{\n{F}_{L^{2}}\n{w}_{L^{2}}}{|\al|(y_{2}-y_{1})\delta}},\\
        \frac{|\lra{F,\frac{\rho_{\delta}^{c}w}{U''}}|}{|\alpha|(y_{2}-y_{1})\delta}\lesssim&
\min\lr{\frac{\n{F}_{H^{-1}}\n{(\pa_{y},|\al|)w}_{L^{2}}}{|\alpha| (y_{2}-y_{1})\delta}
        +\frac{\n{F}_{H^{-1}}\n{w}_{L^\infty}}{|\alpha| (y_{2}-y_{1})\delta^{\f32}},\frac{\n{F}_{L^{2}}\n{w}_{L^{2}}}{|\al|(y_{2}-y_{1})\delta}}.
\end{align*}
In the same way, we also have
\begin{align*}
\frac{|\operatorname{Im}\lra{F,\frac{\rho_{\delta} w}{U-\la}}|}{|\al|}\lesssim&
\min\lr{\frac{\n{F}_{H^{-1}}\n{(\pa_{y},|\al|)w}_{L^{2}}}{|\alpha| (y_{2}-y_{1})\delta}
        +\frac{\n{F}_{H^{-1}}\n{w}_{L^\infty}}{|\alpha| (y_{2}-y_{1})\delta^{\f32}},\frac{\n{F}_{L^{2}}\n{w}_{L^{2}}}{|\al|(y_{2}-y_{1})\delta}},\\
\frac{|\lra{F,\frac{\rho_{\delta}w}{U''}}|}{|\alpha|(y_{2}-y_{1})\delta}\lesssim&
\min\lr{\frac{\n{F}_{H^{-1}}\n{(\pa_{y},|\al|)w}_{L^{2}}}{|\alpha| (y_{2}-y_{1})\delta}
        +\frac{\n{F}_{H^{-1}}\n{w}_{L^\infty}}{|\alpha| (y_{2}-y_{1})\delta^{\f32}},\frac{\n{F}_{L^{2}}\n{w}_{L^{2}}}{|\al|(y_{2}-y_{1})\delta}}.
\end{align*}
Therefore, by Cauchy-Schwarz inequality, we arrive at
\begin{align}
\mathcal{E}^{in}(w)+\mathcal{E}^{ex}(w)\lesssim \frac{\n{F}_{L^{2}}\n{w}_{L^{2}}}{|\al|(y_{2}-y_{1})\delta}
+\frac{\n{F}_{L^{2}}^{2}}{|\alpha|^2(y_{2}-y_{1})^2\delta^2}+\mathcal{E}(w),\label{equ:L2,in,ex}
\end{align}
\begin{align}
\mathcal{E}^{in}(w)+\mathcal{E}^{ex}(w)\lesssim \frac{\n{F}_{H^{-1}}\n{(\pa_{y},|\al|)w}_{L^{2}}}{|\alpha| (y_{2}-y_{1})\delta}+\frac{\delta^{-2}\n{F}_{H^{-1}}^{2}}{|\alpha|^2(y_{2}-y_{1})^2\delta^2}+\mathcal{E}(w).\label{equ:H-1,in,ex}
\end{align}

 Next, we estimate $\mathcal{E}(w)$.
By energy estimate \eqref{equ:energy estimate,real part,w,nu} and Sobolev inequality, we have
\begin{align*}
\nu\n{\pa_yw}_{L^{2}}^{2}\lesssim& \nu\n{w}^2_{L^2}+|\operatorname{Re}\lra{F,w}|,\\
\n{w}_{L^{\infty}}^{2}\lesssim& \n{\pa_yw}_{L^{2}}\n{w}_{L^{2}}+\n{w}_{L^{2}}^{2},
\end{align*}
and hence get
\begin{align}\label{equ:estimate,Ew}
\mathcal{E}(w)\lesssim  \frac{\nu|\operatorname{Re}\lra{F,w}|}{|\al|^{2}(y_{2}-y_{1})^{2}\delta^{4}}+\frac{\delta|\operatorname{Re}\lra{F,w}|^{\frac{1}{2}}\n{w}_{L^{2}}}{\nu^{\frac{1}{2}}}+\delta \n{w}_{L^{2}}^{2}.
\end{align}

For $F\in L^{2}$, combining \eqref{equ:L2,in,ex} and \eqref{equ:estimate,Ew}, we obtain
\begin{align*}
\n{w}_{L^{2}}^{2}\lesssim \frac{\n{F}_{L^{2}}^{2}}{|\alpha|^2(y_{2}-y_{1})^2\delta^2}+\frac{\nu^{2}\n{F}_{L^{2}}^{2}}{|\alpha|^{4}(y_{2}-y_{1})^{4}\delta^{8}}
+\frac{\delta^{4}\n{F}_{L^{2}}^{2}}{\nu^{2}}.
\end{align*}
By the optimal choice that $\delta=\nu^{\frac{1}{3}}|\al|^{-\frac{1}{3}}(y_{2}-y_{1})^{-\frac{1}{3}}$, we arrive at
\begin{align}\label{equ:L2,w,F}
\n{w}_{L^{2}}\lesssim& \frac{\delta^{2}}{\nu}\n{F}_{L^{2}}.
\end{align}

For $F\in H^{-1}$, by energy estimate \eqref{equ:energy estimate,real part,w,H-1}, we have
\begin{align}\label{equ:energy estimate,H1,used}
 \n{(\pa_{y},|\al|)w}_{L^{2}}\lesssim& \n{w}_{L^2}+\nu^{-1}\n{F}_{H^{-1}}.
\end{align}
Combining \eqref{equ:L2,in,ex} and \eqref{equ:estimate,Ew}, with $\delta=\nu^{\frac{1}{3}}|\al|^{-\frac{1}{3}}(y_{2}-y_{1})^{-\frac{1}{3}}$, we get
\begin{align}\label{equ:H-1,w,F}
\n{w}_{L^{2}}^{2}\lesssim & \frac{\n{F}_{H^{-1}}^{2}}{\nu|\alpha|(y_{2}-y_{1})\delta}+\frac{\n{F}_{H^{-1}}^{2}}{|\alpha|^2(y_{2}-y_{1})^{2}\delta^{4}}
+\frac{\delta^{2}\n{F}_{H^{-1}}^{2}}{\nu^{2}}\lesssim \frac{\delta^{2}\n{F}_{H^{-1}}^{2}}{\nu^{2}}.
\end{align}
Hence, we complete the proof of \eqref{equ:re,estimate,L2,lemma}--\eqref{equ:re,estimate,H-1,lemma} for $\n{w}_{L^{2}}$. Moreover, by \eqref{equ:L2,w,F} and \eqref{equ:H-1,w,F}, we actually have
\begin{align}\label{equ:in,ex,L2,H-1}
\mathcal{E}^{in}(w)+\mathcal{E}^{ex}(w)\lesssim \min\lr{\frac{\delta^{2}}{\nu^{2}}\n{F}_{H^{-1}}^{2},\frac{\delta^{4}}{\nu^{2}}\n{F}_{L^{2}}^{2}}.
\end{align}
 We are left to deal with $\n{w}_{L^{1}}$ and $\n{(\pa_{y},|\al|)w}_{L^{2}}$.

\vspace{1em}
\textbf{Estimate on $\n{w}_{L^{1}}$.}
By H\"{o}lder inequality, exterior $L^{1}$ estimate \eqref{equ:estimate,w,L1,ex}, and \eqref{equ:in,ex,L2,H-1}, we have
\begin{align*}
\n{w}_{L^{1}}^{2}\leq& (y_{2}-y_{1})\n{w}_{L^{2}(y_{1},y_{2})}^{2}+\n{w}_{L^{1}((-1,1)\setminus (y_{1},y_{2}))}^{2}\\
\lesssim & (y_{2}-y_{1})\lrs{\mathcal{E}^{in}(w)+\mathcal{E}^{ex}(w)}\\
\lesssim&(y_{2}-y_{1})\min\lr{\frac{\delta^{2}}{\nu^{2}}\n{F}_{H^{-1}}^{2},\frac{\delta^{4}}{\nu^{2}}\n{F}_{L^{2}}^{2}}.
\end{align*}

\textbf{Estimate on $\n{(\pa_{y},|\al|)w}_{L^{2}}$.} By energy estimate \eqref{equ:energy estimate,real part,w,nu} and \eqref{equ:L2,w,F}, we have
\begin{align*}
\n{(\pa_{y},|\al|)w}_{L^{2}}^{2}\lesssim  \n{w}^2_{L^2}+\nu^{-1}|\operatorname{Re}\lra{F,w}|\lesssim \frac{\n{w}_{L^{2}}\n{F}_{L^{2}}}{\nu}
\lesssim \frac{\delta^{2}}{\nu^{2}}\n{F}_{L^{2}}^{2}.
\end{align*}
Similarly, by \eqref{equ:energy estimate,H1,used} and \eqref{equ:H-1,w,F}, we get
\begin{align*}
\n{(\pa_{y},|\al|)w}_{L^{2}}   \lesssim& \n{w}_{L^2}+\nu^{-1}\n{F}_{H^{-1}}\lesssim \nu^{-1}\n{F}_{H^{-1}}.
\end{align*}
This completes the proof.
\end{proof}

Now we are in a position to prove resolvent estimates.
\begin{proposition}[Resolvent estimates for $F\in L^{2}$ and $F\in H^{-1}$]\label{lemma:re,w,full}
For $\la\in \R$, we have
\begin{align}
\nu^{\frac{1}{3}}|\al|^{\frac{2}{3}}\lrs{|\la-U(0)|^{\frac{1}{2}}+\nu^{\frac{1}{4}}|\al|^{-\frac{1}{4}}}^{\frac{2}{3}}\n{w}_{L^{2}}\lesssim & \n{F}_{L^{2}},\label{equ:re,L2,w,full}\\
\nu^{\f23}|\alpha|^{\f13}\lrs{|\lambda-U(0)|^{\f12}+\nu^{\f14}|\alpha|^{-\f14}}^{\f13}\n{(\pa_{y},|\al|)w}_{L^{2}}\lesssim&
\n{F}_{L^{2}},\label{equ:re,L2,wH1,full}\\
\nu^{\frac{2}{3}}|\al|^{\frac{1}{3}}\lrs{|\la-U(0)|^{\frac{1}{2}}+\nu^{\frac{1}{4}}|\al|^{-\frac{1}{4}}}^{\frac{1}{3}}\n{w}_{L^{2}}\lesssim & \n{F}_{H^{-1}}.\label{equ:re,H-1,w,full}
\end{align}
In particular, we have
\begin{align}
\nu^{\f38}|\alpha|^{\f58}\lrs{\n{w}_{L^1}+|\alpha|^{\f12}\|u\|_{L^2}}+\nu^{\f34}|\alpha|^{\f14}\n{(\pa_{y},|\al|)w}_{L^{2}}+\nu^{\f12}|\alpha|^{\f12}\n{w}_{L^2} \lesssim&\n{F}_{L^2},\label{equ:re,L2,w,main}\\
\nu\n{(\pa_{y},|\al|)w}_{L^2}+\nu^{\frac{3}{4}}|\al|^{\frac{1}{4}}\n{w}_{L^{2}}\lesssim& \n{F}_{H^{-1}}.\label{equ:re,H-1,w,main}
\end{align}
\end{proposition}
\begin{proof}
We divide the proof into the following cases.\smallskip

{\it Case 1.} $\lambda\leq  U(0)+\nu^{\frac{1}{2}}|\al|^{-\frac{1}{2}}.$

By Young's inequality and resolvent estimate \eqref{equ:energy estimate,im,w,0},
we have
\begin{align*}
&\nu^{\frac{1}{3}}|\al|^{\frac{2}{3}}\lrs{|\la-U(0)|^{\frac{1}{2}}+\nu^{\frac{1}{4}}|\al|^{-\frac{1}{4}}}^{\frac{2}{3}}\n{w}_{L^{2}}^{2}\\
\lesssim& |\al||\la-U(0)|\n{w}_{L^{2}}^{2}+\nu^{\frac{1}{2}}|\al|^{\frac{1}{2}}\n{w}_{L^{2}}^{2}\lesssim \babs{\blra{F,\frac{w}{U''}}},
\end{align*}
which implies \eqref{equ:re,L2,w,full}. Then using that
\begin{align}\label{equ:energy estimate,wH1,w,used}
\n{(\partial_y,|\alpha|)w}_{L^2}\lesssim& \n{w}_{L^2}+\nu^{-1}\n{F}_{H^{-1}},
\end{align}
 we obtain \eqref{equ:re,L2,wH1,full} and \eqref{equ:re,H-1,w,full}.\smallskip

{\it Case 2.} $\lambda\geq U(1).$

By Young's inequality and resolvent estimate \eqref{equ:energy estimate,im,w,1},
we have
\begin{align*}
&\nu^{\frac{1}{3}}|\al|^{\frac{2}{3}}\lrs{|\la-U(0)|^{\frac{1}{2}}+\nu^{\frac{1}{4}}|\al|^{-\frac{1}{4}}}^{\frac{2}{3}}\n{w}_{L^{2}}^{2}\\
\lesssim&\nu^{\frac{1}{3}}|\al|^{\frac{2}{3}}|\la-U(1)|^{\frac{1}{3}}\n{w}_{L^{2}}^{2}+\nu^{\frac{1}{3}}|\al|^{\frac{2}{3}}\n{w}_{L^{2}}^{2}\\
\lesssim& |\al||\la-U(1)|\n{w}_{L^{2}}^{2}+\nu^{\frac{1}{3}}|\al|^{\frac{2}{3}}\n{w}_{L^{2}}^{2}\lesssim \babs{\blra{F,\frac{w}{U''}}},
\end{align*}
which, together with \eqref{equ:energy estimate,wH1,w,used}, implies \eqref{equ:re,L2,w,full}, \eqref{equ:re,L2,wH1,full}, and \eqref{equ:re,H-1,w,full}.\smallskip

	{\it Case 3.1.} $\lambda \in (U(0)+\nu^{\frac{1}{2}}|\al|^{-\frac{1}{2}},U(1))$, $\nu|\al|^{2}\lesssim |\la-U(0)|^{\frac{1}{2}}+\nu^{\frac{1}{4}}|\al|^{-\frac{1}{4}}$.

In the case, we use Lemma \ref{lemma:resol,w,L2,H-1,delta} to get
\begin{align}
\frac{\nu}{\delta^{2}}\n{w}_{L^{2}}+\frac{\nu(y_{2}-y_{1})^{\frac{1}{2}}}{\delta^{2}}\n{w}_{L^{1}}
+\frac{\nu}{\delta}\n{(\pa_{y},|\al|)w}_{L^{2}}\lesssim& \n{F}_{L^{2}},\label{equ:re,estimate,L2,lemma,proof}\\
\frac{\nu}{\delta}\n{w}_{L^{2}}+\frac{\nu(y_{2}-y_{1})^{\frac{1}{2}}}{\delta}\n{w}_{L^{1}}\lesssim& \n{F}_{H^{-1}}.
\end{align}
Due to that $\delta=\nu^{\frac{1}{3}}|\al|^{-\frac{1}{3}}(y_{2}-y_{1})^{-\frac{1}{3}}$, $\nu^{\frac{1}{4}}|\al|^{-\frac{1}{4}}\leq y_{2}\leq 1$, $(y_{2}-y_{1})^{2}\sim \la-U(0)$, we arrive at the desired resolvent estimates \eqref{equ:re,L2,w,full}, \eqref{equ:re,L2,wH1,full}, and \eqref{equ:re,H-1,w,full}.\smallskip

{\it Case 3.2.} $\lambda \in (U(0)+\nu^{\frac{1}{2}}|\al|^{-\frac{1}{2}},U(1))$, $\nu|\al|^{2}\gtrsim |\la-U(0)|^{\frac{1}{2}}+\nu^{\frac{1}{4}}|\al|^{-\frac{1}{4}}$.

Noticing that $\nu|\al|^{2}\gtrsim \nu^{\frac{1}{4}}|\al|^{-\frac{1}{4}}$ which implies that $|\al|\gtrsim \nu^{-\frac{1}{3}}\gg 1$, by energy estimates \eqref{equ:energy estimate,real part,w,nu}--\eqref{equ:energy estimate,real part,w,H-1}, we arrive at
\begin{align}\label{equ:energy estimate,real part,strong}
\nu|\al|\n{(\pa_{y},|\al|)w}_{L^{2}}\lesssim&\n{F}_{L^{2}},\quad
\nu\n{(\pa_{y},|\al|)w}_{L^{2}}\lesssim \n{F}_{H^{-1}}.
\end{align}
Therefore, using the condition that $\nu|\al|^{2}\gtrsim |\la-U(0)|^{\frac{1}{2}}+\nu^{\frac{1}{4}}|\al|^{-\frac{1}{4}}$, we have
\begin{align*}
&\nu^{\frac{1}{3}}|\al|^{\frac{2}{3}}\lrs{|\la-U(0)|^{\frac{1}{2}}+\nu^{\frac{1}{4}}|\al|^{-\frac{1}{4}}}^{\frac{2}{3}}\n{w}_{L^{2}}\\
\lesssim&\nu^{\frac{1}{3}}|\al|^{\frac{2}{3}}(\nu|\al|^{2})^{\frac{2}{3}}\n{w}_{L^{2}}=\nu|\al|^{2}\n{w}_{L^{2}}\lesssim \n{F}_{L^{2}},
\end{align*}
and complete the proof of \eqref{equ:re,L2,w,full}. In the same way, by using \eqref{equ:energy estimate,real part,strong} we also obtain estimates \eqref{equ:re,L2,wH1,full} and \eqref{equ:re,H-1,w,full}.

For \eqref{equ:re,L2,w,main} and \eqref{equ:re,H-1,w,main}, we are left to deal with $\n{w}_{L^1}+|\alpha|^{\f12}\|u\|_{L^2}$, as other terms can be treated by using \eqref{equ:energy estimate,real part,w,H-1}, \eqref{equ:re,L2,w,full}, \eqref{equ:re,L2,wH1,full}, and \eqref{equ:re,H-1,w,full}. Due to that $|\al|^{\frac{1}{2}}\n{u}_{L^{2}}\lesssim \n{w}_{L^{1}}$, we only need to estimate $\n{w}_{L^{1}}$.

For the case $\lambda\leq  U(0)+\nu^{\frac{1}{2}}|\al|^{-\frac{1}{2}}$, by energy estimate \eqref{equ:L2,w,resolvent,0,full,proof} and \eqref{equ:basic estimate,U,L1}, we have
\begin{equation*}
\begin{aligned}
\n{w}_{L^1}  =&\int_{-\theta}^\theta|w| \mathrm{d} y+\int_{(-1,1) \backslash(-\theta, \theta)}|w| \mathrm{d} y \\
 \leq&(2 \theta)^{\frac{1}{2}}\n{w}_{L^2}+\left\|\sqrt{U(y)-U(0)} w\right\|_{L^2}\left\|\left(U(y)-U(0)\right)^{-\frac{1}{2}}\right\|_{L^2((-1,1) \backslash(-\theta, \theta))} \\
 \lesssim& \theta^{\frac{1}{2}}\left(\nu\left|\al\right|\right)^{-\frac{1}{2}}\n{F}_{L^2}+\theta^{-\frac{1}{2}} \nu^{-\frac{1}{4}}\left|\al\right|^{-\frac{3}{4}}\n{F}_{L^2}
 \lesssim \nu^{-\frac{3}{8}}\left|\al\right|^{-\frac{5}{8}}\n{F}_{L^2},
\end{aligned}
\end{equation*}
where in the last inequality we have optimizied the choice of $\theta$.

For the case $\la\geq U(1)$, by Cauchy-Schwarz inequality, we have that
$$\n{w}_{L^{1}}\lesssim \n{w}_{L^{2}}\lesssim \nu^{-\frac{1}{3}}|\al|^{-\frac{2}{3}}\n{F}_{L^{2}}
\lesssim  \nu^{-\frac{3}{8}}\left|\al\right|^{-\frac{5}{8}}\n{F}_{L^{2}}.$$

For the case $\lambda \in (U(0)+\nu^{\frac{1}{2}}|\al|^{-\frac{1}{2}},U(1))$, we point out that \eqref{equ:re,estimate,L2,lemma,proof} for $F\in L^{2}$ is valid, as we can use a more stronger estimate \eqref{equ:estimate,vp,B,infty,L2} in Lemma \ref{Local-L-inf-est-psi}.
Therefore, we use \eqref{equ:re,estimate,L2,lemma,proof} to arrive at the desired estimate \eqref{equ:re,L2,w,main} for $\n{w}_{L^{1}}$.
\end{proof}

\subsection{Weak-type resolvent estimates}
In this subsection, we shall establish the weak-type resolvent estimate, which will be helpful in deriving more refined estimates for $\n{u}_{L^{2}}$. As a preliminary step, we give basic estimates for the Rayleigh equation, which is given by
\begin{equation}\label{equ:Ray,U,delta}
	\text{\bf Ray}_{\delta}W=:(U(y) - \lambda + i\delta)W-U''(y)\Psi= f, \quad (\partial_y^2 - |\alpha|^2)\Psi=W, \quad \Psi(\pm 1) = 0,
\end{equation}
 where $\lambda \in \mathbb{R}$. For $\delta\in \R\setminus \lr{0}$, we denote by $W = \text{\bf Ray}_\delta^{-1} f$, which is the solution of \eqref{equ:Ray,U,delta}.

\begin{lemma}\label{lemma:ray,est,W,Phi}
	   Let $(W,\Psi)$ solve \eqref{equ:Ray,U,delta} with $f(\pm 1) = 0$ and $\delta\ll 1$. Then it holds that
	   \begin{align}\label{equ:ray,est,W,Phi}
	   	\n{(\partial_y, |\alpha|)\Psi}_{L^2} + \delta^{\frac{1}{2}}\n{W}_{L^2} + \delta^{\frac{3}{2}}(|\lambda-U(0)| + \delta)^{-\frac{1}{2}}\n{\pa_yW}_{L^2} \lesssim  \bbn{(\partial_y, |\al|)\frac{f}{U''}}_{L^2}.
	   \end{align}
\end{lemma}

\begin{proof}  First, by Lemma \ref{lemma:proper,phi,ray,equ}, we have
that $\n{(\partial_y, |\alpha|)\Psi}_{L^2}\lesssim \bbn{(\partial_y, |\al|)\frac{f}{U''}}_{L^2}$,
 and hence obtain
\begin{align}\label{lemma:est,Phi,H1}
          \n{W}_{H^{-1}}=\n{(\partial_y, |\alpha|)\Psi}_{L^2} \lesssim \n{(\partial_y, |\alpha|)f}_{L^2}
          \lesssim \bbn{(\partial_y, |\al|)\frac{f}{U''}}_{L^2}.
\end{align}
We take the inner product between \eqref{equ:Ray,U,delta} and $\frac{W}{U''}$ to yield
\begin{align*}
        \int_{-1}^{1} \frac{U - \la}{U''}|W|^2 dy + \n{(\partial_y, |\alpha|)\Psi}_{L^2}^2 + i\delta\bbn{\frac{W}{\sqrt{U''}}}_{L^2}^2 = \bblra{f,\frac{W}{U''}} ,
\end{align*}
which together with \eqref{lemma:est,Phi,H1} implies
\begin{align}\label{lemma:est,W,L2}
        \delta\n{W}_{L^2}^2\lesssim
        \bbabs{\operatorname{Im}\bblra{f,\frac{W}{U''}}}\lesssim \bbn{\frac{f}{U''}}_{H_{0}^1}\n{W}_{H^{-1}}
         \lesssim \n{(\partial_y, |\alpha|)f}_{L^2}^2,
\end{align}
where in the second-to-last inequality we have used that $f(\pm 1)=0$.
Noting that
\begin{align*}
	    \pa_yW = \frac{f' +U''\pa_y\Psi+U'''\Psi}{U - \la + i\delta} + \frac{U'W}{U- \la + i\delta},
\end{align*}
by \eqref{lemma:est,Phi,H1} and \eqref{lemma:est,W,L2}, we have
\begin{align*}
         \n{\pa_yW}_{L^2}
         \leq&\delta^{-1}\n{f' +U''\pa_y\Psi+U'''\Psi}_{L^2}+\n{W}_{L^2}\bbn{\frac{U'}{U- \la + i\delta}}_{L^\infty}\\
         \lesssim& \delta^{-1}\n{(\partial_y, |\alpha|)f}_{L^2}+\delta^{-\frac12}\n{(\partial_y, |\alpha|)f}_{L^2}\delta^{-1}
         \big(|\la-U(0)|+\delta\big)^{\frac12}\\
         \lesssim& \delta^{-\frac32}
         \big(|\la-U(0)|+\delta\big)^{\frac12} \bbn{(\partial_y, |\al|)\frac{f}{U''}}_{L^2},
\end{align*}
where in the second-to-last inequality we have used the fact that
$$\bn{\frac{U'}{U- \la + i\delta}}_{L^\infty}\lesssim \delta^{-1}
\big(|\lambda-U(0)|+\delta\big)^{\frac12}.$$
Hence, we complete the proof of \eqref{equ:ray,est,W,Phi}.
\end{proof}

 Next, we establish the weak-type resolvent estimates for $F\in L^2$ and $F\in H^{-1}$, respectively. Here, we set parameters\footnote{The parameter $\delta_{1}$ might be large when $|\la-U(0)|\gg 1$.} as follows
\[
\delta=\nu^{\frac14}|\al|^{-\frac14}\ll 1,\quad  \delta_1=\nu^{\frac13}|\al|^{-\frac13}\big(|\la-U(0)|^{\frac12}+\nu^{\frac14}|\al|^{-\frac14}\big)^{\frac23}=
\delta^{\frac{4}{3}}\lrs{|\la-U(0)|^{\frac12}+\delta}^{\frac{2}{3}}.
\]

\begin{lemma}[Weak-type resolvent estimate for $F\in L^{2}$]\label{lemma:weak,est,w,FL2}
		 Let $(w,\psi)$ be the solution to \eqref{equ:OS,navier-slip} with $F \in L^2$. If $\nu |\al|^2 \lesssim |\la- U(0)|^{\frac{1}{2}} + |\nu|^{\frac{1}{4}}|\al|^{-\frac14}$ and $f(\pm 1) = 0$, then  it holds that
		 \begin{align*}
		  \babs{\blra{\frac{w}{U''}, f }} \lesssim |\al|^{-1} \n{F}_{L^2} \lrs{\n{\text{\bf Ray}_{\delta_1}^{-1} f}_{L^2} + |\nu|^{\frac{1}{2}}|\al|^{-\frac12} \delta_1^{-\frac{1}{2}} \n{\partial_y (\text{\bf Ray}_{\delta_1}^{-1} f)}_{L^2} }.
		 \end{align*}
	\end{lemma}
	\begin{proof}
For $\vp \in H_0^1(-1,1)$, we test \eqref{equ:OS,navier-slip} by $\frac{\vp}{U''}$ and use the integration by parts to get
     \begin{align}\label{lemma:est,test,phi,U''}
     	&\babs{\lra{\al (U - \la + i \delta_1) w - \al U''(\pa_y^2 - |\al|^2)^{-1} w, \frac{\vp}{U''}}}\\
     	\leq& \babs{\blra{F, \frac{\vp}{U''}}}+\nu \n{\pa_yw}_{L^2} \bn{\big(\frac{\vp}{U''}\big)'}_{L^2}+ \bbabs{\nu |\al|^2 - \al\delta_1} \n{w}_{L^2} \bn{\frac{\vp}{U''}}_{L^2}.\notag
     \end{align}
	Using the condition that $\nu|\al|^2\lesssim \abs{\la-U(0)}^{\frac12}+\nu^{\frac14} |\al|^{-\frac14}$, we obtain
 \begin{align*}(\nu|\al|^2)^{\frac23}\leq \big(\abs{\la-U(0)}^{\frac12}+\nu^{\frac14} |\al|^{-\frac14}\big)^{\frac23}=\delta_1 |\nu|^{-\frac{1}{3}}|\al|^{\frac13},
  \end{align*}
  which implies $\nu|\al|^2\lesssim \nu^{\frac{1}{3}}|\al|^{\frac{2}{3}}(\delta_1 |\nu|^{-\frac{1}{3}}|\al|^{\frac13})\leq |\al|\delta_{1}$. Hence, we further deduce from \eqref{lemma:est,test,phi,U''} that
     \begin{align*}
	& |\al| \bbabs{\bblra{ \frac{w}{U''}, (U- \la+ i \delta_1) \vp -  U''(\pa_y^2 - |\al|^2)^{-1} \vp }}\\
	\lesssim &\n{F}_{L^2} \bn{\frac{\vp}{U''}}_{L^2}+\nu \n{w'}_{L^2} \bn{\big(\frac{\vp}{U''}\big)'}_{L^2}+ |\al| \delta_{1} \n{w}_{L^2} \bn{\frac{\vp}{U''}}_{L^2}\\
	\lesssim& \n{F}_{L^2} \n{\vp}_{L^2} + \nu \n{\pa_yw}_{L^2} \n{\pa_y\vp}_{L^2}+ |\al| \delta_1 \n{w}_{L^2} \n{\vp}_{L^2},
    \end{align*}
	which along with estimates \eqref{equ:re,L2,w,full} and \eqref{equ:re,L2,wH1,full} in Proposition \ref{lemma:re,w,full} implies
		   \begin{align}\label{lemma:est-wphi,FL2,fina}
		   	& | \al|\babs{\blra{ \frac{w}{U''},(U- \la- i \delta_1)\vp - U'' (\pa_y^2 - |\al|^2)^{-1} \vp}}\\
			\lesssim& \n{F}_{L^2} \n{\vp}_{L^2} + \nu^{\frac12}|\al|^{-\frac12} \delta_1^{-\frac12} \n{F}_{L^2} \n{\pa_y\vp}_{L^2}.\nonumber
		\end{align}
Taking $\vp = \text{\bf Ray}_{\delta_1}^{-1} f$,
		we have finished the proof.
     \end{proof}

\begin{lemma}[Weak-type resolvent estimate for $F\in H^{-1}$]\label{lemma:weak,est,w,FH-1}
	 Let $(w,\psi)$ be the solution to \eqref{equ:OS,navier-slip} with $F\in H^{-1}$. If $\nu|\al|^2\lesssim \abs{\la-U(0)}^{\frac12}+\nu^{\frac14}|\al|^{-\frac14}$, then it holds that
	\begin{align}\label{equ:weak,est,w,FH-1}
		\babs{\blra{\frac{w}{U''}, f }}
		\lesssim &|\al|^{-1}\sum_{j\in \lr{-1,1}}|f(j)| \n{F}_{H^{-1}}\big(|U(1)-\la|+\delta^{\frac43}\big)^{-\frac34}\delta^{-1}\\
		&+ |\al|^{-1} \n{F}_{H^{-1}} \lrs{\bn
		{\text{\bf Ray}_{\delta_1}^{-1} f}_{H^{1}}+ \delta^{-\frac{4}{3}}(|\la-U(0) |^{\frac{1}{2}} + \delta)^{\frac{1}{3}}\bn{\text{\bf Ray}_{\delta_1}^{-1} f}_{L^2}}.\notag
	\end{align}
\end{lemma}

\begin{proof}
Following a procedure similar to the proof of estimate \eqref{lemma:est-wphi,FL2,fina}, we also have
	\begin{align}\label{equ:est,w,phi,U'',F,H-1}
		&|\al| \babs{\blra{ \frac{w}{U''}, (U- \la+ i \delta_1) \vp -U''  (\pa_y^2 - |\al|^2)^{-1} \vp } }\\
		\lesssim& \n{F}_{H^{-1}} \n{\vp}_{H^{1}} + \nu^{-\frac13} |\al|^{\frac{1}{3}}\big(|\la-U(0)|^{\frac12}+\delta\big)^{\frac13} \n{F}_{H^{-1}} \n{\vp}_{L^2},\nonumber
	\end{align}
	where we have used the $H^{-1}$ estimate \eqref{equ:re,H-1,w,full} in Proposition \ref{lemma:re,w,full}.
	
	Without loss of generality, we might assume that $\vp \in H^{1}(-1,1)$ with $\vp(-1) = 0$. Then, we construct an auxiliary function $\vp_1 \in H_0^1(-1,1)$. Specifically, for any $\delta_* \in (0, \delta^{\frac{4}{3}}] \subset (0,1]$, we introduce a smooth function $\chi_{w}(y)$ satisfying
\begin{align*}
&0\leq \chi_{w}(y)\leq 1,\quad \chi_w(1-\delta_{*})=0,\quad \chi_{w}(1)=1,\quad \text{supp}\chi_w \subset [1-\delta_*, 1],\\
		&\n{\chi_w}_{L^\infty} = 1, \quad \n{\chi_w}_{L^2} \leq\delta_*^{\frac{1}{2}},\quad
		\n{\chi_w'}_{L^2} \leq \delta_*^{-\frac12}, \\
		&\n{(\pa_y^2 - |\al|^2)^{-1}\chi_w}
		_{L^\infty} \leq \n{\chi_w}_{L^1(1-\delta_*, 1)} \lesssim \delta_*, \\
		&\n{(U(y)  - \la)\chi_w + U''(\pa_y^2 - |\al|^2)^{-1}\chi_w}_{L^\infty} \lesssim |U(1)-\la| + \delta_*.
	\end{align*}
	Taking $\vp_1(y):=\vp(y)-\vp(1)\chi_w(y)$, by estimate \eqref{equ:est,w,phi,U'',F,H-1}, we deduce
	\begin{align*}
		I:=& \bbabs{\blra{ \frac{w}{U''}, (U - \la + i\delta_1)\vp - U''(\pa_y^2 - |\al|^2)^{-1}\vp }} \\
		\leq& |\vp(1)| \bbabs{\blra{ \frac{w}{U''}, (U - \la + i\delta_1)\chi_w - U''(\pa_y^2 - |\al|^2)^{-1}\chi_w }} \\
		&  + \bbabs{\blra{ \frac{w}{U''}, (U- \la + i\delta_1)\vp_1 - U''(\pa_y^2 - |\al|^2)^{-1}\vp_1 }}\\
\lesssim&  |\vp(1)| \n{w}_{L^1(1-\delta_*, 1)} \n{(U - \la)\chi_w + U''(\pa_y^2 - |\al|^2)^{-1}\chi_w}_{L^\infty} + \delta_1 |\vp(1)|\n{w}_{L^2} \n{\chi_w}_{L^2} \\
		    & + |\al|^{-1} \n{F}_{H^{-1}} \n{\vp_1}_{H^{1}} + |\al|^{-1} \delta^{-\frac{4}{3}} (|\la -U(0) |^{\frac12} + \delta)^{\frac13} \n{F}_{H^{-1}} \n{\vp_1}_{L^2},
\end{align*}
where we have used that $\nu=\delta^{4}|\al|$. Noting that $\delta_{1}=\delta^{\frac{4}{3}}\lrs{|\la-U(0)|^{\frac12}+\delta}^{\frac{2}{3}}$,
by resolvent estimates in Proposition \ref{lemma:re,w,full}, we have
\begin{align*}
\n{w}_{L^1(1-\delta_*, 1)} \leq &\delta_*^{\frac32}\n{\pa_yw}_{L^2}\lesssim\delta_*^{\frac32}\nu^{-1}\n{F}_{H^{-1}}=\delta_*^{\frac32}|\al|^{-1}\delta^{-4}\n{F}_{H^{-1}},\\
\delta_{1}\n{w}_{L^{2}}\lesssim & |\al|^{-1}\delta^{-\frac{4}{3}}(|\la-U(0)|^{\frac{1}{2}}+\delta)^{\frac{1}{3}}\n{F}_{H^{-1}}.
\end{align*}
Together with the above properties of $\chi_w$, we obtain
\begin{align*}
	I	
		      \lesssim& |\al|^{-1}|\vp(1)|\n{F}_{H^{-1}} (|U(1)-\la| + \delta_*)\delta_*^{\frac32} \delta^{-4}  \\
		    &  + |\al|^{-1} |\vp(1)| \n{F}_{H^{-1}}(\n{\chi_w}_{H^{1}} + \delta^{-\frac{4}{3}}(|\la -U(0) |^{\frac12} + \delta)^{\frac13} \n{\chi_w}_{L^2}) \\
		     &  + |\al|^{-1} \n{F}_{H^{-1}} (\n{\vp}_{H^{1}} + \delta^{-\frac{4}{3}}  (|\la -U(0)|^{\frac12} + \delta)^{\frac13} \n{\vp}_{L^2})\\
\lesssim& |\al|^{-1}|\vp(1)|\n{F}_{H^{-1}}\lrs{|U(1)-\la| + \delta_*}\delta_{*}^{\frac{3}{2}}\delta^{-4} \\
&+ |\al|^{-1} |\vp(1)| \n{F}_{H^{-1}}\lrs{\delta_{*}^{-\frac{1}{2}} + \delta^{-\frac{4}{3}}(|\la -U(0) |^{\frac12} + \delta)^{\frac13}
\delta_{*}^{\frac{1}{2}}} \\
		     &  + |\al|^{-1} \n{F}_{H^{-1}} \lrs{\n{\vp}_{H^{1}} + \delta^{-\frac{4}{3}}  (|\la -U(0)|^{\frac12} + \delta)^{\frac13} \n{\vp}_{L^2}}.
\end{align*}
Taking $\delta_* = (|U(1)-\la| + \delta^{\frac43})^{-\frac12} \delta^2\lesssim \delta^{\frac{4}{3}}$, we have
\begin{align*}
(|U(1)-\la| + \delta_*)\delta_{*}^{\frac{3}{2}}\delta^{-4}\lesssim (|U(1)-\la| + \delta^{\frac43})^{\frac14} \delta^{-1}.
\end{align*}
Using that $(|\la -U(0) |^{\frac12} + \delta)^{\frac13} \lesssim \delta^{-\frac23} (|U(1)-\la| + \delta^{\frac43})^{\frac12},$ we get
	\begin{align}\label{equ:weak,est,w,FH-1,proof,I}
		I\lesssim &|\al|^{-1} |\vp(1)| \n{F}_{H^{-1}} (|U(1)-\la| + \delta^{\frac43})^{\frac14} \delta^{-1} \\
		&  + |\al|^{-1} \n{F}_{H^{-1}} (\n{\vp}_{H^{1}} + \delta^{-\frac{4}{3}}  (|\la -U(0) |^{\frac12} + \delta)^{\frac13} \n{\vp}_{L^2}).\notag
	\end{align}
By taking $\vp = \text{\bf Ray}_{\delta_1}^{-1} f$ with $(U(1)-\la + i\delta_1)\vp(1) = f(1)$, we have
$$|\vp(1)| \leq (|U(1)-\la| + \delta_1)^{-1}|f(1)|\lesssim (|U(1)-\la| + \delta^{\frac{4}{3}})^{-1}|f(1)|,$$
where we have used that $\delta^{\frac{4}{3}}\lesssim |U(1)-\la|+\delta_{1}$.
Together with \eqref{equ:weak,est,w,FH-1,proof,I}, we arrive at
	\begin{align*}
		\bbabs{\blra{\frac{w}{U''}, f }}
		\lesssim& |\al|^{-1}|f(1)|\n{F}_{H^{-1}}(|U(1)-\la| + \delta^{\frac43})^{-\frac34}\delta^{-1} \\
		& + |\al|^{-1}\n{F}_{H^{-1}}(\n{\text{\bf Ray}_{\delta_1}^{-1} f}_{H^{1}} + \delta^{-\frac{4}{3}}(|\la-U(0) |^{\frac{1}{2}} + \delta)^{\frac{1}{3}}\n{\text{\bf Ray}_{\delta_1}^{-1} f}_{L^2}),
	\end{align*}
which completes the proof of \eqref{equ:weak,est,w,FH-1} for the case $f(-1)=0$.

For the general case $f(\pm 1)\neq 0$, we consider $\vp_1(y):=\vp(y)-\vp(1)\chi_w(y)-\vp(-1)\chi_{w}(-y)$. Repeating the proof, we arrive at the weak-type resolvent estimate \eqref{equ:weak,est,w,FH-1}.
\end{proof}

 With the help of the above weak-resolvent estimates, we are able to set up the following stronger estimates of $\n{u}_{L^2}$ for $F\in L^2$ and $F\in H^{-1}$, respectively.

    	\begin{proposition} \label{proposition:est,stronger,u,FL2,FH-1}
    		Under the same assumptions of Proposition $\ref{lemma:re,w,full}$, it holds that
    	\begin{align}
    		\nu^{\frac{1}{4}}|\al|^{\frac{3}{4}}\n{u}_{L^{2}}\leq \nu^{\frac{1}{6}} |\al|^{\frac{5}{6}}\Big( |\la- U(0)|^{\frac{1}{2}} + \nu^{\frac{1}{4}}|\al|^{-\f14} \Big)^{\frac{1}{3}} \n{u}_{L^2} \lesssim & \n{F}_{L^2},\label{equ:est,stronger,u,FL2}\\
		\nu^{\f12}|\al|^{\f12}\n{u}_{L^2} \lesssim&  \n{F}_{H^{-1}}.\label{equ:est,stronger,u,FH-1}
	\end{align}
    	\end{proposition}
    	\begin{proof}

We first deal with \eqref{equ:est,stronger,u,FL2} for $F\in L^{2}$. We divide the proof into two cases.\smallskip
    		
    		 {\it Case 1 with $F\in L^{2}$.}  $\nu |\al|^2 \gtrsim |\la- U(0)|^{\frac{1}{2}} +\nu^{\f14}|\al|^{-\f14}$.
    		
    		 In this case, we have $\nu|\al|^2\gtrsim \nu^{\f14}|\al|^{-\f14}$ which gives $\nu|\al|^3\gtrsim 1$. Using estimate \eqref{equ:proper,u,L2L1} in Lemma \ref{lemma:proper,u,L2,Linf} and estimate \eqref{equ:re,L2,wH1,full} in Proposition \ref{lemma:re,w,full}, we arrive at
    		\begin{align*}
    			\n{u}_{L^2} & \lesssim |\al|^{-1} \n{w}_{L^2}\lesssim  |\al|^{-2} \n{(\pa_y, |\al|) w}_{L^2}\\
    			&\lesssim \nu^{-\frac{2}{3}} |\al|^{-\frac{1}{3}} |\al|^{-2} \big(|\la - U(0)|^{\frac{1}{2}} +\nu^{\f14}|\al|^{-\f14})^{-\frac{1}{3}} \n{F}_{L^2}\\
    			&= \nu^{-\frac{1}{6}} |\al|^{-\frac{5}{6}} (|\la - U(0)|^{\frac{1}{2}} +\nu^{\f14}|\al|^{-\f14})^{-\frac{1}{3}} \n{F}_{L^2} \times (\nu |\al|^3 )^{-\frac{1}{2}}\\
    			&\lesssim \nu^{-\frac{1}{6}} |\al|^{-\frac{5}{6}} (|\la - U(0)|^{\frac{1}{2}} +\nu^{\f14}|\al|^{-\f14})^{-\frac{1}{3}} \n{F}_{L^2},
    		\end{align*}
    		which implies \eqref{equ:est,stronger,u,FL2}.\smallskip
    	
    	{\it Case 2 with $F\in L^{2}$.}  $\nu |\al|^2 \lesssim |\la- U(0)|^{\frac{1}{2}} +\nu^{\f14}|\al|^{-\f14}$.

     	Taking $f = U''\psi$ in Lemma \ref{lemma:weak,est,w,FL2}, we deduce that
    	\begin{align*}
    		\n{u}_{L^2}^2 =& |\lra{ w, \psi}|=\babs{\blra{ \frac{w}{U''}, U''\psi}}\\
    \lesssim&|\al|^{-1} \n{F}_{L^2} \lrs{\n{\text{\bf Ray}_{\delta_1}^{-1} (U''\psi)}_{L^2} + |\nu|^{\frac{1}{2}}|\al|^{-\f12} \delta_1^{-\frac{1}{2}} \n{\partial_y (\text{\bf Ray}_{\delta_1}^{-1} (U''\psi))}_{L^2} }.
    	\end{align*}
    Using Lemma \ref{lemma:ray,est,W,Phi}, we have
    	\begin{align*}
\n{u}_{L^2}^2
    		\lesssim& |\al|^{-1} \n{F}_{L^2}  \big( \delta_1^{-\frac{1}{2}}  + \nu^{\f12}|\al|^{-\f12}
    		\delta_1^{-2} (|U(0) - \la| + \delta_1)^{\frac{1}{2}} \big) \n{(\pa_{y},|\al|)\psi}_{L^2}.
    	\end{align*}
 Noting that $\delta_1 = \nu^{\f13}|\al|^{-\f13}(|\la-U(0) |^{\frac{1}{2}} + \nu^{\f14}|\al|^{-\f14})^{\frac{2}{3}}$, we have
  \begin{align*}
  |U(0) - \la| + \delta_1 \lesssim (|\la - U(0)|^{\frac{1}{2}} + \nu^{\f14}|\al|^{-\f14})^2,\quad \nu^{\f12}{\al}^{-\f12} \delta_1^{-\frac{3}{2}} (|U(0)-\la| + \delta_1)^{\frac{1}{2}}\lesssim 1,
  \end{align*}
   which gives
    \begin{align*}
    \n{u}_{L^{2}}^{2}\lesssim |\al|^{-1} \delta_1^{-\frac{1}{2}} \n{F}_{L^2} \n{u}_{L^2}.
    \end{align*}
   Hence, we complete the proof of \eqref{equ:est,stronger,u,FL2}.

        Next, we prove \eqref{equ:est,stronger,u,FH-1} for $F\in H^{-1}$. We also consider the following two cases.\smallskip

	{\it Case 1 with $F\in H^{-1}$.} $|U(0) - \la| \gg 1$.
	
    In this case, we can use energy estimates \eqref{equ:energy estimate,im,w,0}--\eqref{equ:energy estimate,im,w,1} in Lemma \ref{lemma:energy estimate,real part} and obtain
    \begin{align*}
    |\al|\n{w}_{L^{2}}^{2}\lesssim \babs{\blra{F,\frac{w}{U''}}}\lesssim \n{F}_{H^{-1}}\n{w}_{H^{1}},
    \end{align*}
    which together with \eqref{equ:energy estimate,real part,w,H-1} in Lemma \ref{lemma:energy estimate,real part} implies that
    \begin{align*}
    \n{w}_{L^{2}}^{2}\lesssim \nu^{-1}|\al|^{-1}\n{F}_{H^{-1}}^{2}.
    \end{align*}
   By \eqref{equ:proper,u,L2L1} in Lemma \ref{lemma:proper,u,L2,Linf}, we arrive at
   \begin{align*}
   \n{u}_{L^2}\lesssim |\al|^{-1}\n{w}_{L^2}\lesssim \nu^{-\frac{1}{2}}|\al|^{-\frac{3}{2}}\n{F}_{H^{-1}}\lesssim\nu^{-\frac{1}{2}}|\al|^{-\frac{1}{2}}\n{F}_{H^{-1}}.
   \end{align*}

	{\it Case 2.1 with $F\in H^{-1}$. } $\abs{U(0) - \la} \lesssim 1$ and $\nu |\al|^2\gtrsim |U(0) - \la|^{\frac{1}{2}} + \nu^{\f14}|\al|^{-\f14}$.
	
	In this case, we have $\nu |\al|^2\geq \nu^{\f14}|\al|^{-\f14}$ which yields $\nu^{-1}\leq |\al|^{3}$. Together with  \eqref{equ:re,H-1,w,main} in Proposition \ref{lemma:re,w,full} and \eqref{equ:proper,u,L2L1} in Lemma \ref{lemma:proper,u,L2,Linf}, we get
	\begin{align*}
		\n{u}_{L^2}\lesssim |\al|^{-1}\n{w}_{L^2}\lesssim |\al|^{-1}\nu^{-\f34}|\al|^{-\f14}\n{F}_{H^{-1}}\lesssim \nu^{-\f12}|\al|^{-\f12}\n{F}_{H^{-1}}.
	\end{align*}
	
	{\it Case 2.2 with $F\in H^{-1}$.} $\abs{U(0) - \la} \lesssim 1$ and $\nu |\al|^2\leq \abs{U(0) - \la}^{\frac{1}{2}} + \nu^{\f14}|\al|^{-\f14}$.
	
In this case, by Lemma \ref{lemma:weak,est,w,FH-1} with $f=U''\psi$, noting that $(U''\psi)(\pm 1)=0$, we have
	\begin{align}\label{equ:est,stronger,u,FH-1,proof,1}
		&\n{u}^2_{L^2}=\babs{\blra{\frac{w}{U''}, U''\psi} }\\
		\lesssim& |\al|^{-1} \n{F}_{H^{-1}}\big(\n
		{\text{\bf Ray}_{\delta_1}^{-1} (U''\psi)}_{H^{1}}+\delta^{-\f43}(\abs{U(0)-\la}^{\f12}+\delta)^{\f13}\n
		{\text{\bf Ray}_{\delta_1}^{-1} (U''\psi)}_{L^2}\big).\notag
	\end{align}
Then with $\delta_{1}=\delta^{\frac{4}{3}}\lrs{|\la-U(0)|^{\f12}+\delta}^{\frac{2}{3}}$ and $\delta=\nu^{\frac{1}{4}}|\al|^{-\frac{1}{4}}$, we use Lemma
 \ref{lemma:ray,est,W,Phi}
with $f = U''\psi$ to arrive at
	\begin{align}\label{equ:est,stronger,u,FH-1,proof,2}
		&\n{\text{\bf Ray}_{\delta_1}^{-1} (U''\psi)}_{H^{1}}+\delta^{-\f43}(\abs{U(0)-\la}^{\f12}+\delta)^{\f13}\n
		{\text{\bf Ray}_{\delta_1}^{-1} (U''\psi)}_{L^2}\\
		\lesssim & \Big(\delta_1^{-\frac{3}{2}} (|\la - U(0)| + \delta_1)^{\frac{1}{2}}  + \delta_1^{-\frac{1}{2}} \delta^{-\frac{4}{3}} (|\la-U(0)|^{\f12}+\delta)^{\f13} \Big)\n{(\pa_{y},|\al|)\psi}_{L^2}\notag\\
\lesssim & \delta_1^{-\frac{1}{2}} \delta^{-\frac{4}{3}} (|\la-U(0)|^{\f12}+\delta)^{\f13} \n{(\pa_{y},|\al|)\psi}_{L^2}
		= \nu^{-\f12}|\al|^{\f12}\n{u}_{L^2},\notag
	\end{align}
where we have used that $(\abs{U(0)-\la}+\delta_1)^{\f12} \lesssim \abs{U(0)-\la}^{\f12}+\delta$. Combining \eqref{equ:est,stronger,u,FH-1,proof,1} and \eqref{equ:est,stronger,u,FH-1,proof,2}, we complete the proof of \eqref{equ:est,stronger,u,FH-1}.
        \end{proof}

 \subsection{Resolvent estimates for $F\in H_{0}^{1}$}\label{sec:Resolvent Estimates for F}

 In this subsection, we aim to set up the resolvent estimates for $F\in H^{1}_0$.

\begin{proposition}\label{proposition:reso,est,u,w,FH1}
 Assume that $\la\in \R$ and $F\in H^{1}_0(-1,1)$. Let $w$ be the solution of the system \eqref{equ:OS,navier-slip}, it holds that
\begin{align}\label{equ:reso,est,u,w,FH1}
 |\al|\n{u}_{L^2}+\nu^{\frac{1}{8}}|\al|^{\frac{7}{8}}\n{u}_{L^\infty}+\nu^{\f12}|\al|^{\f12}\n{(\pa_y,|\al|)w}_{L^2}+\nu^{\f14}|\al|^{\f34}\n{w}_{L^2}\lesssim \n{F}_{H^{1}}.
\end{align}
\end{proposition}

\begin{proof} First, noticing that
	\begin{align}\label{equ:est,w,L2,wH1,u}
		\n{w}^2_{L^2}\lesssim \n{(\pa_y,|\al|)w}_{L^2}\n{(\pa_{y},|\al|)\psi}_{L^{2}}=\n{(\pa_y,|\al|)w}_{L^2}\n{u}_{L^2},
	\end{align}
 from \eqref{equ:energy estimate,real part,w,nu} in Lemma \ref{lemma:energy estimate,real part}, we have
	\begin{align}\label{equ:est,nu,wH1,w'uF,L2}
		\nu\n{(\pa_y,|\al|)w}_{L^2}^2
		&\lesssim \nu\n{w}^2_{L^2}+\bbabs{\mathrm{Re}\blra{F,\frac{w}{U''}}}\nonumber\\
		&\lesssim \nu\n{(\pa_y,|\al|)w}_{L^2}\n{u}_{L^2}+\n{(\pa_y,|\al|)F}_{L^2}\n{u}_{L^2},
	\end{align}
	which implies
	\begin{align}\label{equ:est,wH1,w'uF,L2}
		\nu\n{(\pa_y,|\al|) w}^2_{L^2}\lesssim |\al|\n{u}^2_{L^2}+|\al|^{-1}\n{F}_{H^{1}}^{2}.
	\end{align}

	Next, we divide the proof into four cases.\smallskip
		
{\it Case 1.} $\la \leq U(0)$.

In this case, by \eqref{equ:L2,w,resolvent,0,full,proof} in Lemma \ref{lemma:energy estimate,real part}, we obtain
\begin{align}\label{equ:est,case1,wu,uwL2,FH1}
\abs{\al}\n{u}^2_{L^2}
	\lesssim& \babs{\blra{F,\frac{w}{U''}}}
	\lesssim  \bbn{\f{F}{U''}}_{H^{1}}\n{u}_{L^2}
	\lesssim \f{\abs{\al}}{2}\n{u}^2_{L^2}+\abs{\al}^{-1} \n{F}_{H^{1}}^2.
\end{align}
Together with \eqref{equ:est,wH1,w'uF,L2}, we have
\begin{align}\label{equ:est,case1,wH1,wL2,FH1}
	\nu\n{(\pa_y,|\al|)w}_{L^2}^2
	\lesssim |\al|^{-1}\n{F}^2_{H^{1}}.
\end{align}
Using \eqref{equ:est,case1,wu,uwL2,FH1} and \eqref{equ:est,case1,wH1,wL2,FH1} again, we have
\begin{align}\label{equ:est,case1,uLinfty,FH1}
    \n{u}_{L^\infty}\lesssim  \n{u}^{\f12}_{L^2}\n{w}^{\f12}_{L^2}\lesssim \nu^{-\f18}|\al|^{-\f78}\n{F}_{H^{1}}.
\end{align}

{\it Case 2.} $\la \geq U(1)$.

In this case, we use \eqref{equ:L2,w,resolvent,1,full,proof} in Lemma \ref{lemma:energy estimate,real part} to obtain
\begin{align}\label{equ:est,case1,wu,uwL2,FH1,la>1}
\abs{\al}\n{u}^2_{L^2}
	\lesssim& \babs{\blra{F,\frac{w}{U''}}}\lesssim \f{\abs{\al}}{2}\n{u}^2_{L^2}+\abs{\al}^{-1} \n{F}_{H^{1}}^2,
\end{align}
which together with \eqref{equ:est,wH1,w'uF,L2} gives
\begin{align*}
	\nu\n{(\pa_y,|\al|)w}_{L^2}^2
	\lesssim& |\al|^{-1}\n{F}^2_{H^{1}},\quad
\n{u}_{L^\infty}\lesssim \n{u}_{L^2}^{\f12}\n{w}_{L^2}^{\f12}\lesssim \nu^{-\f18}|\al|^{-\f78}\n{F}_{H^{1}}.
\end{align*}

{\it Case 3.1.} $\la \in (U(0),U(1))$ and $\nu|\al|^2\gtrsim |U(0)-\la|^{\f12}+\nu^{\f14}|\al|^{-\f14}$.

In this case, we notice that $\nu|\al|^2\gtrsim \nu^{\f14}|\al|^{-\f14}$ which gives $(\nu|\al|^3)^{\f14}\lesssim  \nu|\al|^3$. Hence, using resolvent estimates \eqref{equ:re,L2,w,main} in Proposition \ref{lemma:re,w,full}, we arrive at
\begin{align*}
	\n{w}_{L^2}\lesssim& \nu^{-\f12}|\al|^{-\f12}\n{F}_{L^2}\leq \nu^{-\f12}|\al|^{-\f32}\n{(\pa_y,|\al|)F}_{L^2}\lesssim \nu^{-\f14}|\al|^{-\f34}\n{F}_{H^{1}},\\
	\n{\pa_yw}_{L^2}\lesssim& \nu^{-\f34}|\al|^{-\f14}\n{F}_{L^2}\leq \nu^{-\f34}|\al|^{-\f54}\n{(\pa_y,|\al|)F}_{L^2}\lesssim  \nu^{-\f12}|\al|^{-\f12}\n{F}_{H^{1}},\\
	\n{u}_{L^2}\lesssim& \nu^{-\f38}|\al|^{-\f{9}{8}}\n{F}_{L^2}\leq \nu^{-\f38}|\al|^{-\f{17}{8}}\n{(\pa_y,|\al|)F}_{L^2}\lesssim  |\al|^{-1}\n{F}_{H^{1}},\\
	\n{u}_{L^\infty}\lesssim& \n{u}^{\f12}_{L^2}\n{w} ^{\f12}_{L^2}\lesssim \nu^{-\f18}|\al|^{-\f78}\n{F}_{H^{1}}.
\end{align*}

{\it Case 3.2.} $\la\in (U(0),U(1))$ and $\nu|\al|^2\lesssim |U(0)-\la|^{\f12}+\nu^{\f14}|\al|^{-\f14}$.

We decompose $w=w_3+w_4+w_5$, where $w_3$, $w_4$ and $w_5$ solve the following system
\begin{equation}\label{equ:navierslip,NS,per,decom345}
	\left\{
	\begin{aligned}
		&\big(U(y)-\la+i(-U''\delta_1-\nu|\al|)\big)w_3- U''\psi_3=\f{F}{i\alpha},\\
		&-\nu(\pa^2_y-\al^2)w_4+i\al\big((U-\la)w_4- U''\psi_4\big)=
		\nu\pa_y^2w_3,\\
		&-\nu(\pa^2_y-\al^2)w_5+i\al\big((U-\la)w_5- U''\psi_5\big)=
		U''\al\delta_1w_3,\\
		&(\pa^2_y-\al^2)\psi_3=w_3,\quad
		(\pa^2_y-\al^2)\psi_4=w_4,\quad (\pa^2_y-\al^2)\psi_5=w_5,\\
		& \psi_3|_{y=\pm 1}=\psi_4|_{y=\pm 1}=\psi_5|_{y=\pm 1}=0.
	\end{aligned}\right.
\end{equation}
where $\delta_{1}=\delta^{\frac{4}{3}}\lrs{|\la-U(0)|^{\f12}+\delta}^{\frac{2}{3}}$ and $\delta=\nu^{\frac{1}{4}}|\al|^{-\frac{1}{4}}$.

For $(w_{3},\psi_{3})$, by Lemma \ref{lemma:ray,est,W,Phi}, we have
\begin{align*} \n{(\pa_y,|\al|)\psi_3}_{L^2}+\delta_1^{\f12}\n{w_3}_{L^2}+\delta_1^{\f32}\big(|U(0)-\la|+\delta_1\big)^{-\f12}\n{\pa_yw_3}_{L^2}\lesssim |\al|^{-1}\n{F}_{H^{1}}.
\end{align*}
Due to that $|U(0)-\la|+\delta_1\lesssim |U(0)-\la| +\delta^{2}$, we arrive at
\begin{align}\label{equ:est,varphi3}
	\n{(\pa_y,|\al|)\psi_3}_{L^2}+\delta^{\frac{2}{3}}\lrs{|U(0)-\la|^{\frac{1}{2}} +\delta}^{\f13}\n{w_3}_{L^2}+\delta^{2}\n{\pa_yw_3}_{L^2}
	\lesssim& |\al|^{-1}\n{F}_{H^{1}}.
\end{align}

For $(w_{4},\psi_{4})$, by Proposition \ref{lemma:re,w,full}, Proposition \ref{proposition:est,stronger,u,FL2,FH-1} and \eqref{equ:est,varphi3}, we obtain
\begin{align}\label{equ:est,varphi4}	&\nu^{\f12}|\al|^{\f12}\n{(\pa_y,|\al|)\psi_4}_{L^2}+\nu^{\f23}|\al|^{\f13}\lrs{|U(0)-\la|^{\frac{1}{2}}+\delta}^{\f13}\n{w_4}_{L^2}\\
	\lesssim& \nu\n{\pa_y^2w_3}_{H^{-1}}\leq \nu\n{\pa_yw_3}_{L^2}\lesssim \nu^{\f12}|\al|^{-\f12}\n{F}_{H^{1}}.\nonumber
\end{align}

For $(w_{5},\psi_{5})$, using Proposition \ref{lemma:re,w,full} and \eqref{equ:est,varphi3}, we also have
\begin{align*}
	&\nu^{\f13}|\al|^{\f23}\lrs{|U(0)-\la|^{\frac{1}{2}}+\delta}^{\f23}\n{w_5}_{L^2}
	\lesssim \n{U''\al\delta_1w_3}_{L^2}\lesssim
 \delta^{\frac{2}{3}}\big(|U(0)-\la|^{\frac{1}{2}}+\delta\big)^{\f13}\n{F}_{H^{1}}.
\end{align*}
Then by Proposition \ref{proposition:est,stronger,u,FL2,FH-1} and \eqref{equ:est,varphi3}, we get
\begin{align*}
	&\nu^{\f16}|\al|^{\f56}\lrs{|U(0)-\la|^{\frac{1}{2}}+\delta}^{\f13}\n{(\pa_y,|\al|)\psi_5}_{L^2}\\
	\lesssim& \n{U''\al\delta_1w_3}_{L^2}
	\lesssim\delta^{\frac{2}{3}}\lrs{|U(0)-\la|^{\frac{1}{2}}+\delta}^{\f13}\n{F}_{H^{1}},
\end{align*}
which gives
\begin{align}\label{equ:est,varphi5}
	\n{(\pa_y,|\al|)\psi_5}_{L^2}+\nu^{\f16}|\al|^{-\f16}\lrs{|U(0)-\la|^{\frac{1}{2}}+\delta}^{\f13}\n{w_5}_{L^2}
	\lesssim |\al|^{-1}\n{F}_{H^{1}}.
\end{align}

Recalling the decomposition $w=w_3+w_4+w_5$, we are able to deduce from \eqref{equ:est,varphi3}, \eqref{equ:est,varphi4} and \eqref{equ:est,varphi5} that
\begin{align*}
	\n{u}_{L^2}+\nu^{\frac{1}{4}}|\al|^{-\frac{1}{4}}\n{w}_{L^2}
	\lesssim |\al|^{-1}\n{F}_{H^{1}},
\end{align*}
 which together with \eqref{equ:est,wH1,w'uF,L2} implies
 \begin{align*}
 	\n{(\pa_y,|\al|) w}^2_{L^2}\lesssim&
 	\nu^{-1}|\al|\n{u}_{L^2}^2+\nu^{-1}|\al|^{-1}\n{F}_{H^{1}}^{2}\lesssim  \nu^{-1}|\al|^{-1}\n{F}^2_{H^{1}},\\
 \n{u}_{L^\infty}\lesssim& \n{u}_{L^2}^{\f12}\n{w}_{L^2}^{\f12}\lesssim \nu^{-\f18}|\al|^{-\f78}\n{F}_{H^{1}}.
 \end{align*}
 Hence, we complete the proof of \eqref{equ:reso,est,u,w,FH1}.
\end{proof}

\subsection{Summary of main results}\label{sec:Summary of Results under the Navier-slip Boundary Condition}
Now, we present the main results on the resolvent estimates under the Navier-slip boundary condition. We consider the following equation
\begin{equation}\label{equ:navierslip-NS,summary}
\left\{
\begin{aligned}
&-\nu(\partial^2_y-\alpha^2)w+i\alpha(U-\lambda)w-i\alpha U''\psi+o(\nu,\al)w=F,\\
&(\partial^2_y-\alpha^2)\psi=w,\quad \psi(\pm 1)=w(\pm 1)=0,
\end{aligned}\right.
\end{equation}
where $|o(\nu,\al)|\ll \nu^{\frac{1}{2}}|\al|^{\frac{1}{2}}$.
\begin{proposition}\label{lemma:summary,reso,navierslip}
Let $(w,\psi)$ be the solution to \eqref{equ:navierslip-NS,summary}. There exist positive constants $\nu_0$, $\varepsilon_0$, such that if $|o(\nu,\al)|\leq \ve_{0}\nu^{\frac{1}{2}}|\al|^{\frac{1}{2}}$, $\nu\in(0,\nu_{0}]$, the following resolvent estimates hold
\begin{align}
\nu^{\frac{1}{4}}|\al|^{\frac{3}{4}}\n{u}_{L^{2}}+\nu^{\f34}|\alpha|^{\f14}\n{(\pa_{y},|\al|)w}_{L^{2}}+\nu^{\f12}|\alpha|^{\f12}\n{w}_{L^2} \lesssim&\n{F}_{L^2},\label{equ:re,L2,w,full,sum}\\
\nu^{\f12}|\al|^{\f12}\n{u}_{L^2}+\nu\n{(\pa_{y},|\al|)w}_{L^2}+\nu^{\frac{3}{4}}|\al|^{\frac{1}{4}}\n{w}_{L^{2}}\lesssim &\n{F}_{H^{-1}},
\label{equ:re,H-1,w,full,sum}\\
 |\al|\n{u}_{L^2}+\nu^{\frac{1}{8}}|\al|^{\frac{7}{8}}\n{u}_{L^\infty}+
\nu^{\f12}|\al|^{\f12}\n{(\pa_y,|\al|)w}_{L^2}+\nu^{\f14}|\al|^{\f34}\n{w}_{L^2}\lesssim &\n{F}_{H^{1}}.\label{equ:re,H1,w,full,sum}
\end{align}
\end{proposition}
\begin{proof}
For the case where $o(\nu,\al)\equiv 0$, Proposition \ref{lemma:summary,reso,navierslip} serves as a summary of Propositions \ref{lemma:re,w,full}, \ref{proposition:est,stronger,u,FL2,FH-1}, and \ref{proposition:reso,est,u,w,FH1}. In the general case, the proposition is derived via a perturbation argument.
\end{proof}

\section{Boundary layer corrector}\label{sec:Boundary Layer Corrector}

Let $(w,\psi)$ be the solution to the Orr-Sommerfeld (OS) equation with the non-slip boundary condition
 \begin{equation}\label{equ:nonslip,wNa,NS,per}
	\left\{
	\begin{aligned}
		&-\nu(\pa^2_y-\al^2)w+i\al(U-\la)w-i\al U''\psi+o(\nu,\al)w=F,\\
		&(\pa^2_y-\al^2)\psi=w,\quad \psi(\pm 1)=\psi'(\pm 1)=0,
	\end{aligned}\right.
\end{equation}
where $|o(\nu,\al)|\ll \nu^{\frac{1}{2}}|\al|^{\frac{1}{2}}$. We then decompose it into
\begin{align*}
	w=w_{Na}+c_1w_{cor,1}+c_2w_{cor,2},
\end{align*}
where $w_{Na}$ is the solution to the OS equation with the Navier-slip boundary condition
 \begin{equation}\label{equ:navierslip,wNa,NS,per}
	\left\{
	\begin{aligned}
		&-\nu(\pa^2_y-\al^2)w_{Na}+i\al(U-\la)w_{Na}-i\al U''\psi_{Na}+o(\nu,\al)w_{Na}=F,\\
		&(\pa^2_y-\al^2)\psi_{Na}=w_{Na},\quad \psi_{Na}(\pm 1)=\psi''_{Na}(\pm 1)=w_{Na}(\pm 1)=0,
	\end{aligned}\right.
\end{equation}
and the boundary layer correctors $w_{cor,j}$ $(j=1,2)$ satisfy the following homogeneous OS equation
\begin{equation}\label{equ:corrector1,NS,per}
	\left\{
	\begin{aligned}
		&-\nu(\pa^2_y-\al^2)w_{cor,j}+i\al(U-\la)w_{cor,j}-i\alpha U''\psi_{cor,j}+o(\nu,\al)w_{cor,j}=0,\\
		&(\pa^2_y-\al^2)\psi_{cor,j}=w_{cor,j},\quad\psi_{cor,j}(\pm 1)=0,\\
        & \psi'_{cor,1}(1)=\psi'_{cor,2}(-1)=1,\quad\psi'_{cor,1}(-1)=\psi'_{cor,2}(1)=0.
	\end{aligned}\right.
\end{equation}

Due to the non-slip boundary condition that $
\lra{ e^{\pm |\al|y},w}=0,$
the coefficients are explicitly given by
\begin{equation}\label{equ:form,c1,c2}
\left\{
	\begin{aligned}
		&c_1(\la)=-\int_{-1}^1\frac{\operatorname{sinh}\lrs{|\al|(1+y)}}
{\operatorname{\operatorname{sinh}}(2|\al|)}w_{Na}(y)dy,\\
		&c_2(\la)=\int_{-1}^1\frac{\operatorname{\operatorname{sinh}}\lrs{|\al|(1-y)}}{\operatorname{\operatorname{sinh}}(2|\al|)}w_{Na}(y)dy.
	\end{aligned}\right.
\end{equation}

To provide the resolvent estimates with non-slip boundary condition, we are left to estimate the boundary layer correctors $w_{cor,1}$, $w_{cor,2}$ and the coefficients $c_1(\la)$, $c_2(\la)$.

\subsection{Estimates on $w_{cor,1}$ and $w_{cor,2}$}
We will construct the approximate solutions of  problems  \eqref{equ:corrector1,NS,per}  via the Airy function.
  Let $\mathrm{\bf Ai}(y)$ denote the Airy function, which is a solution to the equation $f''-yf=0$. We define the function
 \begin{align}\label{equ:def,A0}
 	A_{0}^{(1)}(z):=\int_{e^{\f{i\pi}{6}z}}^{+\infty}\mathrm{\bf Ai}(t)dt
 	=e^{\f{i\pi}{6}}\int_{z}^{+\infty}\mathrm{\bf Ai}(e^{\f{i\pi}{6}}t)dt,\quad A_{0}^{(2)}(z)=:A_{0}^{(1)}(-\ol{z}),
 \end{align}
and the quantities
\begin{align}
L_{1}&=|{\al U'(-1)}/{\nu}|^{\f13},\quad d_{1}=\f{U(-1)-\lambda-i\nu\al}{U'(-1)},\label{equ:def,L1-,d1-,d-1}\\
L_{2}&=|{\al U'(1)}/{\nu}|^{\f13},\qquad d_{2}=\f{U(1)-\lambda-i\nu\al}{U'(1)}\label{equ:def,L1+,d1+}.
\end{align}
Next, we define the approximate boundary correctors $W_{app,1}$ and $W_{app,2}$ as follows
\begin{align*}
	W_{app,1}(y)=\mathrm{\bf Ai}\big(e^{\f{i\pi}{6}}L_{1}(y+1+d_{1})
	\big),\quad W_{app,2}(y)=\mathrm{\bf Ai}\big(e^{\f{i5\pi}{6}}L_{2}(1-y+d_{2})\big),
\end{align*}
which solve the following problems respectively
\begin{equation*}
\left\{
\begin{aligned}
	-\nu(\pa^2_y-|\al|^2)W_{app,1}+i\al(U_{1}(y)-\la)W_{app,1}+o(\nu,\al)W_{app,1}&=0,\\	-\nu(\pa^2_y-|\al|^2)W_{app,2}+i\al(U_{2}(y)-\la)W_{app,2}+o(\nu,\al)W_{app,1}&=0,
\end{aligned}
\right.
\end{equation*}
where \begin{align*}
U_{1}(y)=U'(-1)(y+1)+U(-1),\quad U_{2}(y)=U'(1)(y-1)+U(1).
\end{align*}
Furthermore, the stream functions $\Psi_{app,i}$ are given by
\[
	(\pa^2_y-|\al|^2)\Psi_{app,j}=W_{app,j},\quad \Psi_{app,j}(\pm 1)=0, \quad \mathrm{for} \,\, j=1,2.
\]

We begin by presenting the following lemma, which gives effective estimates for the approximate boundary correctors.
\begin{lemma}[{\hspace{-0.01em}\cite[Lemma B.3]{CWZ-CMP}}]
	\label{lemma:est,Wapp12,LinftyL2}
There exist $c>0$, $C>0$ such that
	\begin{align*}
&|W_{app,i}(y)|\leq C(1+|L_{i}d_{i}|)^{\frac{1}{2}}|A_{0}^{(i)}(L_{i}d_{i})|e^{-cL_{i}(1+|L_{i}d_{i}|)^{\frac{1}{2}}(1-(-1)^{i}y)},\\
&\n{(1-(-1)^{i}y)^{m}W_{app,i}}_{L^{2}}\leq C(L_{i})^{-\frac{1+2m}{2}}(1+|L_{i}d_{i}|)^{\frac{1}{4}-\frac{m}{2}}|A_{0}^{(i)}(L_{i}d_{i})|,\quad m\geq 0,\\
&\n{W_{app,i}}_{L^{\infty}}\leq
C(1+|L_{i}d_{i}|)^{\frac{1}{2}}|A_{0}^{(i)}(L_{i}d_{i})|,\\
&\n{(1-(-1)^{i}y)^{m}W_{app,i}}_{L^{1}}\leq
C(L_{i})^{-1-m}(1+L_{i}d_{i})^{-\frac{m}{2}}|A_{0}^{(i)}(L_{i}d_{i})|,\\
	    & \n{(\pa_{y},|\al|)\Psi_{app,i}}_{L^{2}}\leq C(L_{i})^{-\f32}(1+|L_{i}d_{i}|)^{-\f14}|A_{0}^{(i)}(L_{i}d_{i})|.
	\end{align*}
\end{lemma}

To return to the Orr-Sommerfeld equation, it is necessary to incorporate the error terms that satisfy the following systems
  \begin{equation}\label{equ:navierslip,NS,w,error1,2}
 	\left\{
 	\begin{aligned}
 		&-\nu(\pa^{2}_{y}-\al^{2})w_{err,j}+i\al(U(y)-\la)w_{err,j}-i\al U''\psi_{err,j}+o(\nu,\al)w_{err,j}\\
 		&\qquad \,\,\,   =-i\al\big(U(y)-U_{j}(y)\big)W_{app,j}+i\al U''\Psi_{app,j},\quad  \mathrm{for}\,\, j=1,2,\\
 		&(\pa^2_y-\al^2)\psi_{err,j}=w_{err,j},\quad w_{err,j}(\pm 1)=\psi_{err,j}(\pm 1)=\psi''_{err,j}(\pm 1)=0,
 	\end{aligned}\right.
 \end{equation}
 Consequently, the solutions $w_{cor,1}$ and $w_{cor,2}$ can be expressed as
 \begin{equation}\label{equ:form,wcor1,wcor2}
	\left\{
	\begin{aligned}
		&w_{cor,1}=C_{11}\big(W_{app,1}+w_{err,1}\big)+C_{12}\big(W_{app,2}+w_{err,2}\big),\\
		&w_{cor,2}=C_{21}\big(W_{app,1}+w_{err,1}\big)+C_{22}\big(W_{app,2}+w_{err,2}\big),
	\end{aligned}\right.
\end{equation}
 where $C_{ij},i,j=1,2$ are constants determined by matching the boundary conditions. Specifically, using the boundary conditions
\begin{equation*}
\left\{
\begin{aligned} \int_{-1}^1w_{cor,j}(y)\frac{\operatorname{\operatorname{sinh}}\big(\al(1-y)\big)}{\operatorname{\operatorname{sinh}}(2\al)}dy&=-\pa_y\psi_{cor,j}(-1),\\ \int_{-1}^1w_{cor,j}(y)\frac{\operatorname{\operatorname{sinh}}\big(\al(1+y)\big)}{\operatorname{\operatorname{sinh}}(2\al)}dy&=\pa_y\psi_{cor,j}(1),
\end{aligned}
\right.
\end{equation*}
provided that $A_1A_2-B_1B_2\neq 0$, we obtain
 \begin{equation*}
 \left\{
 \begin{aligned}
 	C_{11}=\f{A_2}{A_1A_2-B_1B_2} \quad \mathrm{and}\quad
 	 C_{12}=\f{-B_1}{A_1A_2-B_1B_2},\\
 	C_{21}=\f{B_2}{A_1A_2-B_1B_2} \quad \mathrm{and}\quad
 	C_{22}=\f{-A_1}{A_1A_2-B_1B_2},
 \end{aligned}
 \right.
 \end{equation*}
 where
 \begin{equation}\label{equ:coefficients,A,B}
 \left\{
 \begin{aligned}
 	&A_1=\int_{-1}^1\frac{\operatorname{\operatorname{sinh}}\big(\al(1+y)\big)}{\operatorname{\operatorname{sinh}}(2\al)}\big(W_{app,1}+w_{err,1}\big)dy=:A_{app,1}+A_{err,1},\\
 	&A_2=\int_{-1}^1\frac{\operatorname{\operatorname{sinh}}\big(\al(1-y)\big)}{\operatorname{\operatorname{sinh}}(2\al)}\big(W_{app,2}+w_{err,2}\big)dy=:A_{app,2}+A_{err,2},\\
 	&B_1=\int_{-1}^1\frac{\operatorname{\operatorname{sinh}}\big(\al(1-y)\big)}{\operatorname{\operatorname{sinh}}(2\al)}\big(W_{app,1}+w_{err,1}\big)dy=:B_{app,1}+B_{err,1},\\
 	&B_2=\int_{-1}^1\frac{\operatorname{\operatorname{sinh}}\big(\al(1+y)\big)}{\operatorname{\operatorname{sinh}}(2\al)}\big(W_{app,2}+w_{err,2}\big)dy=:B_{app,2}+B_{err,2}.
 \end{aligned}
 \right.
 \end{equation}

Next, we provide estimates for for $A_{app,i}$ and $B_{app,i}\,\, (i=1,2).$
   \begin{lemma}[Estimates on approximation coefficients]\label{lemma:est,Aapp12,Bapp12}
   	     If $\min(L_{1},L_{2})\gg 1$,  we have the upper bounds
   	     \begin{align*}
   	     	&\abs{A_{app,1}}\lesssim (L_{1})^{-2}	\abs{A_{0}^{(1)}(L_{1}d_{1})},\quad
   	     	\abs{A_{app,2}}\lesssim (L_{2})^{-2}\abs{A_{0}^{(2)}\big(L_{2}d_{2}\big)},
   	     \end{align*}
   and the lower bounds
   \begin{align*}
   &\abs{B_{app,1}}\gtrsim (L_{1})^{-1}	\abs{A_{0}^{(1)}(L_{1}d_{1})},\quad
   	     	\abs{B_{app,2}}\gtrsim (L_{2})^{-1}\abs{A_{0}^{(2)}\big(L_{2}d_{2}\big)},
   \end{align*}
   where $A_{0}^{(2)}(z)=A_{0}^{(1)}(-\ol{z})$.
   \end{lemma}
   \begin{proof}
The proof follows arguments analogous to those in \cite[Lemma 5.7]{cwz-mams2024}, which we omit here for conciseness.
   \end{proof}

The following lemma provides estimates for the error terms.
 \begin{lemma}[Estimates on error terms]\label{lemma:est,uw,error12,L2L1Linfty}
 	Let $(w_{err,1},w_{err,2})$ be the solutions to \eqref{equ:navierslip,NS,w,error1,2}. Then the following estimates hold
 	\begin{align}\label{equ:error term,estimates,1}
 		&\n{w_{err,i}}_{L^2}+	(L_{i})^{-\frac{3}{8}}\n{w_{err,i}}_{L^\infty} +(L_{i})^{-\frac{3}{4}}\n{(\pa_y,|\al|)w_{err,i}}_{L^2}\\
 		&+(L_{i})^{\frac{3}{4}}\n{u_{err,i}}_{L^2}+(L_{i})^{\frac{3}{8}}\n{u_{err,i}}_{L^\infty}
 		\lesssim  (L_{i})^{-\f34}\big(1+|L_{i}d_{i}|\big)^{-\f14}\babs{A_{0}^{(i)}(L_{i}d_{i})}.\notag
 	\end{align}

 \end{lemma}
    \begin{proof}
    We only provide the proof of \eqref{equ:error term,estimates,1} for $w_{err,1}$, as the argument for $w_{err,2}$ is
similar.

    First, we decompose the error term as $w_{err,1}=w_{err,1}^1+w_{err,1}^2$, where $w_{err,1}^1$ and $w_{err,1}^2$ solve the following system
    	   \begin{equation}\label{equ:navierslip,NS,w,error1,decom12}
    		\left\{
    		\begin{aligned}
    			&-\nu(\pa^2_y-\al^2)w_{err,1}^1+i\al(U-\la)w_{err,1}^1-i\al U''\psi_{err,1}^1+o(\nu,\al)w_{err,1}^{1}
    			=-i\al\big(U-U_{1}\big)W_{app,1},\\
    			&-\nu(\pa^2_y-\al^2)w_{err,1}^2+i\al(U-\la)w_{err,1}^2-i\al U''\psi_{err,1}^2+o(\nu,\al)w_{err,1}^{2} =i\al U''\Psi_{app,1},\\
    			&(\pa^2_y-\al^2)\psi_{err,1}^i=w_{err,1}^i,\quad w_{err,1}^i(\pm 1)=\psi_{err,1}^i(\pm 1)=(\psi^{i}_{err,1})^{''}(\pm 1)=0,i=1,2.
    		\end{aligned}\right.
    	\end{equation}	
    	Notice that $|U(y)-U_{1}(y)|=\abs{U(y)-U(-1)-U'(-1)(y+1)}\lesssim (1+y)^2$, we deduce from Lemma \ref{lemma:est,Wapp12,LinftyL2}	that
    	\begin{align}\label{equ:est,UWapp1,Psiapp1,L2}
    	\n{(U(y)-U_{1}(y))W_{app,1}}_{L^2}
    	\lesssim&
    		(L_{1})^{-\f52}\big(1+|L_{1}d_{1}|\big)^{-\f34}\abs{A_{0}^{(1)}(L_{1}d_{1})}.
    	\end{align}
 By applying resolvent estimate \eqref{equ:re,L2,w,full,sum} in Proposition \ref{lemma:summary,reso,navierslip} and the above estimate \eqref{equ:est,UWapp1,Psiapp1,L2},
    	we obtain
    	\begin{align}\label{equ:est,wuw',err11,L2L1}
 &\nu^{\f12}|\al|^{\f12}\n{w_{err,1}^1}_{L^2}+\nu^{\f34}|\al|^{\f14}\n{(\pa_y,|\al|)w_{err,1}^1}_{L^2}+\nu^{\f14}|\al|^{\f34}\n{u_{err,1}^1}_{L^2}\\
    	\lesssim& \n{i\al(U(y)-U_{1}(y))W_{app,1}}_{L^2}\nonumber\\
    	\lesssim& (L_{1})^{-\f52}|\al|\big(1+|L_{1}d_{1}|\big)^{-\f34}\abs{A_{0}^{(1)}(L_{1}d_{1})}\notag\\
    	\lesssim& (L_{1})^{-1}\nu^{\f12}|\al|^{\f12}\big(1+|L_{1}d_{1}|\big)^{-\f34}\abs{A_{0}^{(1)}(L_{1}d_{1})},\notag
    	\end{align}
    where we have used the fact that $\lrs{L_{1}}^{-\frac{3}{2}}\simeq \nu^{\f12} |\al|^{-\f12}$.
    Similarly, using \eqref{equ:re,H1,w,full,sum} in Proposition \ref{lemma:summary,reso,navierslip} and Lemma \ref{lemma:est,Wapp12,LinftyL2}, we obtain
    	\begin{align}\label{equ:est,wuw',err12,L2L1}
 &\nu^{\f14}|\al|^{\f34}\n{w_{err,1}^2}_{L^2}+\nu^{\f12}|\al|^{\f12}\n{(\pa_y,|\al|)w_{err,1}^2}_{L^2}+|\al|\n{u_{err,1}^2}_{L^2}\\
    	\lesssim& \n{i\al U''\Psi_{app,1}}_{H^{1}}\lesssim |\al|\n{(|\al|,\pa_{y})\Psi_{app,1}}_{L^2}\nonumber\\
    	\lesssim& (L_{1})^{-\f32}|\al|\big(1+|L_{1}d_{1}|\big)^{-\f14}\abs{A_{0}^{(1)}(L_{1}d_{1})}\nonumber\\
    	\lesssim&  (L_{1})^{-\f34}\nu^{\f14}|\al|^{\f34}\big(1+|L_{1}d_{1}|\big)^{-\f14}\abs{A_{0}^{(1)}(L_{1}d_{1})}.\notag
    \end{align}
Combining \eqref{equ:est,wuw',err11,L2L1} and  \eqref{equ:est,wuw',err12,L2L1}, we arrive at
    	\begin{align}\label{equ:est,wuw',err1,L2}
    	&\n{w_{err,1}}_{L^2}+\nu^{\f14}|\al|^{-\f14}\n{(\pa_y,|\al|)w_{err,1}}_{L^2}+\nu^{-\f14}|\al|^{\f14}\n{u_{err,1}}_{L^2}\\
    	\lesssim&  \Big((L_{1})^{-1}\big(1+|L_{1}d_{1}|\big)^{-\f34}+(L_{1})^{-\f34}\big(1+|L_{1}d_{1}|\big)^{-\f14}\Big)\abs{A_{0}^{(1)}(L_{1}d_{1})}\nonumber\\
    	\lesssim &
    	(L_{1})^{-\f34}\big(1+|L_{1}d_{1}|\big)^{-\f14}\abs{A_{0}^{(1)}(L_{1}d_{1})}.\notag
    \end{align}
   For the $L^{\infty}$ estimate, we use the Gagliardo-Nirenberg inequality that
    	\begin{align*}
    	\n{w_{err,1}}_{L^\infty}\lesssim &
    	\n{w_{err,1}}^{\f12}_{L^2}\n{\pa_yw_{err,1}}^{\f12}_{L^2},\quad
    \n{u_{err,1}}_{L^\infty}\lesssim
    	\n{w_{err,1}}^{\f12}_{L^2}\n{u_{err,1}}^{\f12}_{L^2}.
    	\end{align*}
    Therefore, we complete the proof of \eqref{equ:error term,estimates,1}.
    	\end{proof}

          Next, we derive estimates for the coefficients $A_i, B_i$, which in turn yield bounds for $C_{ij}$.

         \begin{lemma}[Estimates on coefficients $A_{i}$, $B_{i}$, $C_{ij}$]\label{lemma:est,bound,AiBi,Cij,ij12}
         	If $\min(L_{1},L_{2})\gg 1$, then the following estimates hold
          \begin{align*}
             	&\abs{A_{1}}\lesssim (L_{1})^{-\f98}	\abs{A_{0}^{(1)}(L_{1}d_{1})},\quad
         	 \abs{A_{2}}\lesssim  (L_{2})^{-\f98}\abs{A_{0}^{(2)}(L_{2}d_{2})},\\
         	&\abs{B_{1}}\gtrsim  (L_{1})^{-1}	\abs{A_{0}^{(1)}(L_{1}d_{1})},\quad
         	\abs{B_{2}}\gtrsim (L_{2})^{-1}\abs{A_{0}^{(2)}(L_{2}d_{2})}.
         \end{align*}
         In particular, we have
         \begin{align*}
         	&\abs{C_{11}}\lesssim \f{(L_{1})^{\f78}}{\babs{A_{0}^{(1)}(L_{1}d_{1})}}, \qquad \abs{C_{12}}\lesssim \f{L_{2}}{\babs{A_{0}^{(2)}\big(L_{2}d_{2}\big)}},\\
         	& \abs{C_{21}}\lesssim \f{L_{1}}{\babs{A_{0}^{(1)}(L_{1}d_{1})}},\qquad \abs{C_{22}}\lesssim \f{(L_{2})^{\f78}}{\babs{A_{0}^{(2)}\big(L_{2}d_{2}\big)}}.
         \end{align*}
         \end{lemma}
         \begin{proof}
         We only consider estimates on $A_{1}$ and $B_{1}$, as the proof of estimates on $A_{2}$ and $B_{2}$ is similar.
         	Recalling the definition \eqref{equ:coefficients,A,B} of $A_{err,1}$, using \eqref{equ:error term,estimates,1} in Lemma \ref{lemma:est,uw,error12,L2L1Linfty}, we get
         	\begin{align}\label{equ:est,Aerr1,abs}
         		\abs{A_{err,1}}&=\bbabs{\int_{-1}^1\frac{\operatorname{\operatorname{sinh}}\big(\alpha(1+y)\big)}
         {\operatorname{\operatorname{sinh}}(2\alpha)}w_{err,1}dy}
         		\lesssim \n{\pa_y\psi_{err,1}(1)}_{L^\infty}\\
         		&\lesssim\nu^{\f18}|\al|^{-\f18}(L_{1})^{-\f{3}{4}} \babs{A_{0}^{(1)}(L_{1}d_{1})}
         		\lesssim  (L_{1})^{-\f{9}{8}}\babs{A_{0}^{(1)}(L_{1}d_{1})},\notag
         	\end{align}
         	where we used the fact $L_{1}\sim \nu^{-\frac{1}{3}}|\al|^{\f13}$.
         	Similarly, we derive
         		\begin{align}\label{equ:est,Berr1,abs}
         		\abs{B_{err,1}}\lesssim (L_{1})^{-\f{9}{8}}\abs{A_{0}^{(1)}(L_{1}d_{1})}.
         	\end{align}

         For $A_1$, using Lemma \ref{lemma:est,Aapp12,Bapp12} and \eqref{equ:est,Aerr1,abs}, we get
         	\begin{align*}
         		\abs{A_1}&\leq \abs{A_{err,1}}+\abs{A_{app,1}}
         		\lesssim (L_{1})^{-\f{9}{8}}\abs{A_{0}^{(1)}(L_{1}d_{1})}.
         	\end{align*}

         For $B_{1}$, using \eqref{equ:est,Berr1,abs} and lower bound in Lemma \ref{lemma:est,Aapp12,Bapp12}, we have
         	\begin{align*}
         |B_{1}|\geq& |B_{app,1}|-|B_{err,1}|\gtrsim (L_{1})^{-1}	\abs{A_{0}^{(1)}(L_{1}d_{1})}.
         	\end{align*}

We arrive at
\begin{align*}
\lrv{\f{A_{1}}{B_{1}}}\lesssim (L_{1})^{-\frac{1}{8}}\leq \frac{1}{2}, \quad \lrv{\f{A_{2}}{B_{2}}}\lesssim (L_{2})^{-\frac{1}{8}}\leq \frac{1}{2},
\end{align*}
and hence obtain
         	\begin{align*}
         		\abs{A_1A_2-B_1B_2}\geq \abs{B_1B_2}\bbabs{1-\f{A_1A_2}{B_1B_2}}\geq \f12\abs{B_1B_2} >0.
         	\end{align*}

         	Finally, we bound $C_{ij}$ for $i,j=1,2.$ Specifically, with $L_{1}\simeq L_{2}$, we have
         	\begin{align*}
         		|C_{22}|&=\bbabs{\f{A_1}{A_1A_2-B_1B_2}}\leq
         		\f{(L_{1})^{-\f{9}{8}}\babs{A_{0}^{(1)}(L_{1}d_{1})}}{(L_{1})^{-1}\abs{A_{0}^{(1)}(L_{1}d_{1})}(L_{2})^{-1}\abs{A_{0}^{(2)}(L_{2}d_{2})}}\\
         		&\lesssim \f{(L_{1})^{-\f18}L_{2}}{\abs{A_{0}^{(2)}(L_{2}d_{2})}}\lesssim \f{(L_{2})^{\f78}}{\abs{A_{0}^{(2)}\big(L_{2}d_{2}\big)}
         }.
         	\end{align*}
         	 The other estimates are similar and are omitted for brevity. Thus, we complete the proof.
         	 \end{proof}

         We are now in a position to provide bounds for $w_{cor,1}$ and $w_{cor,2}$.

 \begin{lemma}\label{lemma:est,uw,L2Linfty,cor12}
 	 If $\min(L_{1},L_{2})\gg 1$, then for $j=1,2$, the following estimates hold
  \begin{align}
   \n{w_{cor,j}}_{L^2}\lesssim& (L_{j})^{\frac{1}{2}}\lrs{1+|L_{j}d_{j}|}^{\f14},\label{equ:w,cor,L2}\\
     \n{w_{cor,j}}_{L^{\infty}}\lesssim&L_{j}\lrs{1+|L_{j}d_{j}|}^{\f12},\label{equ:w,cor,Lwq}\\
       \n{w_{cor,j}}_{L^{1}}\lesssim&1+ \lrs{L_{j}}^{\frac{1}{4}}(1+|L_{j}d_{j}|)^{-\frac{1}{4}},\label{equ:w,cor,L1}\\
  	\n{u_{cor,j}}_{L^2} \lesssim&(L_{j})^{-\frac{1}{2}}\lrs{1+|L_{j}d_{j}|}^{-\f14},\label{equ:u,cor,L2}\\
  \n{u_{cor,j}}_{L^\infty}\lesssim& 1.\label{equ:u,cor,Lwq}
  	\end{align}

 \end{lemma}
      \begin{proof}
      We only provide estimates for $w_{err,1}$, as estimates for $w_{err,2}$ follow similarly.
      Recalling that
      	\begin{align*}
      		w_{cor,1}=C_{11}(W_{app,1}+w_{err,1})+C_{12}(W_{app,2}+w_{err,2}),
      	\end{align*}
 Using Lemmas \ref{lemma:est,bound,AiBi,Cij,ij12}, \ref{lemma:est,Wapp12,LinftyL2}, and \ref{lemma:est,uw,error12,L2L1Linfty}, with $1+|L_{1}d_{1}|\simeq 1+|L_{2}d_{2}|$ and $L_{j}\sim (\nu/|\al|)^{\frac{1}{3}}$, we obtain
 estimates as follows:
      	\begin{align*}
      		\n{w_{cor,1}}_{L^2}\leq& \abs{C_{11}}\big(\n{W_{app,1}}_{L^2}+\n{w_{err,1}}_{L^2}\big)+\abs{C_{12}}\big(\n{W_{app,2}}_{L^2}+\n{w_{err,2}}_{L^2}\big)	\\
      		\lesssim& 	(L_{1})^{\frac{7}{8}}\lrs{(L_{1})^{-\frac{1}{2}}\big(1+|L_{1}d_{1}|\big)^{\f14}+(L_{1})^{-\f34}\big(1+|L_{1}d_{1}|\big)^{-\f14}}\\
      &+L_{2}\lrs{(L_{2})^{-\frac{1}{2}}\big(1+|L_{2}d_{2}|\big)^{\f14}+(L_{2})^{-\f34}\big(1+|L_{2}d_{2}|\big)^{-\f14}}\\
      		\lesssim &
      		\nu^{-\f16}\abs{\al}^{\f16}\big(1+|L_{2}d_{2}|\big)^{\f14},
      	\end{align*}
    	\begin{align*}
      		\n{w_{cor,1}}_{L^{\infty}}\leq& \abs{C_{11}}\big(\n{W_{app,1}}_{L^{\infty}}+\n{w_{err,1}}_{L^{\infty}}\big)+\abs{C_{12}}\big(\n{W_{app,2}}_{L^{\infty}}+\n{w_{err,2}}_{L^{\infty}}\big)	\\
      		\lesssim& 	(L_{1})^{\frac{7}{8}}\lrs{(\big(1+|L_{1}d_{1}|\big)^{\f12}+(L_{1})^{-\f38}\big(1+|L_{1}d_{1}|\big)^{-\f14}}\\
      &+L_{2}\lrs{\big(1+|L_{2}d_{2}|\big)^{\f12}+(L_{2})^{-\f38}\big(1+|L_{2}d_{2}|\big)^{-\f14}}\\
      		\lesssim &
      		\nu^{-\f13}\abs{\al}^{\f13}\big(1+|L_{2}d_{2}|\big)^{\f12},
      	\end{align*}
        	\begin{align*}
      		\n{w_{cor,1}}_{L^{1}}\leq& \abs{C_{11}}\big(\n{W_{app,1}}_{L^{1}}+\n{w_{err,1}}_{L^{1}}\big)+\abs{C_{12}}\big(\n{W_{app,2}}_{L^{1}}+\n{w_{err,2}}_{L^{1}}\big)	\\
      		\lesssim& 	(L_{1})^{\frac{7}{8}}\lrs{(L_{1})^{-1}+(L_{1})^{-\f34}\big(1+|L_{1}d_{1}|\big)^{-\f14}}\\
      &+L_{2}\lrs{(L_{2})^{-1}+(L_{2})^{-\f34}\big(1+|L_{2}d_{2}|\big)^{-\f14}}\\
      		\lesssim &1+ \lrs{L_{j}}^{\frac{1}{4}}(1+|L_{j}d_{j}|)^{-\frac{1}{4}},
      	\end{align*}
      	\begin{align*}
      	\n{u_{cor,1}}_{L^2}\leq& \abs{C_{11}}\big(\n{U_{app,1}}_{L^2}+\n{u_{err,1}}_{L^2}\big)+\abs{C_{12}}\big(\n{U_{app,2}}_{L^2}+\n{u_{err,2}}_{L^2}\big)	\\
      \lesssim&(L_{1})^{\frac{7}{8}}\lrs{(L_{1})^{-\frac{3}{2}}\big(1+|L_{1}d_{1}|\big)^{-\f14}
      +\nu^{\frac{1}{4}}|\al|^{-\frac{1}{4}}(L_{1})^{-\f34}\big(1+|L_{1}d_{1}|\big)^{-\f14}}\\
      &+L_{2}\lrs{(L_{2})^{-\frac{3}{2}}\big(1+|L_{2}d_{2}|\big)^{-\f14}
      +\nu^{\frac{1}{4}}|\al|^{-\frac{1}{4}}(L_{2})^{-\f34}\big(1+|L_{2}d_{2}|\big)^{-\f14}}\\
      	\lesssim&
      	\nu^{\f16}\abs{\al}^{-\f16}\big(1+|L_{2}d_{2}|\big)^{-\f14},
      	\end{align*}
      	\begin{align*}
      		\n{u_{cor,1}}_{L^\infty}\lesssim \n{w_{cor,1}}_{L^2}^{\f12}\n{u_{cor,1}}_{L^2}^{\f12}\lesssim 1.
      	\end{align*}
      Hence, we complete the proof.
 \end{proof}

\subsection{Estimates on coefficients $c_{1}$ and $c_{2}$}
In the section, we get into the analysis of bounds on the coefficients $c_{1}$ and $c_{2}$, which are given by \eqref{equ:form,c1,c2}.
\begin{lemma}\label{lemma:bounds on c1,c2}
Under the restriction that
\begin{align}\label{equ:restriction,bounds on c12}
\nu |\al|^2 \leq |\la- U(0)|^{\frac{1}{2}} + |\nu|^{\frac{1}{4}}|\al|^{-\f14},
\end{align}
for $F\in L^{2},H^{-1},H^{1}$, we have
\begin{align}\label{equ:bounds on c1,c2,L2}
|c_{1}|+|c_{2}|\lesssim& \nu^{-\frac{3}{8}}|\al|^{-\frac{7}{8}}(1+|\al||\la-U(1)|)^{-\frac{1}{4}}\n{F}_{L^{2}},\\
|c_{1}|+|c_{2}|\lesssim& \nu^{-\frac{1}{2}}|\al|^{-\frac{1}{2}}(|\la-U(1)|+\nu^{\frac{1}{3}}|\al|^{-\frac{1}{3}})^{-\frac{1}{4}}\n{F}_{H^{-1}},\label{equ:bounds on c1,c2,H-1}\\
|c_{1}|+|c_{2}|\lesssim& \nu^{-\frac{1}{8}}|\al|^{-\frac{7}{8}}(1+|\la-U(1)|)^{-\frac{1}{4}}\n{F}_{H^{1}}.\label{equ:bounds on c1,c2,H1}
\end{align}
\end{lemma}
\begin{proof}

For simplicity, we might assume that $o(\nu,\al)\equiv 0$, as it is a perturbation term.
It suffices to estimate the coefficient $c_{1}$, as the bounds on $c_{2}$ follows similarly.

First, we deal with \textbf{the case $F\in L^{2}$}.
We divide the proof into the following cases.\smallskip

{\it Case 1.1.} $U(0)-\la\gtrsim \frac{1}{|\al|}.$

 By energy estimate \eqref{equ:L2,w,resolvent,0,full,proof} in Lemma \ref{lemma:energy estimate,real part}, we have
\begin{align}
|\al|\lrs{|\la-U(0)|\n{w}_{L^{2}}^{2}+\blra{(U(y)-U(0))w,\frac{w}{U''}}+\n{(\pa_{y},|\al|)\psi}_{L^{2}}^{2}}
\lesssim&\babs{\blra{F,\frac{w}{U''}}},\notag
\end{align}
which, together with resolvent estimate \eqref{equ:re,L2,w,full}, yields that
\begin{align}\label{equ:L2,w,weighted,proof}
\lra{(|\la-U(0)|+y^{2})w,w}
\lesssim\nu^{-\frac{1}{3}}|\al|^{-\frac{5}{3}}\lrs{|\la-U(0)|^{\frac{1}{2}}
+\nu^{\frac{1}{4}}|\al|^{-\frac{1}{4}}}^{-\frac{2}{3}}\n{F}_{L^{2}}^{2}.
\end{align}
Hence, we get
\begin{align*}
|c_{1}|\lesssim& \bbn{\frac{\operatorname{\operatorname{sinh}}\lrs{|\al|(1+y)}}{\operatorname{\operatorname{sinh}}(2|\al|)
}}_{L^{2}}\bbn{\frac{1}{\sqrt{|\la-U(0)|+y^{2}}}}_{L_{y}^{\infty}}\bn{\sqrt{|\la-U(0)|+y^{2}}w_{Na}}_{L^{2}}\\
\lesssim& |\al|^{-\frac{1}{2}}|\la-U(0)|^{-\frac{1}{2}}\bn{\sqrt{|\la-U(0)|+y^{2}}w_{Na}}_{L^{2}}\\
\lesssim& |\al|^{-\frac{1}{2}}|\la-U(0)|^{-\frac{1}{2}}
\nu^{-\frac{1}{6}}|\al|^{-\frac{5}{6}}\lrs{|\la-U(0)|^{\frac{1}{2}}+\nu^{\frac{1}{4}}|\al|^{-\frac{1}{4}}}^{-\frac{1}{3}}\n{F}_{L^{2}}.
\end{align*}
In this case, noting that $|\al||\la-U(0)|\gtrsim \frac{1+ |\al||\la-U(1)|}{|\al|}+1$, we discard $(|\al||\la-U(0)|)^{-\frac{1}{4}}$ and use \eqref{equ:restriction,bounds on c12} to obtain
\begin{align*}
|c_{1}|\lesssim&(|\al||\la-U(0)|)^{-\frac{1}{4}}
\nu^{-\frac{1}{6}}|\al|^{-\frac{5}{6}}\lrs{|\la-U(0)|^{\frac{1}{2}}+\nu^{\frac{1}{4}}
|\al|^{-\frac{1}{4}}}^{-\frac{1}{6}}(\nu|\al|^{2})^{-\frac{1}{6}}\n{F}_{L^{2}}\\
\lesssim& (1+|\al||\la-U(1)|)^{-\frac{1}{4}}|\al|^{\frac{1}{4}}\nu^{-\frac{1}{6}}|\al|^{-\frac{5}{6}}\lrs{\nu^{\frac{1}{4}}
|\al|^{-\frac{1}{4}}}^{-\frac{1}{6}}(\nu|\al|^{2})^{-\frac{1}{6}}\n{F}_{L^{2}}\\
=&\nu^{-\frac{3}{8}}|\al|^{-\frac{5}{8}-\frac{1}{4}}(1+|\al||\la-U(1)|)^{-\frac{1}{4}}\n{F}_{L^{2}},
\end{align*}
which is the desired estimate \eqref{equ:bounds on c1,c2,L2}.\smallskip

{\it Case 1.2.} $\frac{1}{|\al|}\gtrsim U(0)-\la\geq-\nu^{\frac{1}{2}}|\al|^{-\frac{1}{2}} .$

In this case, we have
\begin{align}\label{equ:bounds on c1,two}
|c_1|\leq& \int_{-1}^{1-\ve}\frac{\operatorname{sinh}\big(|\al|(1+y)\big)}
{\operatorname{\operatorname{sinh}}(2|\al|)}w_{Na}(y)dy+\int_{1-\ve}^{1}\frac{\operatorname{sinh}\big(|\al|(1+y)\big)}
{\operatorname{\operatorname{sinh}}(2|\al|)}w_{Na}(y)dy\\
\lesssim &e^{-2|\al|\ve}\n{w_{Na}}_{L^{1}}+\ve^{\frac{1}{2}}\bn{\sqrt{|\la-U(0)|+y^{2}}w_{Na}}_{L^{2}}.\notag
\end{align}
By resolvent estimate \eqref{equ:re,L2,w,main} and \eqref{equ:L2,w,weighted,proof}, with $\ve= \frac{|\al|^{-\frac{1}{2}}}{2}$, we get
\begin{align*}
|c_{1}|\lesssim& e^{-|\al|^{\frac{1}{2}}}\nu^{-\f38}|\alpha|^{-\f58}\n{F}_{L^{2}}+|\al|^{-\frac{1}{4}}
\nu^{-\frac{1}{6}}|\al|^{-\frac{5}{6}}\lrs{|\la-U(0)|^{\frac{1}{2}}
+\nu^{\frac{1}{4}}|\al|^{-\frac{1}{4}}}^{-\frac{1}{3}}\n{F}_{L^{2}}\\
\lesssim& e^{-|\al|^{\frac{1}{2}}}\nu^{-\f38}|\alpha|^{-\f58}\n{F}_{L^{2}}
+\nu^{-\frac{1}{4}}|\al|^{-1}\n{F}_{L^{2}}
\lesssim|\al|^{-\frac{1}{4}}\nu^{-\f38}|\alpha|^{-\f78}\n{F}_{L^{2}},
\end{align*}
where we have used that $e^{-|\al|^{\frac{1}{2}}}\lesssim |\al|^{-\frac{1}{2}}$.
Noting that $|\al|\gtrsim 1+|\al||\la-U(1)|$ for this case, we arrive at the desired estimate \eqref{equ:bounds on c1,c2,L2}.\smallskip

{\it Case 2.} $\lambda\geq U(1).$

By resolvent estimate \eqref{equ:re,L2,w,full}, we have
\begin{align*}
|c_{1}|\lesssim& \bbn{\frac{\operatorname{\operatorname{sinh}}\lrs{|\al|(1+y)}}{\operatorname{\operatorname{sinh}}(2|\al|)}}_{L^{2}}\n{w_{Na}}_{L^{2}}\\
\lesssim& |\al|^{-\frac{1}{2}}
\nu^{-\frac{1}{3}}|\al|^{-\frac{2}{3}}\lrs{|\la-U(0)|^{\frac{1}{2}}+\nu^{\frac{1}{4}}|\al|^{-\frac{1}{4}}}^{-\frac{2}{3}}\n{F}_{L^{2}}\\
\lesssim&|\al|^{-\frac{1}{2}}
\nu^{-\frac{1}{3}}|\al|^{-\frac{2}{3}}\lrs{|\la-U(0)|^{\frac{1}{2}}+\nu^{\frac{1}{4}}|\al|^{-\frac{1}{4}}}^{-\frac{5}{8}}(\nu |\al|^{2})^{-\frac{1}{24}}\n{F}_{L^{2}}\\
 =&\nu^{-\frac{3}{8}}|\al|^{-\frac{10}{8}} \lrs{|\la-U(0)|^{\frac{1}{2}}+\nu^{\frac{1}{4}}|\al|^{-\frac{1}{4}}}^{-\frac{5}{8}}\n{F}_{L^{2}}.
\end{align*}
In this case, noting that $|\la-U(0)|^{\frac{1}{2}}+\nu^{\frac{1}{4}}|\al|^{-\frac{1}{4}}\geq 1+|\al||\la-U(1)|^{\frac{1}{2}}$, we get
\begin{align*}
|c_{1}|\lesssim& \nu^{-\frac{3}{8}}|\al|^{-\frac{10}{8}} \lrs{|\la-U(0)|^{\frac{1}{2}}+\nu^{\frac{1}{4}}|\al|^{-\frac{1}{4}}}^{-\frac{1}{2}}\n{F}_{L^{2}}\\
\lesssim&\nu^{-\frac{3}{8}}|\al|^{-1} \lrs{1+|\al||\la-U(1)|}^{-\frac{1}{4}}\n{F}_{L^{2}},
\end{align*}
which suffices for our goal.\smallskip

{\it Case 3.} $\lambda \in (U(0)+\nu^{\frac{1}{2}}|\al|^{-\frac{1}{2}},U(1))$.

Following the proof of \eqref{equ:L2,w,resolvent,0,full,proof}, we actually have
\begin{align}\label{equ:L2,w,resolvent,0,la}
\blra{(U(y)-U(0))w_{Na},\frac{w_{Na}}{U''}}
\lesssim&|\al|^{-1}\babs{\blra{F,\frac{w_{Na}}{U''}}}+|\la-U(0)|\n{w_{Na}}_{L^{2}}^{2},
\end{align}
which, together with resolvent estimate \eqref{equ:re,L2,w,full}, yields that
\begin{align}\label{equ:L2,w,weighted,y,proof}
\blra{y^{2}w_{Na},w_{Na}}
\lesssim&|\al|^{-1}\n{F}_{L^{2}}\n{w_{Na}}_{L^{2}}+|\la-U(0)|\n{w_{Na}}_{L^{2}}^{2}\\
\lesssim& \nu^{-\frac{1}{3}}|\al|^{-\frac{5}{3}}|\la-U(0)|^{-\frac{1}{3}}\n{F}_{L^{2}}^{2}+
\nu^{-\frac{2}{3}}|\al|^{-\frac{4}{3}}|\la-U(0)|^{\frac{1}{3}}\n{F}_{L^{2}}^{2}\notag\\
\lesssim &\nu^{-\frac{2}{3}}|\al|^{-\frac{4}{3}}|\la-U(0)|^{\frac{1}{3}}\n{F}_{L^{2}}^{2},\notag
\end{align}
where in the last inequality we have used that $|\la-U(0)|^{-\frac{2}{3}}\lesssim \nu^{-\frac{1}{3}}|\al|^{-\frac{1}{3}}$ for this case. Then in the same way as \eqref{equ:bounds on c1,two}, we have
\begin{align*}
|c_{1}|\lesssim &e^{-2|\al|\ve}\n{w_{Na}}_{L^{1}}+\ve^{\frac{1}{2}}\bn{yw_{Na}}_{L^{2}}.
\end{align*}
By resolvent estimate \eqref{equ:re,L2,w,main} and \eqref{equ:L2,w,weighted,y,proof}, with $\ve=\frac{|\al|^{-\frac{3}{4}}}{2}$, we obtain
\begin{align*}
|c_{1}|\lesssim&e^{-|\al|^{\frac{3}{4}}}\nu^{-\f38}|\alpha|^{-\f58}\n{F}_{L^{2}}+
|\al|^{-\frac{3}{8}}\nu^{-\frac{1}{3}}|\al|^{-\frac{2}{3}}|\la-U(0)|^{\frac{1}{6}}\n{F}_{L^{2}}\\
\lesssim&|\al|^{-\frac{1}{4}}\nu^{-\f38}|\alpha|^{-\f78}\n{F}_{L^{2}}+\nu^{-\frac{1}{4}}|\al|^{-\frac{11}{8}}\n{F}_{L^{2}}\\
\lesssim&\nu^{-\frac{3}{8}}|\al|^{-\frac{7}{8}}(1+|\al||\la-U(1)|)^{-\frac{1}{4}}\n{F}_{L^{2}},
\end{align*}
where in the last inequality we have used that $|\al|\gtrsim 1+|\al||\la-U(1)|$ for this case.
\vspace{1em}

Next, we deal with \textbf{the case $F\in H^{-1}$}. By weak-type resolvent estimates in Lemma \ref{lemma:weak,est,w,FH-1}, we have
\begin{align*}
|c_{1}|=&\bbabs{\int_{-1}^{1}U''(y)\frac{\operatorname{sinh}\lrs{|\al|(1+y)}}
{\operatorname{\operatorname{sinh}}(2|\al|)}\frac{w_{Na}(y)}{U''(y)}dy}\\
\lesssim&|\al|^{-1} \n{F}_{H^{-1}}\big(|U(1)-\la|+\delta^{\f43}\big)^{-\f34}\delta^{-1}\notag\\
		&+ |\al|^{-1} \n{F}_{H^{-1}} \lrs{\bn
		{\text{\bf Ray}_{\delta_1}^{-1} f}_{H^{1}}+ \delta^{-\frac{4}{3}}(|\la-U(0) |^{\frac{1}{2}} + \delta)^{\frac{1}{3}}\bn{\text{\bf Ray}_{\delta_1}^{-1} f}_{L^2}}.\notag
\end{align*}
where $\delta=\nu^{\f14}|\al|^{-\f14}\ll 1$, $\delta_{1}=\delta^{\frac{4}{3}}\lrs{|\la-U(0)|^{\f12}+\delta}^{\frac{2}{3}}$, and
$f(y)=U''(y)\frac{\operatorname{sinh}\lrs{|\al|(1+y)}}
{\operatorname{\operatorname{sinh}}(2|\al|)}$
with the property that
\begin{align*}
 f(-1)=0,\quad \bbn{\frac{f}{U''}}_{L^{2}}\lesssim |\al|^{-\frac{1}{2}},\quad
\bbn{\lrs{\frac{f}{U''}}'}_{L^{2}}\lesssim |\al|^{\frac{1}{2}}.
\end{align*}
Noting that
\begin{align*}
|\al|^{-1} \n{F}_{H^{-1}}\big(|U(1)-\la|+\delta^{\f43}\big)^{-\f34}\delta^{-1}\lesssim& |\al|^{-1} \n{F}_{H^{-1}}\big(|U(1)-\la|+\delta^{\f43}\big)^{-\f14}\delta^{-\frac{5}{3}}\\
=&\nu^{-\frac{5}{12}}|\al|^{-\frac{7}{12}}(|\la-U(1)|+\nu^{\frac{1}{3}}|\al|^{-\frac{1}{3}})^{-\frac{1}{4}}\n{F}_{H^{-1}},
\end{align*}
with $\nu^{-\frac{5}{12}}|\al|^{-\frac{7}{12}}\leq \nu^{-\frac{1}{2}}|\al|^{-\frac{1}{2}}$, we arrive at
\begin{align}
|c_{1}|\lesssim &\nu^{-\frac{1}{2}}|\al|^{-\frac{1}{2}}(|\la-U(1)|+\nu^{\frac{1}{3}}|\al|^{-\frac{1}{3}})^{-\frac{1}{4}}\n{F}_{H^{-1}}\notag\\
&+|\al|^{-1} \n{F}_{H^{-1}} \lrs{\bn
		{\text{\bf Ray}_{\delta_1}^{-1} f}_{H^{1}}+ \delta^{-\frac{4}{3}}(|\la-U(0) |^{\frac{1}{2}} + \delta)^{\frac{1}{3}}\bn{\text{\bf Ray}_{\delta_1}^{-1} f}_{L^2}}.\label{equ:weak,est,w,FH-1,bounds,c12}
\end{align}
To conclude \eqref{equ:bounds on c1,c2,H-1}, it remains to bound the following two terms
\begin{align*}
\bn{\text{\bf Ray}_{\delta_1}^{-1} f}_{L^{2}},\quad \bn{\text{\bf Ray}_{\delta_1}^{-1} f}_{H^{1}}.
\end{align*}
Recall that
\begin{equation}\label{equ:Ray,U,delta,bounds,c12}
	\text{\bf Ray}_{\delta_{1}}W=:(U - \lambda + i\delta_{1})W-U''\Psi= f, \quad (\partial_y^2 - |\alpha|^2)\Psi=W, \quad \Psi(\pm 1) = 0.
\end{equation}

{\it Case 1.} $|\lambda|\gtrsim |U(0)|+|U(1)|$.

Taking the inner product between \eqref{equ:Ray,U,delta,bounds,c12} and $\f{W}{U''}$, we get
\begin{align*}
        \int_{-1}^{1} \f{U - \la}{U''}|W|^2 dy + \n{(\partial_y, |\alpha|)\Psi}_{L^2}^2 + i\delta_{1}\bbn{\f{W}{\sqrt{U''}}}_{L^2}^2 = \bblra{f,\f{W}{U''}} ,
\end{align*}
which gives that
\begin{align*}
|\la|\n{W}_{L^{2}}^{2}\lesssim& \int_{-1}^{1} |U(y)||W|^2 dy+\n{(\partial_y, |\alpha|)\Phi}_{L^2}^{2}+\bbn{\frac{f}{U''}}_{L^{2}}\n{W}_{L^{2}}\\
\lesssim &\n{W}_{L^{2}}^{2}+\bbn{\frac{f}{U''}}_{L^{2}}\n{W}_{L^{2}}.
\end{align*}
Due to that $|\lambda|\gtrsim |U(0)|+|U(1)|$, we obtain
\begin{align}\label{equ:Rf,L2,c1}
\bn{\text{\bf Ray}_{\delta_1}^{-1} f}_{L^{2}}=\n{W}_{L^{2}}\lesssim |\la-U(1)|^{-1}\bbn{\frac{f}{U''}}_{L^{2}}\lesssim
|\al|^{-\frac{1}{2}}|\la-U(1)|^{-1}.
\end{align}

For $\bn{\text{\bf Ray}_{\delta_1}^{-1} f}_{H^{1}}$, noting that
\begin{align*}
\frac{(U - \lambda + i\delta_{1})}{U''}\pa_yW-\pa_y\Psi= \lrs{\frac{f}{U''}}'-\lrs{\frac{U}{U''}}'W,
\end{align*}
we test by $\pa_yW$ and take the real part to obtain
\begin{align*}
|\la|\n{\pa_yW}_{L^{2}}^{2}
\lesssim& \int_{-1}^{1} \frac{|U|}{U''}|\pa_yW|^2 dy+\n{\pa_y\Psi}_{L^2}\n{\pa_yW}_{L^{2}}\\
&+\bbn{\lrs{\frac{f}{U''}}^{'}}_{L^{2}}\n{\pa_yW}_{L^{2}}+
\bbn{\lrs{\frac{U}{U''}}'}_{L^{\infty}}\n{W}_{L^{2}}\n{\pa_yW}_{L^{2}}.
\end{align*}
Due to that $|\lambda|\gtrsim |U(0)|+|U(1)|$ and $|\al|\geq 1$, we use \eqref{equ:Rf,L2,c1} to obtain
\begin{align*}
|\la-U(1)|\n{\pa_yW}_{L^{2}}\lesssim& |\la|\n{\pa_yW}_{L^{2}}\lesssim \n{\pa_y\Psi}_{L^{2}}+\bbn{\lrs{\frac{f}{U''}}^{'}}_{L^{2}}+\n{W}_{L^{2}}\\
\lesssim&  |\al|^{-1}\n{W}_{L^{2}}+|\al|^{\frac{1}{2}}+\n{W}_{L^{2}}\lesssim |\al|^{\frac{1}{2}},
\end{align*}
which gives that
\begin{align}\label{equ:Rf,H1,c1}
\bn{\text{\bf Ray}_{\delta_1}^{-1} f}_{H^{1}}=\n{W}_{H^{1}}\lesssim |\al|^{\frac{1}{2}}|\la-U(1)|^{-1}.
\end{align}

Therefore, with $\delta=\nu^{\f14}|\al|^{-\f14}$, putting \eqref{equ:Rf,L2,c1}--\eqref{equ:Rf,H1,c1} into \eqref{equ:weak,est,w,FH-1,bounds,c12}, we have
\begin{align*}
& |\al|^{-1} \n{F}_{H^{-1}} \lrs{\bn
		{\text{\bf Ray}_{\delta_1}^{-1} f}_{H^{1}}+ \delta^{-\frac{4}{3}}(|\la-U(0) |^{\frac{1}{2}} + \delta)^{\frac{1}{3}}\bn{\text{\bf Ray}_{\delta_1}^{-1} f}_{L^2}}\\
\lesssim&|\al|^{-1} \n{F}_{H^{-1}} \lrs{
		|\al|^{\frac{1}{2}}|\la-U(1)|^{-1}+ \delta^{-\frac{4}{3}}|\la-U(0) |^{\frac{1}{6}}|\al|^{-\frac{1}{2}}|\la-U(1)|^{-1}},
\end{align*}
which is good enough to conclude the desired estimate \eqref{equ:bounds on c1,c2,H-1}.\smallskip

{\it Case 2.} $|\lambda|\lesssim |U(0)|+|U(1)|$.

In this case, due to that $\delta_{1}=\delta^{\frac{4}{3}}\lrs{|\la-U(0)|^{\f12}+\delta}^{\frac{2}{3}}\ll 1$, by Lemma \ref{lemma:proper,phi,ray,equ,refined}, we have
\begin{align}\label{equ:Phi,bounds,c12}
\n{(\pa_{y},|\al|)\Psi}_{L^{2}}\lesssim |\al|^{-\frac{1}{2}}.
\end{align}

For $\bn{\text{\bf Ray}_{\delta_1}^{-1} f}_{L^{2}}$, we introduce the decomposition that $W=W_{1}+W_{2}$ with
\begin{align*}
&(\pa_{y}^{2}-|\al|^{2})\Psi_{i}=W_{i},\quad \Psi_{i}(\pm1)=0,\\
&W_{1}=\frac{f(1)\rho(y)}{U(y)-\la+i\delta_{1}},\quad \rho(y)=\max\lr{1-|\al|(1-y),0},\\
&(U(y) - \lambda + i\delta_{1})W_{2}-U''(y)\Psi_{2}= \wt{f}=f-f(1)\rho(y)+U''(y)\Psi_{1}, \quad \wt{f}(\pm 1)=0.
\end{align*}
For $\n{W_{1}}_{L^{2}}$, we have
\begin{align}\label{equ:W1,bounds,c12}
\n{W_{1}}_{L^{2}}\lesssim |f(1)|\bbn{\frac{\rho(y)}{U(y)-\la+i\delta_{1}}}_{L^{2}}\lesssim \delta_{1}^{-\frac{1}{2}}|f(1)|
\lesssim\delta_{1}^{-\frac{1}{2}}.
\end{align}
For $\n{(\pa_{y},|\al|)\Psi_{1}}_{L^{2}}$, by Hardy-type inequality in Lemma \ref{lemma:hardy inequality}, we get
\begin{align}\label{equ:Phi1,bounds,c12}
\n{(\pa_{y},|\al|)\Psi_{1}}_{L^{2}}^{2}=&\babs{\lra{W_{1},\Psi_{1}}}
=\bbabs{\int_{1-|\al|^{-1}}^{1}
\frac{f(1)\rho(y)\overline{\Psi}_{1}(y)}
{U(y)-\la+i\delta_{1}}
dy}\\
\lesssim& |\al|^{-\frac{1}{2}}\n{\pa_{y}(\rho\Psi_{1})}_{L^{2}(1-|\al|^{-1},1)}+
|\al|\n{\rho \Psi_{1}}_{L^{2}(1-|\al|^{-1},1)}\notag\\
\lesssim& |\al|^{-\frac{1}{2}}\n{\pa_{y}\Psi_{1}}_{L^{2}}+|\al|^{\frac{1}{2}}\n{\Psi_{1}}_{L^{2}}\lesssim
|\al|^{-1}\n{(\pa_{y},|\al|)\Psi_{1}}_{L^{2}}.\notag
\end{align}
For $\n{(\pa_{y},|\al|)\Psi_{2}}_{L^{2}}$, we use \eqref{equ:Phi,bounds,c12} and \eqref{equ:Phi1,bounds,c12} to obtain
\begin{align}\label{equ:Phi2,bounds,c12}
\n{(\pa_{y},|\al|)\Psi_{2}}_{L^{2}}\lesssim \n{(\pa_{y},|\al|)\Psi_{1}}_{L^{2}}+\n{(\pa_{y},|\al|)\Psi}_{L^{2}}\lesssim
|\al|^{-\frac{1}{2}}.
\end{align}

For $\n{W_{2}}_{L^{2}}^{2}$, by testing $\frac{W_{2}}{U''}$ and taking the imaginary part, we have
\begin{align}\label{equ:W2,bounds,c12}
\delta_{1}\n{W_{2}}_{L^{2}}^{2}\lesssim \babs{\operatorname{Im}\blra{\frac{\wt{f}}{U''},W_{2}}}\lesssim
\bbn{(\pa_{y},|\al|)\frac{\wt{f}}{U''}}_{L^{2}}\n{(\pa_{y},|\al|)\Psi_{2}}_{L^{2}}\lesssim 1
\end{align}
where in the last inequality we have used \eqref{equ:Phi2,bounds,c12} and that
\begin{align*}
\bbn{(\pa_{y},|\al|)\frac{\wt{f}}{U''}}_{L^{2}}\lesssim \bbn{(\pa_{y},|\al|)\frac{f}{U''}}_{L^{2}}+\bbn{(\pa_{y},|\al|)\frac{\rho}{U''}}_{L^{2}}+\n{(\pa_{y},|\al|)\Psi_{1}}_{L^{2}}
\lesssim |\al|^{\frac{1}{2}}.
\end{align*}
Therefore, combining \eqref{equ:W1,bounds,c12} and \eqref{equ:W2,bounds,c12}, we arrive at
\begin{align}\label{equ:W,bounds,c12}
\bn{\text{\bf Ray}_{\delta_1}^{-1} f}_{L^{2}}=\n{W}_{L^{2}}\lesssim \delta_{1}^{-\frac{1}{2}}.
\end{align}

For $\bn{\text{\bf Ray}_{\delta_1}^{-1} f}_{H^{1}}$, we note that
\begin{align*}
\frac{(U(y) - \lambda + i\delta_{1})}{U''}\pa_yW-\pa_y\Psi= \lrs{\frac{f}{U''}}'-\lrs{\frac{U}{U''}}'W.
\end{align*}
On the one hand, due to that $(f/U'')'\leq |\al|e^{-|\al|(1-y)}$, we get
\begin{align*}
\bbn{\frac{U''}{U-\la+i\delta_{1}}\lrs{\frac{f}{U''}}'}_{L^{2}}
\lesssim& |\al|e^{-\frac{|\al|}{2}}
\bbn{\frac{U''1_{(-\frac{1}{2},\frac{1}{2})}}{U-\la+i\delta_{1}}}_{L^{\infty}}+
\bbn{\frac{U''1_{(-\frac{1}{2},\frac{1}{2})^{c}}}{U-\la+i\delta_{1}}}_{L^{2}}
|\al|\\
\lesssim& \delta_{1}^{-1}|\al|^{-\frac{1}{2}}+\delta_{1}^{-\frac{1}{2}}|\al|.
\end{align*}
On the other hand, we have
\begin{align*}
\bbn{\frac{U''}{U-\la+i\delta_{1}}\lrs{\frac{U}{U''}}'}_{L^{\infty}}\lesssim &
\bbn{\frac{U'}{U-\la+i\delta_{1}}}_{L^{\infty}}+\bbn{\frac{U'''}{U-\la+i\delta_{1}}}_{L^{\infty}}\\
\lesssim& \bbn{\frac{y}{U(y)-U(0)+(U(0)-\la)+i\delta_{1}}}_{L^{\infty}}\\
\lesssim&\delta_{1}^{-1}(|\la-U(0)|+\delta_{1})^{\frac{1}{2}} .
\end{align*}
Then
by \eqref{equ:Phi,bounds,c12} and \eqref{equ:W,bounds,c12}, we obtain
\begin{align*}
&\n{\pa_{y}(\text{\bf Ray}_{\delta_1}^{-1} f)}_{L^{2}}
\lesssim \bbn{\frac{U''}{U-\la+i\delta_{1}}\lrs{\pa_y\Psi+\lrs{\frac{f}{U''}}'-\lrs{\frac{U}{U''}}'W}}_{L^{2}}\\
\lesssim&\delta_{1}^{-1}|\al|^{-\frac{1}{2}}+\delta_{1}^{-\frac{1}{2}}|\al|+\bbn{\frac{U''}{U-\la+i\delta_{1}}}_{L^{\infty}}\n{\pa_y\Psi}_{L^{2}}+\bbn{\frac{U''}{U-\la+i\delta_{1}}\lrs{\frac{U}{U''}}'}_{L^{\infty}}\n{W}_{L^{2}}\\
\lesssim& \delta_{1}^{-1}|\al|^{-\frac{1}{2}}+\delta_{1}^{-\frac{1}{2}}|\al|+\delta_{1}^{-\frac{3}{2}}(|\la-U(0)|+\delta_{1})^{\frac{1}{2}}.
\end{align*}

Due to the condition \eqref{equ:restriction,bounds on c12}, we have $|\al|\lesssim \delta^{-\frac{4}{3}}(|\la-U(0) |^{\frac{1}{2}} + \delta)^{\frac{1}{3}}$.
With $\delta=\nu^{\f14}|\al|^{-\f14}$ and $\delta_{1}=\delta^{\frac{4}{3}}\lrs{|\la-U(0)|^{\f12}+\delta}^{\frac{2}{3}}$, we get
\begin{align*}
&|\al|^{-1} \n{F}_{H^{-1}} \lrs{\bn
		{\text{\bf Ray}_{\delta_1}^{-1} f}_{H^{1}}+ \delta^{-\frac{4}{3}}(|\la-U(0) |^{\frac{1}{2}} + \delta)^{\frac{1}{3}}\bn{\text{\bf Ray}_{\delta_1}^{-1} f}_{L^2}}\\
\lesssim&|\al|^{-1} \n{F}_{H^{-1}}\lrs{\delta_{1}^{-1}|\al|^{\frac{1}{2}}+\delta_{1}^{-\frac{3}{2}}(|\la-U(0)|+\delta_{1})^{\frac{1}{2}}
+\delta^{-\frac{4}{3}}(|\la-U(0) |^{\frac{1}{2}} + \delta)^{\frac{1}{3}}\delta_{1}^{-\frac{1}{2}}}\\
\lesssim& |\al|^{-1}\delta^{-2}\lesssim \nu^{-\frac{1}{2}}|\al|^{-\frac{1}{2}}(|\la-U(1)|+\nu^{\frac{1}{3}}|\al|^{-\frac{1}{3}})^{-\frac{1}{4}},
\end{align*}
where we have used $\delta_{1}^{-1}\lesssim \delta^{-2}$ and $1\lesssim (|\la-U(1)|+\nu^{\frac{1}{3}}|\al|^{-\frac{1}{3}})^{-\frac{1}{4}}$. Together with \eqref{equ:weak,est,w,FH-1,bounds,c12}, we arrive at the desired estimate \eqref{equ:bounds on c1,c2,H-1}.

\vspace{1em}
Finally, we deal with \textbf{the case $F\in H^{1}$}.\smallskip

{\it Case 1.} $|\lambda|\gtrsim |U(0)|+|U(1)|$.

By energy estimates \eqref{equ:energy estimate,im,w,0} and \eqref{equ:energy estimate,im,w,1}, we have
\begin{align*}
|\al||\la|\n{w_{Na}}_{L^{2}}^{2}\lesssim \babs{\blra{F,\frac{w_{Na}}{U''}}}\lesssim \bbn{(\pa_{y},|\al|)\frac{F}{U''}}_{L^{2}}\n{(\pa_{y},|\al|)\psi_{Na}}_{L^{2}}
\end{align*}
which, together with resolvent estimate \eqref{equ:reso,est,u,w,FH1} gives that
\begin{align*}
\n{w_{Na}}_{L^{2}}\lesssim |\la|^{-\frac{1}{2}}|\al|^{-1}\n{F}_{H^{1}}.
\end{align*}
Therefore, for this case, we have
\begin{align*}
|c_{1}|\lesssim \n{w_{Na}}_{L^{1}}\lesssim\n{w_{Na}}_{L^{2}}\lesssim \nu^{-\frac{1}{8}}|\al|^{-\frac{7}{8}}(1+|\la-U(1)|)^{-\frac{1}{4}}\n{F}_{H^{1}},
\end{align*}
which completes the proof of \eqref{equ:bounds on c1,c2,H1}.\smallskip

{\it Case 2.} $|\lambda|\lesssim |U(0)|+|U(1)|$.

In this case, we directly use resolvent estimate \eqref{equ:reso,est,u,w,FH1} to obtain
\begin{align*}
|c_{1}|\lesssim  \n{u_{Na}}_{L^{\infty}}\lesssim \nu^{-\frac{1}{8}}|\al|^{-\frac{7}{8}}\n{F}_{H^{1}},
\end{align*}
which suffices for our desired estimate \eqref{equ:bounds on c1,c2,H1} due to that $1+|\la-U(1)|\sim 1$.
\end{proof}

\section{Resolvent estimates with non-slip boundary condition}\label{sec:Orr-Sommerfeld Equation with the Non-slip Boundary Condition}
Now we are in a position to establish the resolvent estimates for the Orr-Sommerfeld equation with the non-slip boundary condition which reads
\begin{equation}\label{equ:nonslip,NS,per,equ}
	\left\{
	\begin{aligned}
		&-\nu(\pa^2_y-\al^2)w+i\al(U-\la)w-i\al U''\psi+o(\nu,\al)w=F,\\
		&(\pa^2_y-\al^2)\psi=w,\quad \psi(\pm 1)=\psi'( \pm 1)=0,
	\end{aligned}\right.
\end{equation}
where $|o(\nu,\al)|\ll \nu^{\frac{1}{2}}|\al|^{\frac{1}{2}}$.

\begin{proposition}\label{proposition:est,u,L2,FL2H1H-1}
Let $(w,\psi)$ be the solution to \eqref{equ:nonslip,NS,per,equ}. There exist positive constants $\nu_0$, $\varepsilon_0$, such that if $|o(\nu,\al)|\leq \ve_{0}\nu^{\frac{1}{2}}|\al|^{\frac{1}{2}}$, $\nu\in(0,\nu_{0}]$, the following resolvent estimates hold
\begin{align}
& \nu^{\frac{5}{8}}|\al|^{\frac{7}{8}}\|w\|_{L^2}+\nu^{\frac{3}{8}}\left|\al\right|^{\frac{5}{8}}\|u\|_{L^{\infty}}+\nu^{\frac{1}{4}}\left|\al\right|^{\frac{3}{4}}\|u\|_{L^2} \lesssim\|F\|_{L^2},\label{equ:reso,est,u,w,FL2,nonslip} \\
& \nu^{\frac{3}{4}}\left|\al\right|^{\frac{1}{4}}\|w\|_{L^2}+\nu^{\frac{5}{8}}\left|\al\right|^{\frac{3}{8}}\|u\|_{L^{\infty}}+\nu^{\frac{1}{2}}\left|\al\right|^{\frac{1}{2}}\|u\|_{L^2} \lesssim\|F\|_{H^{-1}}, \label{equ:reso,est,u,w,FH-1,nonslip}\\
& \nu^{\frac{3}{8}} |\al|^{\frac{5}{8}}\|w\|_{L^2}+\nu^{\frac{1}{8}}|\al|^{\frac{7}{8}}\|u\|_{L^{\infty}}+\|u\|_{L^2} \lesssim\left\|\left(\partial_y,|\al|\right) F\right\|_{L^2} . \label{equ:reso,est,u,w,FH1,nonslip}
\end{align}

\end{proposition}
\begin{proof}
The proof of which is divided into the following two cases.\smallskip

{\it Case 1.} $\nu |\al|^2 \lesssim |\la- U(0)|^{\frac{1}{2}} + \nu^{\frac{1}{4}}|\al|^{-\f14}$.

Noting that $w=w_{Na}+c_1w_{cor,1}+c_2w_{cor,2}$,
by Proposition \ref{lemma:summary,reso,navierslip}, we have
\begin{equation}\label{equ:reso,NS,nonslip,proof}
\left\{
	\begin{aligned}
\nu^{\f12}\abs{\al}^{\f12}\n{w_{Na}}_{L^2}+\nu^{\f38}\abs{\al}^{\f58}\n{u_{Na}}_{L^\infty}+\nu^{\f14}\abs{\al}^{\f34}\n{u_{Na}}_{L^2}&\lesssim \n{F}_{L^2},\\
	\nu^{\f34}\abs{\al}^{\f14}\n{w_{Na}}_{L^2}+\nu^{\f58}\abs{\al}^{\f38}\n{u_{Na}}_{L^\infty}+\nu^{\f12}\abs{\al}^{\f12}\n{u_{Na}}_{L^2}
&\lesssim \n{F}_{H^{-1}},\\
\nu^{\f14}\abs{\al}^{\f34}\n{w_{Na}}_{L^2}+\nu^{\f18}\abs{\al}^{\f78}\n{u_{Na}}_{L^\infty}+\abs{\al}\n{u_{Na}}_{L^2}&\lesssim \n{F}_{H^{1}}.
	\end{aligned}
\right.
\end{equation}
By Lemma \ref{lemma:bounds on c1,c2} and Lemma \ref{lemma:est,uw,L2Linfty,cor12}, we have
\begin{equation}\label{equ:reso,c12,nonslip,proof}
\left\{
\begin{aligned}
|c_{1}|+|c_{2}|\lesssim& \nu^{-\frac{3}{8}}|\al|^{-\frac{7}{8}}(1+|\al||\la-U(1)|)^{-\frac{1}{4}}\n{F}_{L^{2}},\\
|c_{1}|+|c_{2}|\lesssim& \nu^{-\frac{1}{2}}|\al|^{-\frac{1}{2}}(|\la-U(1)|+\nu^{\frac{1}{3}}|\al|^{-\frac{1}{3}})^{-\frac{1}{4}}\n{F}_{H^{-1}},\\
|c_{1}|+|c_{2}|\lesssim& \nu^{-\frac{1}{8}}|\al|^{-\frac{7}{8}}(1+|\la-U(1)|)^{-\frac{1}{4}}\n{F}_{H^{1}},
\end{aligned}
\right.
\end{equation}
and
\begin{equation}\label{equ:reso,wcor12,nonslip,proof}
\left\{
\begin{aligned}
 \n{w_{cor,j}}_{L^2}\lesssim& (L_{j})^{\frac{1}{2}}\lrs{1+|L_{j}d_{j}|}^{\f14}\sim  \nu^{-\frac{1}{4}}\abs{\al}^{\frac{1}{4}}
 (|\la-U(1)|+
  \nu^{\frac{1}{3}}|\al|^{-\frac{1}{3}})^{\frac{1}{4}},\\
  	\n{u_{cor,j}}_{L^2} \lesssim&(L_{j})^{-\frac{1}{2}}\lrs{1+|L_{j}d_{j}|}^{-\f14},\\
  \n{u_{cor,j}}_{L^\infty}\lesssim& 1,
  	\end{aligned}
  \right.
  \end{equation}
  where we have used
  \begin{align*}
L_{j}=|\al/\nu|^{\frac{1}{3}},\quad
(|\la-U(1)|+ \nu^{\frac{1}{3}}|\al|^{-\frac{1}{3}})\sim
 |\al/\nu|^{-\frac{1}{3}}(1+|\al/\nu|^{\frac{1}{3}}|\la-U(1)|)\sim L_{j}^{-1}(1+|L_{j}d_{j}|).
  \end{align*}
   In this case,
  we have that $\nu |\al|^2 \leq |\la- U(0)|^{\frac{1}{2}} + \nu^{\frac{1}{4}}|\al|^{-\f14}$ and hence get
  \begin{align}\label{equ:reso,wcor12,nonslip,proof,Ldj}
  1+|L_{j}d_{j}|\lesssim |\al|^{-\frac{2}{3}}\nu^{-\frac{1}{3}}\lrs{\lrs{\nu|\al|^{2}}^{\frac{1}{3}}+|\al||\la-U(1)|}
  \lesssim  |\al|^{-\frac{2}{3}}\nu^{-\frac{1}{3}}\lrs{1+|\al||\la-U(1)|}.
  \end{align}
Putting \eqref{equ:reso,NS,nonslip,proof}, \eqref{equ:reso,c12,nonslip,proof}, \eqref{equ:reso,wcor12,nonslip,proof}, and
\eqref{equ:reso,wcor12,nonslip,proof,Ldj} together, we complete the proof of \eqref{equ:reso,est,u,w,FL2,nonslip}--\eqref{equ:reso,est,u,w,FH1,nonslip}.\smallskip

{\it Case 2.} $\nu |\al|^2 \gg |\la- U(0)|^{\frac{1}{2}} + |\nu|^{\frac{1}{4}}|\al|^{-\f14}$.

Testing \eqref{equ:nonslip,NS,per,equ} by $\psi$, we obtain
\begin{align*}
         &\nu\n{w}_{L^2}^2+i\al\int_{-1}^1\lrs{(U-\la)(\abs{\pa_y\psi}^2+\abs{\al}^2\abs{\psi}^2)+U''\abs{\psi}^2}dy\\
         		=&-i\al\int_{-1}^1 U'\pa_y\psi\ol{\psi}dy-\int_{-1}^1F\ol{\psi}dy.
         	\end{align*}
         Noting that
         $$\operatorname{Im}\lrs{-i\al \int_{-1}^{1} U'\pa_y\psi\ol{\psi}dy}=\frac{\al}{2}\int_{-1}^{1} U''|\psi|^{2}dy,$$
and taking the imaginary part, we have
\begin{align}\label{equ:im,psi,test,nonslip}
&|\al|\int_{-1}^1\big(U(y)-U(0)\big)(\abs{\pa_y\psi}^2+\abs{\al}^2\abs{\psi}^2)+\frac{U''\abs{\psi}^2}{2}dy\\
\leq&|\operatorname{Im}\lra{F,\psi}|+|\al||\la-U(0)|\int_{-1}^{1}(\abs{\pa_y\psi}^2+\abs{\al}^2\abs{\psi}^2)dy.\notag
\end{align}
By taking the real part, we obtain
\begin{align*}
\nu\n{w}_{L^{2}}^{2}=&\nu\n{\pa_y^2\psi}_{L^{2}}^{2}+\nu \al^{2}\lrs{2\n{\pa_y\psi}_{L^{2}}^{2}+\al^{2}\n{\psi}_{L^{2}}^{2}}\\
\leq& |\operatorname{Re}\lra{F,\psi}|+\frac{2|\al|}{\nu|\al|^{3}}\n{U'\pa_y\psi}_{L^{2}}^{2}+\frac{\nu|\al|^{4}}{2}\n{\psi}_{L^{2}}^{2}.\notag
\end{align*}
By using \eqref{equ:im,psi,test,nonslip} and that $|U'(y)|^{2}\lesssim |y|^{2}\lesssim U(y)-U(0)$, we get
\begin{align*}
\nu\n{\pa_y^2\psi}_{L^{2}}^{2}+\nu \al^{2}\lrs{2\n{\pa_y\psi}_{L^{2}}^{2}+\al^{2}\n{\psi}_{L^{2}}^{2}}\lesssim
|\lra{F,\psi}|+\frac{|\la-U(0)|}{\nu|\al|^{2}}\int_{-1}^{1}(\abs{\pa_y\psi}^2+\abs{\al}^2\abs{\psi}^2)dy.
\end{align*}
Due to the condition that $|\la- U(0)|\ll \nu^{2}|\al|^{4}$, we arrive at
\begin{align*}
\nu\n{w}_{L^{2}}^{2}=\nu\n{\pa_y^2\psi}_{L^{2}}^{2}+\nu \al^{2}\lrs{2\n{\pa_y\psi}_{L^{2}}^{2}+\al^{2}\n{\psi}_{L^{2}}^{2}}\lesssim |\lra{F,\psi}|.
\end{align*}
Noting that
\begin{align*}
|\lra{F,\psi}|\lesssim & |\al|^{-1}\n{F}_{H^{-1}}\n{w}_{L^{2}}\lesssim |\al|^{-2}\n{F}_{L^{2}}\n{w}_{L^{2}}\lesssim |\al|^{-3}
\n{(\pa_{y},|\al|)F}_{L^{2}}\n{w}_{L^{2}},
\end{align*}
using that $\nu |\al|^{3}\gtrsim 1$, we derive \eqref{equ:reso,est,u,w,FL2,nonslip}--\eqref{equ:reso,est,u,w,FH1,nonslip}.
\end{proof}

\section{Linear inviscid damping and enhanced dissipation}\label{sec:Inviscid Damping and Enhanced Dissipation}

In this section, we focus on deriving the linear enhanced dissipation estimates, which play a crucial role in the analysis of nonlinear stability.
To achieve this, we consider the following linearized Naiver-Stokes system with non-slip boundary condition
\begin{equation}\label{equ:linear,nonslip,NS,fj}
	\left\{
	\begin{aligned}
		&\pa_{t}\om-\nu(\pa^2_y-\al^2)\om+i\al U\om-i\al U''\phi=i\al f_1+\pa_y f_2,\\
		&(\pa^2_y-\al^2)\phi=\om,\quad \phi(t,\pm 1)=\phi'( t,\pm 1)=0,\quad \om(0,y)=\om_{\al}^{\mathrm{in}}(y),
	\end{aligned}\right.
\end{equation}
where we omit the subscript $\al$ for brevity.
\begin{proposition}\label{proposition:est,u,om,timespace,f1234}
	Let $(\om,\phi)$ be the solution to \eqref{equ:linear,nonslip,NS,fj} with $\lra{\om^{\mathrm{in}},e^{\pm \abs{\al}y}}=0$. There exist positive constants $\nu_0$, $\varepsilon_0$, such that for $\nu\in(0,\nu_0]$, $\ve\in[0,\ve_{0}]$, it holds that
	\begin{align}\label{equ:est,u,om,timespace,f1234}
		&\abs{\al}\n{e^{\varepsilon \nu^{\f12}t}u}_{L_{t}^{\infty}L_{y}^{2}}^{2}
		+\abs{\al}\n{e^{\varepsilon \nu^{\f12}t}u}_{L_{t,y}^{2}}^{2}+\n{e^{\varepsilon \nu^{\f12}t}u}_{L_{t}^{\infty}L_{y}^{\infty}}^{2}+\nu^{\f12}\abs{\al}^{\f12}\n{e^{\varepsilon \nu^{\f12}t}\om}^2_{L_{t,y}^{2}}\\
		&+ \nu^{\frac{1}{2}}\n{e^{\varepsilon \nu^{\f12}t}\om}_{L_{t}^{\infty}L_{y}^{2}}^{2}+\nu^{\f32}\n{e^{\varepsilon \nu^{\f12}t}\pa_y\om}^2_{L_{t,y}^{2}}+\n{e^{\varepsilon \nu^{\f12}t}\sqrt{1-y^{2}}\om}^2_{L_{t}^{\infty}L_{y}^{2}}\notag\\
		\lesssim& \n{\om^{\mathrm{in}}}_{H_{\al}^{4}}^{2}+\nu^{-1}\n{e^{\varepsilon \nu^{\f12}t}(f_1,f_2)}^2_{L_{t,y}^{2}}.\notag
	\end{align}
\end{proposition}

The proof of Proposition \ref{proposition:est,u,om,timespace,f1234} is postponed to Section \ref{sec:The proof of Proposition}. To establish Proposition \ref{proposition:est,u,om,timespace,f1234}, we will decompose the system \eqref{equ:linear,nonslip,NS,fj} into homogeneous and inhomogeneous parts. In the following two subsections, we will address these two problems separately.

  \subsection{The inhomogeneous linearized NS system}
We first consider the following inhomogeneous linearized Naiver-Stokes system with zero initial data
\begin{equation}\label{equ:linear,inhom,nonslip,NS,fj}
	\left\{
	\begin{aligned}
		&\pa_{t} \om_{I}-\nu(\pa^2_y-\al^2) \om_{I}+i\al U \om_{I}-i\al U''\phi_{I}=i\al f_1+ \pa_y f_2+f_3+f_4,\\
		&(\pa^2_y-\al^2)\phi_{I}= \om_{I},\quad \phi_{I}(t,\pm 1)=\phi'_{I}( t,\pm 1)=0,\quad  \om_{I}(0,y)=0.
	\end{aligned}\right.
\end{equation}

\begin{proposition}\label{proposition:est,u,timespace,L2L2,LinftyL2,f1234L2}
Let $( \om_I,\phi_{I})$ solve \eqref{equ:linear,inhom,nonslip,NS,fj}. There exist positive constants $\nu_0$, $\varepsilon_0$, such that for $\nu\in(0,\nu_0]$, $\ve\in[0,\ve_{0}]$, the following space-time estimates hold
\begin{align}
		\n{e^{\ve(\nu|\al|)^{\frac{1}{2}}t}u_I}^2_{L_{t,y}^{2}}
		\lesssim &\nu^{-1}\abs{\al}^{-1}\n{e^{\ve(\nu|\al|)^{\frac{1}{2}}t}(f_1,f_2)}^2_{L_{t,y}^{2}}+\nu^{-\f12}\abs{\al}^{-\f32}\n{e^{\ve (\nu\abs{\al})^{\f12}t}f_3}^2_{L_{t,y}^{2}}
        \label{equ:est,u,timespace,L2L2,LinftyL2,f1234L2}\\
		&+\abs{\al}^{-1}\n{e^{\ve(\nu|\al|)^{\frac{1}{2}}t}(\pa_y,|\al|) f_4)}^2_{L_{t,y}^{2}},\notag\\
		\n{e^{\ve(\nu|\al|)^{\frac{1}{2}}t}\om_I}^2_{L_{t,y}^{2}}
		\lesssim &\nu^{-\f32}\abs{\al}^{-\f12}\n{e^{\ve(\nu|\al|)^{\frac{1}{2}}t}(f_1,f_2)}^2_{L_{t,y}^{2}}+\nu^{-\f54}\abs{\al}^{-\f74}\n{e^{\ve (\nu\abs{\al})^{\f12}t}f_3}^2_{L_{t,y}^{2}}\label{equ:est,w,timespace,L2L2,LinftyL2,f1234L2}\\
		&+\nu^{-\f34}\abs{\al}^{-\f54}\n{e^{\ve(\nu\abs{\al})^{\f12}t}(\pa_y,|\al|) f_4)}^2_{L_{t,y}^{2}}.\notag
	\end{align}
\end{proposition}

\begin{proof}
First, we introduce the following notations
	\begin{align*}
		\wt{\om}_{I}(\tau,y)&=\int_0^{+\infty}e^{\ve(\nu|\al|)^{\frac{1}{2}}t-it\tau} \om_I(t,y)dt,\quad
		\wt{\phi}_{I}(\tau,y)=\int_0^{+\infty}e^{\ve(\nu|\al|)^{\frac{1}{2}}t-it\tau}\phi_{I}(t,y)dt, \\
		\wt{f}_j(\tau,y)&=\int_0^{+\infty}e^{\ve(\nu|\al|)^{\frac{1}{2}}t-it\tau}f_j(t,y)dt,\quad j\in \{1,2,3,4\},
	\end{align*}
	where $c\in \R.$ Consequently, we turn to consider the OS equation with the non-slip boundary condition
	\begin{equation}\label{equ:linear,inhom,nonslip,NS,Fj}
		\left\{
		\begin{aligned}
			&-\nu(\pa^2_y-\al^2)\wt{\om}_{I}+i\al\big( U\wt{\om}_{I}+\f{\tau}{\al}+i\ve\nu^{\frac{1}{2}}|\al|^{-\frac{1}{2}}-U''\wt{\psi}_{I}\big)= i\al \wt{f}_1+ \pa_y\wt{f}_2+\wt{f}_3+\wt{f}_4,\\
			&(\pa^2_y-\al^2)\wt{\phi}_{I}=\wt{\om}_{I},\quad \wt{\phi}_{I}(\pm 1)=\wt{\phi}_{I}'(\pm 1)=0.
		\end{aligned}\right.
	\end{equation}

	By resolvent estimates in Proposition \ref{proposition:est,u,L2,FL2H1H-1}, we get
     \begin{align}\label{equ:est,uly2,F1234ly2}
     	\n{\wt{u}_{I}}_{L^2}
     	\lesssim \nu^{-\f12}\abs{\al}^{-\f12}\n{(\wt{f}_{1},\wt{f}_{2})}_{L^{2}}+\nu^{-\f14}\abs{\al}^{-\f34}\n{\wt{f}_{3}}_{L^2}+\abs{\al}^{-1}\n{(\pa_{y},|\al|)\wt{f}_{4}}_{L^2}.
     \end{align}
	Applying Plancherel's theorem, we have
	\begin{align*}
		&\int_0^{+\infty}\n{e^{\ve(\nu|\al|)^{\frac{1}{2}}t}u}^2_{L^2}dt\sim
		\int_{\R}\n{\wt{u}(\tau)}^2_{L^2}d\tau, \\
		&\int_0^{+\infty}\n{e^{\ve(\nu|\al|)^{\frac{1}{2}}t}{\pa_y^k}f_j(t)}^2_{L^2}dt\sim
		\int_{\R}\n{\pa_y^k \wt{f}_j(\tau)}^2_{L^2}d\tau,\quad j\in \{1,2,3,4\},\quad k\in\{0,1\}.
	\end{align*}
	Thus, from \eqref{equ:est,uly2,F1234ly2}, we obtain
	  \begin{align*}
		\n{e^{\ve(\nu|\al|)^{\frac{1}{2}}t}u_I}_{L_{t,y}^{2}}
		\lesssim& \nu^{-\f12}\abs{\al}^{-\f12}\n{e^{\ve(\nu|\al|)^{\frac{1}{2}}t}(f_1,f_2)}_{L_{t,y}^{2}}+\nu^{-\f14}\abs{\al}^{-\f34}\n{e^{\ve(\nu|\al|)^{\frac{1}{2}}t}f_3}_{L_{t,y}^{2}}
		\notag\\
		&+\abs{\al}^{-1}\n{e^{\ve(\nu|\al|)^{\frac{1}{2}}t}(\pa_{y},|\al|)f_4}_{L_{t,y}^{2}}.
		\notag
	\end{align*}
Similarly, by invoking Proposition \ref{proposition:est,u,L2,FL2H1H-1} again, we derive
		\begin{align*}
		\n{e^{\ve(\nu|\al|)^{\frac{1}{2}}t}\om_I}_{L_{t,y}^{2}}
		\lesssim &\nu^{-\f34}\abs{\al}^{-\f14}\n{e^{\ve(\nu|\al|)^{\frac{1}{2}}t}(f_1,f_2)}_{L_{t,y}^{2}}+\nu^{-\f58}\abs{\al}^{-\f78}
        \n{e^{\ve(\nu|\al|)^{\frac{1}{2}}t}f_3}_{L_{t,y}^{2}}\notag\\
        &+\nu^{-\f38}\abs{\al}^{-\f58}
        \n{e^{\ve(\nu|\al|)^{\frac{1}{2}}t}f_4}_{L_{t,y}^{2}}\notag.
	\end{align*}
Hence, we end the proof of the $\n{\om_{I}}_{L_{t,y}^{2}}$ and $\n{u_{I}}_{L_{t,y}^{2}}$ estimates in \eqref{equ:est,u,timespace,L2L2,LinftyL2,f1234L2} and
\eqref{equ:est,w,timespace,L2L2,LinftyL2,f1234L2}.
\end{proof}

 \subsection{The homogeneous linearized NS system}
In this subsection, we consider the following homogeneous linearized Naiver-Stokes system
\begin{equation}\label{equ:linear,hom,nonslip,NS,initial}
	\left\{
	\begin{aligned}
		&\pa_{t} \om_{H}-\nu(\pa^2_{y}-\al^2) \om_{H}+i\al U \om_{H}-i\al U''\phi_{H}=0,\\
		&(\pa^2_y-\al^2)\phi_{H}= \om_{H},\quad \phi_{H}(t,\pm 1)=\phi'_{H}( t,\pm 1)=0,\quad  \om_{H}(0,y)= \om_{\al}^{\mathrm{in}}(y).
	\end{aligned}\right.
\end{equation}

\begin{proposition}\label{proposition:est,u,timespace,L2L2,LinftyL2,w0H1}
 Let $( \om_{H},\phi_{H})$ solve \eqref{equ:linear,hom,nonslip,NS,initial}.
  There exist positive constants $\nu_0$, $\varepsilon_0$, such that for $\nu\in(0,\nu_0]$, $\ve\in[0,\ve_{0}]$, $\nu\al^{2}\lesssim 1$, the following space-time estimates hold
	\begin{align}\label{equ:est,u,timespace,L2L2,LinftyL2,w0H1}
		&\nu^{\f12}	\abs{\al}^{\f12}\n{e^{\ve_{0}(\nu|\al|)^{\frac{1}{2}}t} \om_{H}}^2_{L_{t,y}^{2}}+	\abs{\al}	\n{e^{\ve_{0}(\nu|\al|)^{\frac{1}{2}}t}u_H}^2_{L_{t,y}^{2}}\lesssim
		\n{\om^{\mathrm{in}}}_{H_{\al}^{4}}^{2}.
	\end{align}
\end{proposition}
\begin{proof}

	Let $\om_{H}^{(0)}$ solve
	\begin{align*}
		&\pa_{t}\om^{(0)}_{H}+i\al U\om^{(0)}_{H}-i\al U''\phi^{(0)}_{H}=0,\quad (\pa_y^2-\abs{\al}^2)\psi_{H}^{(0)}=\om_{H}^{(0)},\quad \om_{H}^{(0)}(0,y)=\om^{\mathrm{in}}(y).
	\end{align*}
We then introduce the function
	\begin{align*}
		& g(t,y)=:e^{-\nu\al^2(\f{1}{3}(U')^2t^3+t)-\nu^{\f12}\abs{\al}^{\f12}t},\quad \om^{(1)}_{H}(t,y)
		=:g(t,y)\om_{H}^{(0)}\quad t\geq 0,\\
		&(\pa^2_y-\al^2)\phi^{(j)}_{H}=\om^{(j)}_{H},\quad \phi^{(j)}_{H}(t,\pm 1)=0,\quad j=0,1.
	\end{align*}
It is straightforward to verify that $ \om^{(1)}_{H}$ satisfies the following equation
	\begin{align*}
		\pa_{t}\om^{(1)}_{H}+\nu\al^2\om^{(1)}_{H}+\big(\nu\al^2(U')^2t^2+\nu^{\f12}\abs{\al}^{\f12}\big)\om_{H}^{(1)}+i\al U(y)\om^{(1)}_{H}=i\al U''(y)g(t,y)\phi^{(0)}_{H},
	\end{align*}
which implies that
	\begin{align*}
		&\pa_{t}\om^{(1)}_{H}-\nu(\pa^2_{y}-\al^2)\om^{(1)}_{H}+i\al U(y)\om^{(1)}_{H}-i\al U''\phi^{(1)}_{H}\\
		=&-\nu\pa^2_{y}\om_{H}^{(1)}-\nu\al^2(U')^2t^2\om_{H}^{(1)}-\nu^{\f12}\abs{\al}^{\f12}\om_{H}^{(1)}+ i\al U''(g(t,y)\phi^{(0)}_{H}-\phi^{(1)}_{H}).
	\end{align*}
We now decompose $ \om_{H}$ as
	\begin{align*}
		\om_{H}=\om_{H}^{(1)}+\om_{H}^{(2)}+\om_{H}^{(3)},\quad (\pa^2_y-\al^2)\phi^{(j)}_{H}=\om^{(j)}_{H},\quad \phi^{j}_{H}(t,\pm 1)=0,\quad j=1,2,3,
	\end{align*}
	where $\om_{H}^{(2)}$ solves
	\begin{equation}\label{equ:def,wH2}
		\left\{
		\begin{aligned}
			&\pa_{t}\om_{H}^{(2)}-\nu(\pa^2_{y}-\al^2)\om_{H}^{(2)}+i\al U(y)\om_{H}^{(2)}-i\al U''\phi_{H}^{(2)}\\
			&=
			\nu\pa^2_{y}\om_{H}^{(1)}+\nu\al^2(U')^2t^2\om_{H}^{(1)}+\nu^{\f12}\abs{\al}^{\f12}\om_{H}^{(1)}- i\al U''(g(t,y)\phi^{(0)}_{H}-\phi^{(1)}_{H})\\
			&\om_{H}^{(2)}(0,y)=0,\quad \lra{\om_{H}^{(2)},\operatorname{sinh}\big(\al (1+y)\big)}=\lra{\om_{H}^{(2)},\operatorname{sinh}\big(\al (1-y)\big)}=0,
		\end{aligned}\right.
	\end{equation}
	and $\om_{H}^{(3)}$ solves
	\begin{equation}\label{equ:def,wH3}
		\left\{
		\begin{aligned}
			&\pa_{t}\om_{H}^{(3)}-\nu(\pa^2_{y}-\al^2)\om_{H}^{(3)}+i\al U(y)\om_{H}^{(3)}-i\al U''\phi_{H}^{(3)}=0,\\
			&\om_{H}^{(3)}(0,y)=0,\, \lra{\om_{H}^{(1)}+\om_{H}^{(3)},\operatorname{sinh}\big(\al (1+y)\big)}=\lra{\om_{H}^{(1)}+\om_{H}^{(3)},\operatorname{sinh}\big(\al (1-y)\big)}=0.
		\end{aligned}\right.
	\end{equation}

	{\bf Step 1.} Estimates for $\om_{H}^{(0)}$ and $\om_{H}^{(1)}$.
	
	For $\n{\om^{(0)}_{H}(t)}_{L_{t}^{\infty}L_{y}^{2}}$,
	by Lemma \ref{lemma:est,euler,w,u,l2,fH1}, we have
	\begin{align}\label{equ:est,wu,l2,win,f,H1,used}
		&\n{ \om_{H}^{(0)}}_{L_{t}^{\infty}L_{y}^2}^{2}+\abs{\al}\int_{0}^{\infty}\big(\n{\pa_y\phi_{H}^{(0)}(t)}^2_{L^2}
		+\abs{\al}^2\n{\phi_{H}^{(0)}(t)}^2_{L^2}\big)dt
		\lesssim
		\n{\om^{\mathrm{in}}}_{H_{\al}^{1}}^{2}
	\end{align}

For $\n{\om^{(1)}_{H}(t)}_{L_{t}^{\infty}L_{y}^{2}}$, noting that $\om^{(1)}_{H}(t,y) =g(t,y)\om_{H}^{(0)},\,\,g(t,y)=e^{-\nu\al^2(\f{1}{3}(U')^2t^3+t)-\nu^{\f12}\abs{\al}^{\f12}t},$ we deduce
	\begin{align}\label{equ:est,omH1,l2,om0,l2}
		\n{\om^{(1)}_{H}(t)}_{L^2}
		\lesssim e^{-\nu^{\f12}\abs{\al}^{\f12}t}\n{\om_{H}^{(0)}(t)}_{L^2}\lesssim e^{-\nu^{\f12}\abs{\al}^{\f12}t}\n{\om^{\mathrm{in}}}_{H_{\al}^{1}}.
	\end{align}

For the $L_{t}^{2}$ estimates, 	
using \eqref{est-psidec1-1} in Lemma \ref{elliptic-est-psidec12} with
	\begin{align}\label{equ:notation,lemma,elliptic,used}
		\phi_{dec,1}=\phi_{H}^{(1)},\quad \phi_{dec,2}=\phi_{H}^{(0)},\quad
		g(t,y)=e^{-\nu\al^2(\f{1}{3}(U')^2t^3+t)-\nu^{\f12}\abs{\al}^{\f12}t},
	\end{align}
	where
	\begin{align*}
	\n{g}_{L^{\infty}}\lesssim e^{-\nu^{\f12}\abs{\al}^{\f12}t} ,\quad 	 \n{g}_{C^1}\lesssim \nu^{\f14}\abs{\al}^{\f34} te^{-\f12\nu^{\f12}\abs{\al}^{\f12}t}  ,\quad \n{\pa_yg}_{H^{1}}\lesssim \nu^{\f14}\abs{\al}t^{\frac{5}{4}}e^{-\f12\nu^{\f12}\abs{\al}^{\f12}t},
	\end{align*}
	we obtain
	\begin{align}\label{equ:est,phiH1,H1,phiH1',pm1,phiH0,l2}
		&\abs{\al}^{\f12}\n{(\pa_y,\abs{\al})\phi_{H}^{(1)}(t)}_{L^2}+\abs{\pa_y\phi_{H}^{(1)}
(t,\pm1)}\\
		\lesssim& \nu^{\f14}|\al|^{\f34} te^{-\f12\nu^{\f12}\abs{\al}^{\f12}t}\abs{\al}^{\f12}
\n{(\pa_y,\abs{\al})\phi_{H}^{(0)}(t)}_{L^2}+e^{-\nu^{\f12}\abs{\al}^{\f12}t}\abs{\pa_y\phi_{H}^{(0)}(t,\pm1)}.\notag
	\end{align}
Letting
\begin{align*}
  \phi^{(0,1)}_{err}(t,y)=:g(t,y)\phi_{H}^{(0)}(t,y)-\phi_{H}^{(1)}(t,y),
\end{align*}
and using \eqref{est-psidec1-2}, we derive
\begin{align*}
	\n{(\pa_y,\abs{\al})\phi^{(0,1)}_{err}}_{L^2}
	\lesssim \nu^{\f14}e^{-\f12\nu^{\f12}
\abs{\al}^{\f12}t}\n{\al t^{\frac{5}{4}}\phi_{H}^{(0)}(t)}_{L^2},
\end{align*}
which, together with \eqref{equ:est,tkpsi,l2,w0,H1H2}, yields
\begin{align}\label{equ:definition,phi,err}
\n{e^{\ve(\nu|\al|)^{\frac{1}{2}}t}
		(\pa_y,\al)\phi_{err}^{(0,1)}}_{L_{t,y}^{2}}\lesssim \nu^{\frac{1}{4}}
\n{\al t^{\frac{5}{4}} \phi_{H}^{(0)}(t)}_{L_{t,y}^{2}}\lesssim \nu^{\frac{1}{4}}\n{\om^{\mathrm{in}}}_{H_{\al}^{2}}.
\end{align}
The estimate \eqref{equ:definition,phi,err} will subsequently be employed to bound $w_{H}^{(2)}$.

Next, putting estimates \eqref{equ:est,wu,l2,win,f,H1,used}, \eqref{equ:est,u,y=01,l2,win,f,H1} and \eqref{equ:est,kpsi,l2,w0,H2,H1} into \eqref{equ:est,phiH1,H1,phiH1',pm1,phiH0,l2}, we deduce
	\begin{align}\label{equ:est,phiH1,l2l2,om0,H1}
		\n{e^{\ve(\nu\abs{\al})^{\f12}t}u_{H}^{(1)}(t)}_{L_{t,y}^{2}}=	\n{e^{\ve(\nu\abs{\al})^{\f12}t}
(\pa_y,\abs{\al})\phi_{H}^{(1)}(t)}_{L_{t,y}^{2}}
 \lesssim \abs{\al}^{-\frac{1}{2}}\n{\om^{\mathrm{in}}}_{H_{\al}^{4}}.
	\end{align}
For $\n{\om^{(1)}_{H}(t)}_{L_{t,y}^{2}}$,  it follows directly from \eqref{equ:est,omH1,l2,om0,l2} that
	\begin{align}\label{equ:est,omH1,l2l2,om0,l2}
		\n{ e^{\ve(\nu\abs{\al})^{\f12}t}\om^{(1)}_{H}}_{L_{t,y}^{2}}
		\lesssim  \nu^{-\f14}\abs{\al}^{-\f14}\n{\om^{\mathrm{in}}}_{H_{\al}^{1}}.
	\end{align}

 For more refined $L_{t,y}^{2}$ estimates for $w_{H}^{(1)}$, we employ the pointwise estimate \eqref{equ:est,om,loc+nloc,depletion} to derive
\begin{align}\label{equ:est,omH2,I3}
	\n{e^{\ve(\nu\abs{\al})^{\f12}t}(U')^2t^4g\om_{H}^{(0)}}^2_{L_{t,y}^{2}}
	\lesssim  I_{1}+I_{2},
\end{align}
where
\begin{align}
I_{1}=&\int_0^{+\infty}\int_{-1}^1 e^{2\ve(\nu\abs{\al})^{\f12}t}(U')^{\f{15}{2}}t^8g^2(t,y)\n{\om^{\mathrm{in}}}_{H_{\al}^{4}}^{2}dy dt,\notag\\
	I_{2}=& \int_0^{+\infty}\int_{-1}^1e^{2\ve(\nu\abs{\al})^{\f12}t}(U')^4t^8g^2(t,y)\f{1}{\lra{t}^{\f74}}\n{\om^{\mathrm{in}}}_{H_{\al}^{4}}^{2} dy dt.\notag
\end{align}
For $I_{1}$, we have
\begin{align*}
I_{1}\lesssim&
	\int_{\nu^{-\f12}}^{+\infty}\int_{-1}^1 \big(\nu\abs{\al}^2(U')^2t^3\big)^{\f{15}{4}}(\nu\abs{\al}^2)^{-\frac{15}{4}}t^{-\f{13}{4}} e^{2\ve(\nu\abs{\al})^{\f12}t}g^2(t,y)\n{\om^{\mathrm{in}}}_{H_{\al}^{4}}^{2}dy dt\notag\\
	&+\int_0^{\nu^{-\f12}}\int_{-1}^1
	\big(\nu\abs{\al}^2(U')^2t^3\big)^{\f83}(\nu\abs{\al}^2)^{-\frac{8}{3}}
	e^{2\ve(\nu\abs{\al})^{\f12}t}g^2(t,y)\n{\om^{\mathrm{in}}}_{H_{\al}^{4}}^{2}dydt\notag\\
	\lesssim&
	(\nu\abs{\al}^2)^{-\frac{15}{4}}\n{\om^{\mathrm{in}}}_{H_{\al}^{4}}^{2}\int_{\nu^{-\f12}}^{+\infty} t^{-\f{13}{4}}  dt
	+(\nu\abs{\al}^2)^{-\frac{8}{3}}\n{\om^{\mathrm{in}}}_{H_{\al}^{4}}^{2}\int_0^{\nu^{-\f12}}
	dt\notag\\
	\lesssim&
	\big((\nu\abs{\al}^2)^{-\frac{15}{4}}\nu^{\f98}+(\nu\abs{\al}^2)^{-\frac{8}{3}}\nu^{-\f12}\big)\n{\om^{\mathrm{in}}}_{H_{\al}^{4}}^{2}
	\lesssim
	(v\al^{2})^{-2}\nu^{-\frac{5}{4}}|\al|^{-1}\n{\om^{\mathrm{in}}}_{H_{\al}^{4}}^{2}.\notag
\end{align*}
For $I_{2}$, with $\nu\al^{2}\lesssim 1$, we obtain
 \begin{align*}
	I_{2}\lesssim&
	(\nu \al^{2})^{-2}(\nu\abs{\al})^{-\f18} \int_0^{+\infty}\int_{-1}^1  \big(\nu\abs{\al}^2(U')^2t^3\big)^2 \big(\nu^{\f12}\abs{\al}^{\f12}t\big)^{\f14}e^{2\ve(\nu\abs{\al})^{\f12}t}g^2\n{\om^{\mathrm{in}}}_{H_{\al}^{4}}^{2}dy dt\notag\\
	\lesssim&
	(\nu \al^{2})^{-2}(\nu|\al|)^{-\frac{1}{8}}(\nu|\al|)^{-\frac{1}{2}}\n{\om^{\mathrm{in}}}_{H_{\al}^{4}}^{2}\lesssim
	(v\al^{2})^{-2}\nu^{-\frac{5}{4}}|\al|^{-1} \n{\om^{\mathrm{in}}}_{H_{\al}^{4}}^{2}.\notag
\end{align*}

Combining the estimates for $I_{1}$ and $I_{2}$, with $w_{H}^{(1)}=gw_{H}^{(0)}$,
we arrive at
\begin{align}\label{equ:est,wu,l2,win,f,H1,better,used}
\n{e^{\ve(\nu\abs{\al})^{\f12}t}(U')^2t^4\om_{H}^{(1)}}_{L_{t,y}^{2}}\lesssim (v\al^{2})^{-1}\nu^{-\frac{5}{8}}|\al|^{-\frac{1}{2}}\n{\om^{\mathrm{in}}}_{H_{\al}^{4}}.
\end{align}

{\bf Step 2.} Vector field estimates for $\om_{H}^{(0)}$ and $\om_{H}^{(1)}$.

To derive higher-order estimates for $\om_{H}^{(0)}$ and $\om_{H}^{(1)}$, we employ the vector field
 $X=\pa_y+i\al U't$, which commutes with  $\pa_{t}+i\al U$. Define
\begin{align}
  \om_{1}^{(0)}=:&X \om_{H}^{(0)},\quad \phi_{1}=:(\pa^{2}_{y}-\al^2)^{-1} \om_{1}^{(0)},\quad \phi_{2}=:X\phi_{H}^{(0)},\\
\om_{1}^{(1)}=:&X \om_{H}^{(1)}=g \om_{1}^{(0)}-\f{2}{3}\nu\al^2U'U''t^3g \om_{H}^{(0)}.\label{equ:definition,w11}
\end{align}
We observe that
\begin{align*}
\phi_{1}=\phi_{2}+\phi_{3,1}+\phi_{3,2},
\end{align*}
where
\begin{align*}
\phi_{3,1}=-i\al t(\pa^2_{y}-\al^2)^{-1}\big(U'''\phi_{H}^{(0)}\big),\quad \phi_{3,2}=- 2i\al t(\pa^2_{y}-\al^2)^{-1}\big(U''\pa_y
	\phi_{H}^{(0)}\big).
\end{align*}
Next, we derive the equation for $\om_{1}^{(0)}$
\begin{align*}
	&\pa_{t} \om_{1}^{(0)}+i\al U \om_{1}^{(0)}-i\al U''\phi_{1}
=i\al U''' \phi_{H}^{(0)}-i\al U''(\phi_{3,1}+\phi_{3,2})=:\phi_{4}.
\end{align*}
Using estimate \eqref{equ:est,wu,l2,win,f,H1} in Lemma \ref{lemma:est,euler,w,u,l2,fH1}, we obtain
	\begin{align}\label{equ:est,w1u1,l2,win,phi4,H1}
		&\sup_{t>0}\n{ \om_{1}^{(0)}(t)}^2_{L^2}+\abs{\al}\int_{0}^{+\infty}\big(\n{\pa_y\phi_{1}(t)}^2_{L^2}+\abs{\al}^2\n{\phi_{1}(t)}^2_{L^2}\big)dt\\
		\lesssim &
		\n{(\pa_y,\abs{\al}) \om_{1}^{(0)}(0)}^2_{L^2}+\abs{\al}^{-1}\int_{0}^{\infty}\big(\n{\pa_y\phi_{4}(t)}^2_{L^2}+\abs{\al}^2\n{\phi_{4}(t)}^2_{L^2}\big)dt.\notag
	\end{align}
A direct calculation shows (see for example \cite[(2.15)]{WZZ-CMP})
\begin{align}\label{equ:est,w1u1,l2,win,phi4,H1,2}
&	|\al|^{-1}\lrs{\n{\pa_y\phi_{4}(t)}_{L^2}^{2}+\abs{\al}^{2}\n{\phi_{4}(t)}_{L^2}^{2}}\\
	\lesssim&|\al|\n{(\pa_{y},\al)\phi_{H}^{(0)}}_{L^2}^{2}
+\abs{\pa_y\phi_{H}^{(0)}(t,\pm1)}^{2}+|\al|^{3}t^{2}\n{\phi_{H}^{(0)}}_{L^2}^{2}.\notag
\end{align}
Putting estimates \eqref{equ:est,w1u1,l2,win,phi4,H1,2}, \eqref{equ:est,wu,l2,win,f,H1,used}, \eqref{equ:est,u,y=01,l2,win,f,H1}, and \eqref{equ:est,tkpsi,l2,w0,H1H2}, we further obtain
	\begin{align}\label{equ:est,w1u1,l2,win,w0,H1,new}
	&\sup_{t>0}\n{ \om_{1}^{(0)}(t)}^2_{L^2}+\abs{\al}\int_{0}^{+\infty}\big(\n{\pa_y\phi_{1}(t)}^2_{L^2}+\abs{\al}^2\n{\phi_{1}(t)}^2_{L^2}\big)dt\\
	\lesssim &
	\n{(\pa_y,\abs{\al}) \om_{1}^{(0)}(0)}^2_{L^2}+\n{(\pa_y^2-\abs{\al}^2) \om^{\mathrm{in}}}^2_{L^2}\lesssim \n{\om^{\mathrm{in}}}^2_{H_{\al}^{2}},\notag
\end{align}

Using that
\begin{align}\label{equ:est,U',nu,al,tg,bound}
	|\nu\al^2U'U''t^3g(t,y)|
	&\lesssim  \nu^{-\f14}\abs{\al}^{\f14}\big(\nu\al^2(U')^2t^3\big)^{\f12}(\nu^{\f12}\abs{\al}^{\f12}t)^{\f32}g(t,y)\\
	&\lesssim \nu^{-\f14}\abs{\al}^{\f14}e^{-\nu\al^2\f{1}{6}(U')^2t^3-\f12\nu^{\f12}\abs{\al}^{\f12}t}.\notag
\end{align}
together with \eqref{equ:est,w1u1,l2,win,w0,H1,new} and \eqref{equ:est,wu,l2,win,f,H1,used}, we derive
\begin{align}\label{equ:est,om11,l2,om0,H1}
	\n{ \om^{(1)}_{1}(t)}_{L^2}
	\leq& \n{g \om_{1}^{(0)}(t)}_{L^2}+\bn{\f{2}{3}\nu\al^2U'U''t^3g \om_{H}^{(0)}(t)}_{L^2}\\
	\lesssim& e^{-\nu^{\f12}\abs{\al}^{\f12}t}\n{ \om_{1}^{(0)}(t)}_{L^2}+\nu^{-\f14}\abs{\al}^{\f14}e^{-\f12\nu^{\f12}\abs{\al}^{\f12}t}\n{ \om_{H}^{(0)}(t)}_{L^2}\notag\\
	\lesssim& \nu^{-\f14}\abs{\al}^{\f14}e^{-\f12\nu^{\f12}\abs{\al}^{\f12}t}\n{(\pa_y^2-\abs{\al}^2) \om^{\mathrm{in}}}_{L^2},\notag
\end{align}
which implies that
\begin{align}\label{equ:est,om11,l2,om0,H1,Lty2}
\n{ e^{\ve(\nu|\al|)^{\frac{1}{2}}t} \om^{(1)}_{1}(t)}_{L_{t,y}^{2}}\lesssim \nu^{-\f{1}{2}}\n{\om^{\mathrm{in}}}_{H_{\al}^{2}}.
\end{align}
 By \eqref{equ:est,w1u1,l2,win,w0,H1,new} and \eqref{equ:est,wu,l2,win,f,H1,better,used}, with $\nu\al^{2}\lesssim 1$, we obtain
\begin{align}\label{equ:est,tU'om11,l2,om0,H1}
&	\n{e^{\ve(\nu|\al|)^{\frac{1}{2}}t}tU' \om^{(1)}_{1}}_{L_{t,y}^{2}}\\
	\leq& \n{e^{\ve(\nu|\al|)^{\frac{1}{2}}t}tU'g \om_{1}^{(0)}}_{L_{t,y}^2}+\nu\al^2\bn{e^{\ve(\nu|\al|)^{\frac{1}{2}}t}(U')^2U''t^4g \om_{H}^{(0)}}_{L_{t,y}^2}\notag\\
	\lesssim& \nu^{-\f13}\abs{\al}^{-\f23}\n{e^{-\nu^{\f12}\abs{\al}^{\f12}t}}_{L_{t}^{2}}\n{ \om_{1}^{(0)}(t)}_{L_{t}^{\infty}L_{y}^2}+\nu^{-\frac{5}{8}}|\al|^{-\frac{1}{2}}\n{\om^{\mathrm{in}}}_{H_{\al}^{4}}\notag\\
\lesssim& \nu^{-\frac{5}{8}}|\al|^{-\frac{1}{2}}\n{\om^{\mathrm{in}}}_{H_{\al}^{4}},\notag
\end{align}
where we have used the following inequalities
\begin{align*}
	|tU'(y)g(t,y)|
	\lesssim&  \nu^{-\f13}\abs{\al}^{-\f23}\big(\nu\al^2(U')^2t^3\big)^{\f13}g(t,y)
	\lesssim \nu^{-\f13}\abs{\al}^{-\f23}e^{-\nu^{\f12}\abs{\al}^{\f12}t},\\
	|\nu\al^2(U')^2U''t^4g(t,y)|
\lesssim&
 \nu^{-\f12}\abs{\al}^{-\f12}\big(\nu\al^2(U')^2t^3\big)(\nu^{\f12}\abs{\al}^{\f12}t)g(t,y)
\lesssim \nu^{-\f12}\abs{\al}^{-\f12}e^{-\f12\nu^{\f12}\abs{\al}^{\f12}t}.
\end{align*}

{\bf Step 3.} Estimates for $ \om_{H}^{(2)}$.

Given the definition \eqref{equ:definition,w11} of $ \om_{1}^{(1)}$, we derive that $\pa_{y} \om_{H}^{(1)}= \om^{(1)}_{1}-i\al U't \om_{H}^{(1)}$, which gives
\begin{align*}
	\pa_{y}^{2} \om_{H}^{(1)}&=\pa_{y}( \om^{(1)}_{1}-i\al U't \om_{H}^{(1)})=\pa_{y} \om^{(1)}_{1}-i\al U''t \om_{H}^{(1)}- i\al U't\pa_{y} \om_{H}^{(1)}\\
	&=\pa_{y} \om^{(1)}_{1}-i\al U''t \om_{H}^{(1)}- i\al U't( \om^{(1)}_{1}-i\al U't \om_{H}^{(1)}).
\end{align*}
Then, using equation \eqref{equ:def,wH2} for $\om_{H}^{(2)}$ and the definition \eqref{equ:definition,phi,err} of $\phi^{(0,1)}_{err}=g\phi_{H}^{(0)}-\phi_{H}^{(1)}$,
 we get
\begin{align}\label{equ:equation,wH2,more}
	&\pa_{t} \om_{H}^{(2)}-\nu(\pa^2_{y}-\al^2) \om_{H}^{(2)}+i\al U \om_{H}^{(2)}-i\al U''\phi_{H}^{(2)}\\
	=&
	\nu\pa^2_{y} \om_{H}^{(1)}+\nu\al^2(U')^2t^2 \om_{H}^{(1)}+\nu^{\f12}\abs{\al}^{\f12} \om_{H}^{(1)}-i\al U''(g(t,y)\phi^{(0)}_{H}-\phi^{(1)}_{H})\notag\\
	=&	\nu\big(\pa_{y} \om^{(1)}_{1}-i\al U''t \om_{H}^{(1)} - i\al U't \om_{1}^{(1)}\big)+\nu^{\f12}\abs{\al}^{\f12}(\pa_{y}^{2}-\al^{2})\phi_{H}^{(1)}- i\al U''\phi^{(0,1)}_{err}.\notag
\end{align}

For $w_{H}^{(2)}$, applying Proposition \ref{proposition:est,u,timespace,L2L2,LinftyL2,f1234L2} with
 \begin{align*}
f_1=&\nu^{\f12}\abs{\al}^{\f32}\phi_H^{(1)},\quad f_2=\nu\om_{1}^{(1)}+\nu^{\f12}\abs{\al}^{\f12}\pa_y\phi_H^{(1)},\\
 f_3=&-i\nu\al U't\om_{1}^{(1)}-i\nu\al U''t\om_{H}^{(1)},\quad f_4=-i\al U''\phi^{(0,1)}_{err},
 \end{align*}
we obtain
\begin{align*}
	\n{e^{\ve(\nu|\al|)^{\frac{1}{2}}t} \om_{H}^{(2)}}_{L_{t,y}^{2}}
	\lesssim &
	\nu^{\f14}\abs{\al}^{-\f14}\n{e^{\ve(\nu|\al|)^{\frac{1}{2}}t} \om_{1}^{(1)}}_{L_{t,y}^{2}}+\nu^{-\f14}\abs{\al}^{\f14}\n{e^{\ve(\nu|\al|)^{\frac{1}{2}}t}(\pa_y,|\al|)
\phi_H^{(1)}}_{L_{t,y}^{2}}\\
&+\nu^{\f38}\abs{\al}^{\f18}\n{e^{\ve(\nu|\al|)^{\frac{1}{2}}t}tU' \om_{1}^{(1)}}_{L_{t,y}^{2}}
	+\nu^{\f18}\abs{\al}^{-\f38}\n{e^{\ve(\nu|\al|)^{\frac{1}{2}}t}(\nu^{\f12}\abs{\al}^{\f12}t) \om_{H}^{(1)}}_{L_{t,y}^{2}}\\
&+\nu^{-\f38}\abs{\al}^{\f38}\n{e^{\ve(\nu|\al|)^{\frac{1}{2}}t}
		(\pa_y,\al)\phi_{err}^{(0,1)}}_{L_{t,y}^{2}},
\end{align*}
which, along with estimates
\eqref{equ:est,om11,l2,om0,H1,Lty2}, \eqref{equ:est,phiH1,l2l2,om0,H1},
\eqref{equ:est,tU'om11,l2,om0,H1}, \eqref{equ:est,omH1,l2l2,om0,l2}, \eqref{equ:definition,phi,err}, gives
\begin{align}\label{equ:est,omH2,l2,om0,H2}
	\n{e^{\ve(\nu|\al|)^{\frac{1}{2}}t} \om_{H}^{(2)}}_{L_{t,y}^{2}}
	\lesssim& \big(\nu^{-\f14}\abs{\al}^{-\f14}+\nu^{-\f18}
\abs{\al}^{-\f{5}{8}}\big)\n{\om^{\mathrm{in}}}_{H_{\al}^{4}}
+\nu^{-\frac{1}{8}}|\al|^{\frac{3}{8}}\n{\om^{\mathrm{in}}}_{H_{\al}^{2}}\\
\lesssim& \nu^{-\f14}\abs{\al}^{-\f14}\n{\om^{\mathrm{in}}}_{H_{\al}^{4}}.\notag
\end{align}

Similarly, by using Proposition \ref{proposition:est,u,timespace,L2L2,LinftyL2,f1234L2} with
\begin{align*}
f_1=\nu \om_{1}^{(1)},\ f_2=0,\ f_3=-i\nu\al U't \om_{1}^{(1)}-i\nu\al U''t \om_{H}^{(1)}+\nu^{\f12}\abs{\al}^{\f12} \om_{H}^{(1)},\ f_4=-i\al U''\phi^{(0,1)}_{err},
\end{align*} we obtain
\begin{align*}
	&\n{e^{\ve(\nu|\al|)^{\frac{1}{2}}t}(\pa_y,\al)\phi_{H}^{(2)}}_{L_{t,y}^{2}}\\
	\lesssim &
	\nu^{-\f12}\abs{\al}^{-\f12}\n{e^{\ve(\nu|\al|)^{\frac{1}{2}}t}\nu \om_{1}^{(1)}}_{L_{t,y}^{2}}+\nu^{-\f14}\abs{\al}^{-\f34}\n{e^{\ve(\nu|\al|)^{\frac{1}{2}}t}\nu\al U't \om_{1}^{(1)}}_{L_{t,y}^{2}}\\
	&+\nu^{-\f14}\abs{\al}^{-\f34}\n{e^{\ve(\nu|\al|)^{\frac{1}{2}}t}\big(\nu\abs{\al} U''t+\nu^{\f12}\abs{\al}^{\f12}\big) \om_{H}^{(1)}}_{L_{t,y}^{2}}+\abs{\al}^{-1}\n{\al U''(\pa_{y},\abs{\al})\phi_{err}^{(0,1)}}_{L_{t,y}^{2}}\\
	\lesssim &
	\nu^{\f12}\abs{\al}^{-\f12}\n{e^{\ve(\nu|\al|)^{\frac{1}{2}}t} \om_{1}^{(1)}}_{L_{t,y}^{2}}+\nu^{\f34}\abs{\al}^{\f14}\n{e^{\ve(\nu|\al|)^{\frac{1}{2}}t}tU' \om_{1}^{(1)}}_{L_{t,y}^{2}},\\
	&+\nu^{\f14}\abs{\al}^{-\f14}\n{e^{\ve(\nu|\al|)^{\frac{1}{2}}t}\big(\nu^{\f12}\abs{\al}^{\f12}t+1\big) \om_{H}^{(1)}}_{L_{t,y}^{2}}+\n{e^{\ve(\nu|\al|)^{\frac{1}{2}}t}
		(\pa_y,\al)\phi_{err}^{(0,1)}}_{L_{t,y}^{2}},
\end{align*}
which, along with estimates \eqref{equ:est,om11,l2,om0,H1,Lty2},
\eqref{equ:est,tU'om11,l2,om0,H1}, \eqref{equ:est,omH1,l2l2,om0,l2}, \eqref{equ:definition,phi,err}, and $\nu\al^{2}\lesssim 1$, gives
\begin{align}\label{equ:est,uH2,l2,om0,H2}
	&\n{e^{\ve(\nu|\al|)^{\frac{1}{2}}t}(\pa_y,\al)\phi_{H}^{(2)}}_{L_{t,y}^{2}}\\
	\lesssim&
	\big(\abs{\al}^{-\f12}+\nu^{\f18}\abs{\al}^{-\f14}\big)\n{ \om^{\mathrm{in}}}_{H_{\al}^{4}}+\nu^{\frac{1}{4}}\n{\om^{\mathrm{in}}}_{H_{\al}^{2}}
	\lesssim
	\abs{\al}^{-\f12}\n{\om^{\mathrm{in}}}_{H_{\al}^{4}}.\notag
\end{align}

{\bf Step 4.} Estimates for $ \om_{H}^{(3)}$.

We introduce the following functions
\begin{equation}\label{equ:notations,gamma}
\left\{
\begin{aligned}
	&\gamma^{+}(y)=\frac{\operatorname{sinh}\big(\al(1+y)\big)}{\operatorname{\operatorname{sinh}}(2\al)},\qquad
	\gamma^{-}(y)=
	\frac{\operatorname{\operatorname{sinh}}\big(\al(1-y)\big)}{\operatorname{\operatorname{sinh}}(2\al)},\\
	&a_{1}(t)=\lra{\gamma^{-}, \om_{H}^{(1)}}=-\lra{\gamma^{-}, \om_{H}^{(3)}}=-\pa_{y}\phi_{H}^{(1)}(t,-1),\\
	&a_{2}(t)=\lra{\gamma^{+}, \om_{H}^{(1)}}=-\lra{\gamma^{+}, \om_{H}^{(3)}}=\pa_{y}\phi_{H}^{(1)}(t,1).
\end{aligned}
\right.
\end{equation}
From \eqref{equ:est,phiH1,H1,phiH1',pm1,phiH0,l2}, we obtain
\begin{align}\label{equ:est,a1a2,phiH0}
	\abs{a_{1}(t)}+\abs{a_{2}(t)}
	\lesssim& \nu^{\f14}|\al|^{\f34} te^{-\f12\nu^{\f12}\abs{\al}^{\f12}t}\abs{\al}^{\f12}
\n{(\pa_y,\abs{\al})\phi_{H}^{(0)}(t)}_{L^2}+e^{-\nu^{\f12}\abs{\al}^{\f12}t}\abs{\pa_y\phi_{H}^{(0)}(t,\pm1)}.
\end{align}

Next, we define
\begin{align*}
	\wt{\om}(\la,y)=&\int_{0}^{+\infty} w_{H}^{(3)}(t,y)e^{-it\la +\ve\nu^{\f12}\abs{\al}
	^{\f12}t}dt,\\
	\wt{\psi}(\la,y)=&\int_{0}^{+\infty}\phi_{H}^{(3)}(t,y)e^{-it\la +\ve\nu^{\f12}\abs{\al}
	^{\f12}t}dt,\\
	c_{1}(\la)=&\int_{0}^{+\infty}a_{1}(t)e^{-it\la +\ve\nu^{\f12}\abs{\al}
	^{\f12}t}dt=\lra{\ga^{-},\wt{\om}(\la)},\\
c_{2}(\la)=&\int_{0}^{+\infty}a_{2}(t)e^{-it\la +\ve\nu^{\f12}\abs{\al}
	^{\f12}t}dt=\lra{\ga^{+},\wt{\om}(\la)}.
\end{align*}
We then turn to consider the following system
\begin{align*}
	&\big(i\la-\ve\nu^{\f12}\abs{\al}
	^{\f12}-\nu(\pa_{y}^{2}-\abs{\al}^{2})+i\al U \big)\wt{\om}(\la,y)-i\al U''\wt{\psi}(\la,y)=0.
\end{align*}
Thus, we have
\begin{align*}
	\wt{\om}(\la,y)=-c_{1}(\la)w_{cor,1}({-\la}\abs{\al}^{-1},y)-c_{2}(\la)w_{cor,2}({-\la}\abs{\al}^{-1},y),
\end{align*}
where $w_{cor,1}$ and $w_{cor,2}$ are defined in
\eqref{equ:corrector1,NS,per}, and $-\la$ are replaced by $\la/\al$.

\textbf{Claim:}
\begin{align}\label{equ:est,c1,c2,om0,H1}
	\n{(1+\abs{\la})^{\frac{1}{2}}c_{1}(\la)}^2_{L^{2}_{\la}(\R)}+\n{(1+\abs{\la})^{\frac{1}{2}}c_{2}(\la)}^2_{L^{2}_{\la}(\R)}
	\lesssim \abs{\al}\n{\om^{\mathrm{in}}}_{H_{\al}^{2}}^{2}.
\end{align}
We postpone the proof of the claim to Step 5.

Now, by Lemma \ref{lemma:est,uw,L2Linfty,cor12}, we have
\begin{align*}
\n{\wt{u}(\la)}_{L_{y}^2}\lesssim& \abs{c_1}\n{u_{cor,1}}_{L_{y}^{2}}+\abs{c_2}\n{u_{cor,2}}_{L_{y}^{2}}\lesssim\nu^{\f16}\abs{\al}^{-\f16}\big(\abs{c_{1}(\la)}+\abs{c_{2}(\la)}\big),\\
	\n{\wt{\om}(\la)}_{L_{y}^{2}}\lesssim &
	\abs{c_1}\n{w_{cor,1}}_{L_{y}^{2}}+\abs{c_2}\n{w_{cor,2}}_{L_{y}^{2}}
	\lesssim \nu^{-\f14}(1+\abs{\la})^{\f14}\big(\abs{c_{1}(\la)}+\abs{c_{2}(\la)}\big).
\end{align*}
Therefore, with \eqref{equ:est,c1,c2,om0,H1} and $\nu\al^{2}\lesssim 1$, we get
\begin{align*}
	\n{u_{H}^{(3)}}^2_{L_{t,y}^{2}}
	\sim \bn{\n{\wt{u}(\la)}_{L_{y}^{2}}}^2_{L_{\la}^{2}}
	\lesssim \nu^{\f13}\abs{\al}^{-\f13}\abs{\al}
	\n{\om^{\mathrm{in}}}_{H_{\al}^{2}}^2\lesssim
	\n{\om^{\mathrm{in}}}_{H_{\al}^{2}}^2.
\end{align*}
This, together with \eqref{equ:est,phiH1,l2l2,om0,H1} and \eqref{equ:est,uH2,l2,om0,H2}, implies
\begin{align}\label{equ:est,uH,l2,om0,H2}
	\n{e^{\ve(\nu\abs{\al})^{\f12}t}u_{H}}^2_{L_{t,y}^{2}}=\n{(\pa_y,\abs{\al})\phi_{H}}^2_{L_{t,y}^{2}}\lesssim \abs{\al}^{-1}
	\n{\om^{\mathrm{in}}}_{H_{\al}^{4}}^{2}.
\end{align}
Similarly, we also have
\begin{align*}
	\nu^{\f12}\abs{\al}^{\f12}\n{e^{\ve(\nu\abs{\al})^{\f12}t} \om_{H}^{(3)}}^2_{L_{t,y}^{2}}
	\sim \nu^{\f12}\abs{\al}^{\f12}\bn{\n{ \wt{\om}(\la)}_{L^2}}^2_{L_{\la}^{2}}
	\lesssim& \abs{\al}^{\f32}\n{ \om^{\mathrm{in}}}^2_{H_{\al}^{2}}\lesssim \n{\om^{\mathrm{in}}}_{H_{\al}^{4}}^{2},
\end{align*}
which along with \eqref{equ:est,omH1,l2l2,om0,l2} and \eqref{equ:est,omH2,l2,om0,H2} gives
\begin{align}\label{equ:est,omH,l2,om0,H2}
	\n{ e^{\ve(\nu\abs{\al})^{\f12}t}w_{H}}^2_{L_{t,y}^{2}}
	\lesssim \nu^{-\f12}\abs{\al}^{-\f12}
	\n{\om^{\mathrm{in}}}_{H_{\al}^{4}}^{2}.
\end{align}

{\bf Step 5.} Proof of Claim \eqref{equ:est,c1,c2,om0,H1}.

It suffices to consider the estimate for $c_{1}(\la)$, as the estimate for $c_{2}(\la)$ can be derived similarly.
We recall that $a_{1}(t)=\lra{\gamma^{-}, \om_{H}^{(1)}(t)}$ and
\begin{align*}
	\big(\pa_{t}+i\al U+\nu\al^2(U')^2t^2+\nu\al^2+ \nu^{\f12}\abs{\al}^{\f12}     \big) \om_{H}^{(1)}=i\al U''g(t,y)\phi_{H}^{(0)}.
\end{align*}
Taking the time derivative of $a_{1}(t)$, we have
\begin{align*}
        \abs{\pa_{t}a_{1}(t)}=\abs{\lra{\gamma^{-},\pa_{t} \om_{H}^{(1)}}}\lesssim I_{3}+I_{4}+I_{5},
\end{align*}
where
\begin{align*}
I_{3}=&\babs{\lra{\gamma^{-},\al \big(U(y)-U(1)\big) \om_{H}^{(1)}}},\notag\\
I_{4}=& \babs{\lra{\gamma^{-},\big(\nu\al^2(U')^2t^2+\nu\al^2+\nu^{\f12}\abs{\al}^{\f12}+ \al U(1) \big) \om_{H}^{(1)}}},\notag\\
 I_{5}=&\babs{\lra{\gamma^{-},i\al U''g(t,y)\phi_{H}^{(0)}}}.
\end{align*}

For \(I_{3}\), noticing that $\gamma^{-}\big(U(y)-U(1)\big)\big|_{y=\pm1}=0$ and using \eqref{equ:est,phiH1,H1,phiH1',pm1,phiH0,l2}, we get
\begin{align*}
        I_{3}\lesssim& \abs{\al}\n{(\pa_y,\abs{\al})\phi_{H}^{(1)}}_{L^2}\n{(\pa_y,\abs{\al})\big(\gamma^{-}(U(y)-U(1))\big)}_{L^2}
        \lesssim\abs{\al}^{\f12}\n{(\pa_y,\abs{\al})
         \phi_{H}^{(1)}}_{L^2}\\
         \lesssim&
         \nu^{\f14}|\al|^{\f34} te^{-\f12\nu^{\f12}\abs{\al}^{\f12}t}\abs{\al}^{\f12}
          \n{(\pa_y,\abs{\al})\phi_{H}^{(0)}(t)}_{L^2}+e^{-\nu^{\f12}\abs{\al}^{\f12}t}\abs{\pa_y\phi_{H}^{(0)}(t,\pm1)},
\end{align*}
where we used the fact that
\begin{align*}
	\n{(\pa_y,\abs{\al})
		\big(\gamma^{-}(U(y)-U(1))\big)}_{L^2}
		\lesssim \abs{\al}^{-\f12}.
\end{align*}

By Lemma \ref{lemma:proper,symme,flow}, we have
\begin{align*}
	(U'(y))^2\lesssim \abs{U(y)-U(0)}\lesssim \abs{U(y)-U(1)}+\abs{U(1)-U(0)}.
\end{align*}
Similarly, for $I_{4}$, we obtain
\begin{align*}
	I_{4}\lesssim&
	\nu\al^2\abs{\lra{\gamma^{-},\big(U(y)-U(1)\big)t^2 \om_{H}^{(1)}}}+\nu\al^2\abs{\lra{\gamma^{-},\big(U(1)-U(0)\big)t^2 \om_{H}^{(1)}}}\\
	&+(\nu\al^2+\nu^{\f12}\abs{\al}^{\f12}+\al U(1))
	\abs{\lra{\gamma^{-}, \om_{H}^{(1)}}}\\
	\lesssim &
	\nu\al^2 t^2\abs{\al}^{-\f12}\n{(\pa_y,\abs{\al})\phi_{H}^{(1)}}_{L^2}+(\nu\al^2 t^2+\nu\al^2+\nu^{\f12}\abs{\al}^{\f12}+ \abs{\al})\abs{a_{1}(t)},
\end{align*}
which together with \eqref{equ:est,phiH1,H1,phiH1',pm1,phiH0,l2}, \eqref{equ:est,a1a2,phiH0} and the fact \(\nu|\al|^2\lesssim 1\), gives
\begin{align*}
     I_{4}\lesssim& (1+\nu\al^2 t^2\abs{\al}^{-\f12})\n{(\pa_y,\abs{\al})\phi_{H}^{(1)}}_{L^2}+(\nu\al^2 t^2+\nu\al^2+\nu^{\f12}\abs{\al}^{\f12}+ \abs{\al})\abs{a_{1}(t)}\\
	\lesssim& (\nu|\al|^2t^2+|\al|)e^{-\f12\nu^{\f12}\abs{\al}^{\f12}t}
       \big(\nu^{\f14}|\al|^{\f34} t\abs{\al}^{\f12}
       \n{(\pa_y,\abs{\al})\phi_{H}^{(0)}(t)}_{L^2}+\abs{\pa_y\phi_{H}^{(0)}(t,\pm1)}\big).
\end{align*}
Since $\n{\ga^{-}}_{L^2}\lesssim |\al|^{-\f12}$, we bound $I_{5}$ as follow
\begin{align*}
	I_{5}\lesssim |\al|e^{-\nu^{\f12}|\al|^{\f12}t}\n{\ga^{-}}_{L^2}\n{\phi_H^{(0)}}_{L^2}\lesssim |\al|^{\f12}e^{-\nu^{\f12}|\al|^{\f12}t}\n{\phi_H^{(0)}}_{L^2}.
\end{align*}

Summing up the above estimates, we arrive at
\begin{align*}
	\abs{\pa_{t} a_{1}(t)}
	\lesssim&
    \nu^{\f14}|\al|^{\f94} te^{-\f14\nu^{\f12}\abs{\al}^{\f12}t}
\n{(\pa_y,\abs{\al})\phi_{H}^{(0)}(t)}_{L^2}+|\al|\abs{\pa_y\phi_{H}^{(0)}(t,\pm1)}+
    \n{\al\phi_H^{(0)}}_{L^2}.
\end{align*}
Applying estimates \eqref{equ:est,wu,l2,win,f,H1,used} and  \eqref{equ:est,u,y=01,l2,win,f,H1}, with $\nu\al^{2}\lesssim 1$, we get
\begin{align*}
	\n{\pa_{t}a_{1}(t)}^2_{L^{2}(\R_{+})}\lesssim (|\al|^{3}+|\al|^2+1)\n{ \om^{\mathrm{in}}}^2_{H_{\al}^{1}}\lesssim \abs{\al}^{2}
\n{\om^{\mathrm{in}}}_{H_{\al}^{2}}^{2}.
\end{align*}
Therefore, we have
\begin{align*}
	\n{(1+\abs{\la})^{\frac{1}{2}}c_{1}(\la)}_{L_{\la}^2}\lesssim &
 \n{(1+\abs{\la})c_{1}(\la)}_{L_{\la}^2}^{\frac{1}{2}} \n{c_{1}(\la)}_{L_{\la}^2}^{\frac{1}{2}}\\
 \sim& \n{a_1(t)}_{H^{1}(\R_{+})}^{\frac{1}{2}}\n{a_{1}(t)}^{\f12}_{L^{2}(\R_{+})}\lesssim \abs{\al}^{\f12}
 \n{\om^{\mathrm{in}}}_{H_{\al}^{2}},
\end{align*}
which completes the proof of \eqref{equ:est,c1,c2,om0,H1}.
\end{proof}

\subsection{Space-time estimates for the full problem}\label{sec:The proof of Proposition}
In this subsection, we aim to prove Proposition \ref{proposition:est,u,om,timespace,f1234}. To achieve this, we first establish the following important lemma.

\begin{lemma}\label{lemma:est,u,om,timespace,f1234}
	Let $(\om,\phi)$ be the solution to \eqref{equ:linear,nonslip,NS,fj} with $\lra{\om^{\mathrm{in}},e^{\pm \abs{\al}y}}=0$ and $f_3=f_4=0$. Then it holds that
		\begin{align}\label{equ:est,u,om,weight,f1234}
		&\n{e^{\ve \nu^{\f12}t}u}^2_{L_{t}^{\infty}L_{y}^{2}}+\nu\n{e^{\ve \nu^{\f12}t}\om}^2_{L_{t,y}^{2}}\\
		\lesssim& \n{u^{\mathrm{in}}}^2_{L_{y}^2}+(1+\ve\nu^{\f12})\n{e^{\varepsilon \nu^{\f12}t}u}^2_{L_{t,y}^{2}}+(\nu\abs{\al}^2)^{-1}\n{e^{\varepsilon \nu^{\f12}t}(f_1,f_2)}^2_{L_{t,y}^{2}}\notag
	\end{align}
and
	\begin{align}\label{equ:est,om,weight,f1234}
	&\n{e^{\varepsilon \nu^{\f12}t}\om}^2_{L_{t}^{\infty}L_{y}^{2}}+\nu\n{e^{\ve \nu^{\f12}t}(\pa_y,\abs{\al})\om}^2_{L_{t,y}^{2}}\\
	\lesssim&
	\n{\om^{\mathrm{in}}}^2_{L^2}+\nu^{-1}
	\n{e^{\ve \nu^{\f12}t}(f_1,f_2)}^2_{L_{t,y}^{2}}+(\nu^{\f12}+\nu\abs{\al}^2)\n{e^{\varepsilon \nu^{\f12}t}\om}^2_{L_{t,y}^{2}}\notag\\
	&
	+\nu^{-\f13}\abs{\al}^{\f23}
	\n{e^{\varepsilon \nu^{\f12}t}u}^{\f43}_{L_{t,y}^{2}}\n{e^{\varepsilon \nu^{\f12}t}\om}^{\f23}_{L_{t,y}^{2}}.\notag
\end{align}
\end{lemma}

\begin{proof}
	Taking the inner product of \eqref{equ:linear,nonslip,NS,fj} with $-\phi$, and taking the real part, we have
\begin{align*}
	\f12\f{d}{dt}\n{u}^2_{L^2}+\nu\n{\om}_{L^2}^2
	=&\mathrm{Re}\Big(-i\al\int_{-1}^1 U'\phi'\ol{\phi}dy\Big)+\mathrm{Re}\lrs{\lra{i\al f_1+\pa_yf_2,\phi}}\\
	\lesssim &\n{(\pa_y,\abs{\al})\phi}_{L^2}^2
	+\n{(f_1,f_2)}_{L^2}\n{(\pa_y,\abs{\al})\phi}_{L^2},
\end{align*}
which, together with Lemma \ref{lemma:proper,u,L2,Linf}, implies
\begin{align}\label{equ:est,dt,ul2,1}
	\f{d}{dt}\n{e^{\varepsilon \nu^{\f12}t}u}^2_{ L^2}+\f32\nu\n{e^{\varepsilon \nu^{\f12}t}\om}^2_{ L^2}
	\lesssim(1+\ve\nu^{\f12})\n{e^{\varepsilon \nu^{\f12}t}u}^2_{L^2}+
	(\nu\abs{\al}^2)^{-1}\n{e^{\varepsilon \nu^{\f12}t}(f_1,f_2)}^2_{L^2}
\end{align}
which is enough to conclude \eqref{equ:est,u,om,weight,f1234}.

Next, we get into the analysis of $L_{t}^{\infty}L_{y}^{2}$ estimate \eqref{equ:est,om,weight,f1234}.
We test equation \eqref{equ:linear,nonslip,NS,fj} with $\f{\om}{U''}$
and take the real part to obtain
\begin{align*}
	&\f12\f{d}{dt}\bn{\f{\om}{\sqrt{U''}}}^2_{L^2}+\nu\bn{(\pa_y,\abs{\al})\f{\om}
		{\sqrt{U''}}}^2_{L^2}\\
	=&
	\operatorname{Re}\lrs{   \blra{\nu \pa_y\om,\om\lrs{\frac{1}{U''}}'}+  \lra{i\al f_1,\f{\om}{U''}}-\lra{f_2,\pa_y(\f{\om}{U''})}+(\nu\pa_y\om+f_2)\f{\ol{\om}}{U''}\Big|_{-1}^1
	},
\end{align*}
which gives
\begin{align}\label{equ:energy,time,w,test}
	     &\f12\f{d}{dt}\n{\f{\om}{\sqrt{U''}}}^2_{L^2}
          +\f{7\nu}{8}\n{(\pa_y,\abs{\al})\f{\om}{\sqrt{U''}}}^2_{L^2}\\
           \lesssim &\nu\n{\om}^2_{L^2}+
	     \n{(f_1,f_2)}\n{(\pa_y,\abs{\al})\om}_{L^2}
          +\abs{(\nu\pa_y\om+f_2)(\pm 1)}\n{\om}
          _{L^\infty}.\notag
\end{align}

To deal with the boundary term, we recall that
\begin{align*}
	\gamma_{1}^{+}(y)=\frac{\operatorname{sinh}\big(\al(1+y)\big)}{\operatorname{\operatorname{sinh}}(2\al)}\quad \operatorname{and}\quad
	\gamma_{1}^{-}(y)=
	\frac{\operatorname{\operatorname{sinh}}\big(\al(1-y)\big)}{\operatorname{\operatorname{sinh}}(2\al)}.
\end{align*}
Since $\lra{\pa_{t}^j\om,\gamma_{1}^{+}}=\lra{\pa_{t}^j\om,\gamma_{1}^{-}}=0$ for $j=0,1$, we get
\begin{align*}
	    0=&\lra{\pa_{t}\om,\gamma_{1}^{+}}=\lra{\nu(\pa^2_y-\al^2)\om-i\al U\om+i\al U''\phi+i\al f_1+\pa_yf_2
           ,\gamma_{1}^{+}}\\
	   =&\lra{\nu\om,(\pa^2_y-\al^2)\gamma_{1}^{+}}+
	     \lra{-i\al U\om,\gamma_{1}^{+}}+\lra{i\al f_1,\gamma_{1}^{+}}\\
	     &+\lra{i\al U''\phi,\gamma_{1}^{+}}+\lra{-f_2,\pa_y\gamma_{1}^{+}}+\big(-\nu\om\pa_y\gamma_{1}^{+}+(\nu\pa_y\om+f_2)\gamma_{1}^{+}\big)\big|_{-1}^1\\
	    =&\lra{-i\al U\om,\gamma_{1}^{+}}+\lra
	     {i\al f_1,\gamma_{1}^{+}}+\lra{i\al U''\phi,\gamma_{1}^{+}}\\
	    &+\lra{-f_2,\pa_y\gamma_{1}^{+}}+\big(-\nu\om\pa_y\gamma_{1}^{+}+(\nu\pa_y\om+f_2)\gamma_{1}^{+}\big)\big|_{-1}^1.
\end{align*}
Noting that
\begin{align*}
	\lra{-i\al U\om,\gamma_{1}^{+}}+\lra{i\al U''\phi,\gamma_{1}^{+}}=\lra{-2i\al U'\phi,\pa_y\gamma_{1}^{+}},
\end{align*}
we derive
\begin{align}\label{equ:energy,time,w,boundary,y=1}
	    &\babs{(\nu\pa_y\om+f_2)(1)}\\
	    =&\babs{\lra{-2i\al U'\phi,\pa_y\gamma_{1}^{+}}
		+\lra{i\al f_1,\gamma_{1}^{+}}
		+\lra{-f_2,\pa_y\gamma_{1}^{+}}+(-\nu\om\pa_y\gamma_{1}^{+})|_{-1}^1}\notag\\
	    \lesssim &\abs{\al}\n{\phi}_{L^2}
        (\n{\pa_y\gamma_{1}^{+}}_{L^2}+
         \n{\gamma_{1}^{+}}_{L^2})+\n{\al f_1}_{L^2}\n{\gamma_{1}^{+}}_{L^2}\notag\\
	    &+\n{f_2}_{L^2}\n{\pa_y\gamma_{1}^{+}}_{L^2}
        +\nu\n{\pa_y\gamma_{1}^{+}(\pm 1)}\n{\om}_{L^\infty}\notag\\
	    \lesssim&\abs{\al}^{\f32}\n{\phi}_{L^2}
         +\abs{\al}^{\f12}\n{(f_1,f_2)}_{L^2}
         +\nu\abs{\al}\n{\om}_{L^\infty},\notag
\end{align}
where in the last inequality we have used that
\begin{align*}
	\abs{\pa_y\gamma_{1}^{+}(\pm1)}&=\abs{\al}\operatorname{coth}(2\abs{\al})\lesssim \abs{\al},\\
	\n{(\pa_y,\abs{\al})\gamma_{1}^{+}}^2_{L^2}&=
	-\lra{\gamma_{1}^{+},(\pa_y^2-\abs{\al}^2)\gamma_{1}^{+}}+(\pa_y\gamma_{1}^{+}\cdot \gamma_{1}^{+})\big|_{-1}^1\lesssim \abs{\al}.
\end{align*}
Similarly, we also have
\begin{align}\label{equ:energy,time,w,boundary}
	\babs{(\nu\pa_y\om+f_2)|_{y=-1}}
	\lesssim
	\abs{\al}^{\f32}\n{\phi}_{L^2}+\abs{\al}^{\f12}
	\n{(f_1,f_2)}_{L^2}+\nu\abs{\al}\n{\om}_{L^\infty}.
\end{align}

Combining \eqref{equ:energy,time,w,test}, \eqref{equ:energy,time,w,boundary,y=1}, and \eqref{equ:energy,time,w,boundary}, we conclude
\begin{align*}
	&\f{d}{dt}\n{\f{\om}{\sqrt{U''}}}_{L^2}^2
+\f32\nu\n{(\pa_y,\abs{\al})\f{\om}{\sqrt{U''}}}^2_{L^2}\\
	\lesssim &
	\nu\n{\om}^2_{L^2}+	\nu^{-1}\n{(f_1,f_2)}^2_{L^2}
	+\big(\abs{\al}^{\f32}\n{\phi}_{L^2}+\abs{\al}^{\f12}
	\n{(f_1,f_2)}_{L^2}+\nu\abs{\al}\n{\om}_{L^\infty}
\big)\n{\om}_{L^\infty},
\end{align*}
which, together with Cauchy-Schwarz and interpolation inequalities, gives
\begin{align}\label{equ:energy,time,w,test,final}
	&\f{d}{dt}\n{e^{\varepsilon \nu^{\f12}t}\f{\om}{\sqrt{U''}}}_{ L^2}^2+\nu\n{e^{\varepsilon \nu^{\f12}t}(\pa_y,\abs{\al})\f{\om}{\sqrt{U''}}}^2_{L^2}\\
	\lesssim &
	\nu^{-1}\n{e^{\varepsilon \nu^{\f12}t}(f_1,f_2)}^2_{L^2}
	+
	(\nu^{\f12}+\nu\abs{\al}^2)\n{e^{\varepsilon \nu^{\f12}t}\om}^2_{L^2}+\nu^{-\f13}\abs{\al}^{\f23}\n{e^{\varepsilon \nu^{\f12}t}u}_{L^2}^{\f43}\n{e^{\varepsilon \nu^{\f12}t}\om}^{\f23}_{L^2}.\notag
\end{align}
Integrating the above inequality \eqref{equ:energy,time,w,test,final} and discarding $\frac{1}{U''}$, we complete the proof of \eqref{equ:est,om,weight,f1234}.
\end{proof}

Now we are going to prove  Proposition \ref{proposition:est,u,om,timespace,f1234}.

\begin{proof}[\textbf{Proof of Proposition $\ref{proposition:est,u,om,timespace,f1234}$}]
	Applying Lemma \ref{lemma:est,u,om,timespace,f1234}, we obtain
		\begin{align}\label{equ:est,u,om,weight,f12}
\n{e^{\ve \nu^{\f12}t}u}^2_{L_{t}^{\infty}L_{y}^{2}}+\nu\n{e^{\ve \nu^{\f12}t} \om}^2_{ L_{t,y}^{2}}
	\lesssim (1+\varepsilon\nu^{\f12})\n{e^{\varepsilon \nu^{\f12}t}u}^2_{L_{t,y}^{2}}+
	(\nu\abs{\al}^2)^{-1}\n{e^{\varepsilon \nu^{\f12}t}(f_1,f_2)}^2_{L_{t,y}^{2}}.
		\end{align}
Next, we divide the proof into two cases.

\textbf{Estimates for $\n{u}_{L_{t,y}^{2}}$, $\n{u}_{L_{t}^{\infty}L_{y}^{2}}$, and $\n{w}_{L_{t,y}^{2}}$.}

{\it Case 1.} $\nu\abs{\al}^2\lesssim 1$.

By Proposition \ref{proposition:est,u,timespace,L2L2,LinftyL2,f1234L2} and Proposition \ref{proposition:est,u,timespace,L2L2,LinftyL2,w0H1} with $f_3=f_4=0$, we get
\begin{align}\label{equ:est,u,om,weight,f12,case1}
	\abs{\al}\n{e^{\varepsilon \nu^{\f12}t}u}^2_{L_{t,y}^{2}}+	\nu^{\f12}\abs{\al}^{\f12}\n{e^{\varepsilon \nu^{\f12}t}\om}^2_{L_{t,y}^{2}}
	\lesssim \n{\om^{\mathrm{in}}}_{H_{\al}^{4}}^{2}+\nu^{-1}\n{e^{\varepsilon \nu^{\f12}t}(f_1,f_2)}^2_{L_{t,y}^{2}}.
\end{align}
Combining \eqref{equ:est,u,om,weight,f12} and \eqref{equ:est,u,om,weight,f12,case1}, we deduce
	\begin{align}\label{equ:est,u,weight,f12,case1,1}
		\abs{\al}\n{e^{\varepsilon \nu^{\f12}t}u}^2_{L_{t}^{\infty}L_{y}^{2}}+\nu^{\f12}\abs{\al}^{\f12}\n{e^{\varepsilon \nu^{\f12}t}\om}^2_{L_{t,y}^{2}}
		\lesssim	\n{\om^{\mathrm{in}}}_{H_{\al}^{4}}^{2}+\nu^{-1}\n{e^{\varepsilon \nu^{\f12}t}(f_1,f_2)}^2_{L_{t,y}^{2}}.
	\end{align}

{\it Case 2.} $\nu\abs{\al}^2\gg 1$.
	
	Using that $\n{u}_{L^{2}}\lesssim |\al|^{-1}\n{\om}_{L^{2}}$,
with $\varepsilon\nu^{\f12}\ll 1$ and $\nu\al^{2}\gg 1$, we get
\begin{align}\label{equ:est,u,om,weight,f12,case2,L2}
	\lrs{1+\varepsilon\nu^{\f12}}\n{e^{\varepsilon \nu^{\f12}t}u}^2_{ L^2}
	\lesssim & \abs{\al}^{-2}\n{e^{\varepsilon \nu^{\f12}t} \om}^2_{ L^2}\ll \nu\n{e^{\varepsilon \nu^{\f12}t} \om}^2_{ L^2},
\end{align}
	which together with \eqref{equ:est,u,om,weight,f12} gives
\begin{align}\label{equ:est,u,om,weight,f12,case2}
     \abs{\al}^2\n{e^{\varepsilon \nu^{\f12}t}u}^2_{L_{t}^{\infty}L_{y}^{2}}+\nu\abs{\al}^2\n{e^{\varepsilon \nu^{\f12}t} \om}^2_{L_{t,y}^{2}}
	\lesssim& \n{\om^{\mathrm{in}}}^2_{L^2}+\nu^{-1}\n{e^{\varepsilon \nu^{\f12}t}(f_1,f_2)}^2_{L_{t,y}^{2}}.
\end{align}
Due to that $\nu\abs{\al}^2\gg 1$ and $\nu\ll 1$, which yields
	\begin{align*}
		\abs{\al}\le \abs{\al}^2,\qquad
		\nu |\al|+
		\nu^{\f12}\abs{\al}^{\f12}\lesssim \nu\abs{\al}^2,
	\end{align*}
together with \eqref{equ:est,u,om,weight,f12,case2,L2} and \eqref{equ:est,u,om,weight,f12,case2}, we arrive at
	\begin{align}\label{equ:est,u,om,weight,f12,case21}
		&\abs{\al}\n{e^{\varepsilon \nu^{\f12}t}u}^2_{L_{t}^{\infty}L_{y}^{2}}+\abs{\al}\n{e^{\varepsilon \nu^{\f12}t}u}^2_{L_{t,y}^{2}}+\nu^{\f12}\abs{\al}^{\f12}\n{e^{\varepsilon \nu^{\f12}t}\om}^2_{L_{t,y}^{2}}\\
		\lesssim& \n{\om^{\mathrm{in}}}^2_{L^2}+\nu^{-1}\n{e^{\varepsilon \nu^{\f12}t}(f_1,f_2)}^2_{L_{t,y}^{2}}.\notag
	\end{align}

\textbf{Estimates for $\n{\om}_{L_{t}^{\infty} L_{y}^{2}}$ and $\n{(\pa_y,\abs{\al})\om}^2_{L_{t,y}^{2}}$.}

{\it Case 1.} $\nu\abs{\al}^2\lesssim 1$.

	By \eqref{equ:est,om,weight,f1234} in Lemma \ref{lemma:est,u,om,timespace,f1234} and the inequality $\n{\om}_{L^{2}}^{2}\lesssim \n{u}_{L^{2}}\n{(\pa_{y},|\al|)\om}_{L^{2}}$, we deduce
	\begin{align*}
		&\n{e^{\varepsilon \nu^{\f12}t}\om}^2_{L_{t}^{\infty}L_{y}^{2}}+\nu\n{e^{\varepsilon \nu^{\f12}t}(\pa_y,\abs{\al})\om}^2_{L_{t,y}^{2}}\\
		\lesssim&
		\n{\om^{\mathrm{in}}}^2_{L^2}+\nu^{-1}
		\n{e^{\varepsilon \nu^{\f12}t}(f_1,f_2)}^2_{L_{t,y}^{2}}+(\nu^{\f12}+\nu\abs{\al}^2)\n{e^{\varepsilon \nu^{\f12}t}\om}^2_{L_{t,y}^{2}}\notag\\
		&
		+\nu^{-\f13}\abs{\al}^{\f23}
		\n{e^{\varepsilon \nu^{\f12}t}u}^{\f43}_{L_{t,y}^{2}}\n{e^{\varepsilon \nu^{\f12}t}\om}^{\f23}_{L_{t,y}^{2}}\notag\\
		\lesssim&
		\n{\om^{\mathrm{in}}}^2_{L^2}+\nu^{-1}
		\n{e^{\varepsilon \nu^{\f12}t}(f_1,f_2)}^2_{L_{t,y}^{2}}+\nu^{\f12}\n{e^{\varepsilon \nu^{\f12}t}u}_{L_{t,y}^{2}}\n{e^{\varepsilon \nu^{\f12}t}(\pa_y,\abs{\al})\om}_{L_{t,y}^{2}}\notag\\
		&
		+\nu\abs{\al}^2\n{e^{\varepsilon \nu^{\f12}t}\om}^2_{L_{t,y}^{2}}	+\nu^{-\f13}\abs{\al}^{\f23}
		\n{e^{\varepsilon \nu^{\f12}t}u}^{\f43}_{L_{t,y}^{2}}\n{e^{\varepsilon \nu^{\f12}t}\om}^{\f23}_{L_{t,y}^{2}}\notag
	\end{align*}
	By Young's inequality and $\nu|\al|^{2}\lesssim 1$, we get
	\begin{align*}
		&\n{e^{\varepsilon \nu^{\f12}t}\om}^2_{L_{t}^{\infty}L_{y}^{2}}+\nu\n{e^{\varepsilon \nu^{\f12}t}(\pa_y,\abs{\al})\om}^2_{L_{t,y}^{2}}\\
		\lesssim&
		\n{\om^{\mathrm{in}}}^2_{L^2}+\nu^{-1}
		\n{e^{\varepsilon \nu^{\f12}t}(f_1,f_2)}^2_{L_{t,y}^{2}}+\n{e^{\varepsilon \nu^{\f12}t}u}^2_{L_{t,y}^{2}}+\nu\al^{2}\n{e^{\varepsilon \nu^{\f12}t}\om}^2_{L_{t,y}^{2}}\notag\\	
&+
\nu^{-\frac{1}{2}}|\al|^{-\frac{1}{6}}\lrs{|\al|\n{e^{\ve \nu^{\f12}t}u}_{L_{t,y}^{2}}^{2}+\nu^{\f12}\abs{\al}^{\f12}\n{e^{\varepsilon \nu^{\f12}t}\om}^2_{L_{t,y}^{2}}}\notag,\\
\lesssim&\n{\om^{\mathrm{in}}}^2_{L^2}+\nu^{-1}
		\n{e^{\varepsilon \nu^{\f12}t}(f_1,f_2)}^2_{L_{t,y}^{2}}+
\nu^{-\frac{1}{2}}|\al|^{-\frac{1}{6}}\lrs{|\al|\n{e^{\ve \nu^{\f12}t}u}_{L_{t,y}^{2}}^{2}+\nu^{\f12}\abs{\al}^{\f12}\n{e^{\varepsilon \nu^{\f12}t}\om}^2_{L_{t,y}^{2}}}\notag.
	\end{align*}
which, together with \eqref{equ:est,u,om,weight,f12,case1}, implies
\begin{align}\label{equ:w,LinfL2,case1}
\nu^{\frac{1}{2}}\lrs{\n{e^{\varepsilon \nu^{\f12}t}\om}^2_{L_{t}^{\infty}L_{y}^{2}}+\nu\n{e^{\varepsilon \nu^{\f12}t}(\pa_y,\abs{\al})\om}^2_{L_{t,y}^{2}}}\lesssim \n{\om^{\mathrm{in}}}_{H_{\al}^{4}}^{2}+\nu^{-1}\n{e^{\varepsilon \nu^{\f12}t}(f_1,f_2)}^2_{L_{t,y}^{2}}.
\end{align}

	{\it Case 2.} $\nu\abs{\al}^2\gg 1$.

By \eqref{equ:est,om,weight,f1234} in Lemma \ref{lemma:est,u,om,timespace,f1234}, with $\nu\al^{2}\gg 1$, we get
	\begin{align*}
		&\n{e^{\varepsilon \nu^{\f12}t}\om}^2_{L_{t}^{\infty}L_{y}^{2}}+\nu\n{e^{\varepsilon \nu^{\f32}t}(\pa_y,\abs{\al})\om}^2_{L_{t,y}^{2}}\\
\lesssim&
		\n{\om^{\mathrm{in}}}^2_{L^2}+\nu^{-1}
		\n{e^{\varepsilon \nu^{\f12}t}(f_1,f_2)}^2_{L_{t,y}^{2}}
		+\nu\abs{\al}^2\n{e^{\varepsilon \nu^{\f12}t}\om}^2_{L_{t,y}^{2}}\notag\\
&+\nu^{-\frac{2}{3}}|\al|^{-\frac{4}{3}}\lrs{|\al|^{2}\n{e^{\varepsilon \nu^{\f12}t}u}_{L_{t,y}^{2}}^{2}+\nu \abs{\al}^{2}\n{e^{\varepsilon \nu^{\frac{1}{2}}t}\om}^2_{L_{t,y}^{2}}}
\end{align*}
which, together with \eqref{equ:est,u,om,weight,f12,case2} and $\nu|\al|^{2}\gg 1$, yields \eqref{equ:w,LinfL2,case1}.

\textbf{Estimate for $\n{\sqrt{1-y^{2}}\om}_{L_{t}^{\infty}L_{y}^{2}}$.}

Recalling that
\begin{equation}\label{equ:linear,nonslip,NS,fj,used}
	\left\{
	\begin{aligned}
		&\pa_{t}\om-\nu(\pa^2_y-\al^2)\om+i\al U\om-i\al U''\phi=i\al f_1+\pa_y f_2,\\
		&(\pa^2_y-\al^2)\phi=\om,\quad \phi(t,\pm 1)=\phi'( t,\pm 1)=0,\quad \om(0,y)=\om^{\mathrm{in}}(y),
	\end{aligned}\right.
\end{equation}
	we test equation \eqref{equ:linear,nonslip,NS,fj,used} with $\f{1-y^{2}}{U''(y)}\om$
	and take the real part to obtain
	\begin{align*}
		&\frac{1}{2}\f{d}{dt}\bbn{\sqrt{\f{1-y^{2}}{U''(y)}}\om}^2_{L^2}+\nu\bbn{\sqrt{\f{1-y^{2}}{U''(y)}}(\pa_y,\abs{\al})\om}^2_{L^2}\\
		=&
		\operatorname{Re}\lrs{\bblra{\nu \pa_y\om,\om\lrs{\frac{1-y^{2}}{U''(y)}}'}+2\lra{i\al y\phi,\pa_y\phi} + \bblra{i\al f_1,\f{1-y^{2}}{U''(y)}\om}
		-\bblra{f_2,\lrs{\f{1-y^{2}}{U''(y)}\om}'}},
	\end{align*}
	which gives
\begin{align*}
		&\frac{1}{2}\f{d}{dt}\bbn{\sqrt{\f{1-y^{2}}{U''(y)}}\om}^2_{L^2}+\nu\bbn{\sqrt{\f{1-y^{2}}{U''(y)}}(\pa_y,\abs{\al})\om}^2_{L^2}\\
		\lesssim &
		\nu^{\frac{1}{2}}\n{\om}^2_{L^2}+\nu^{\f32}\n{\pa_y\om}
        ^2_{L^2}+\n{u}^2_{L^2}+\nu^{-1}
		\n{(f_1,f_2)}_{L^{2}}^{2}+
        \n{\om}_{L^2}\n{f_{2}}_{L^2}.
\end{align*}		
  Applying Young's inequality and estimates for $\n{\om}_{L_{t,y}^{2}}$, $\n{\pa_{y}w}_{L_{t,y}^{2}}$, and
$\n{u}_{L_{t,y}^{2}}$, we get
\begin{align*}
         \n{e^{\varepsilon \nu^{\f12}t}\sqrt{1-y^{2}}\om}^2_{L_{t}^{\infty}L_{y}^{2}}\lesssim\n{\om^{\mathrm{in}}}_{H_{\al}^{4}}^{2}+\nu^{-1}\n{e^{\varepsilon \nu^{\f12}t}(f_1,f_2)}^2_{L_{t,y}^{2}}.
\end{align*}

\textbf{Estimate for $\n{u}_{L_{t}^{\infty} L_{y}^{\infty}}$.}

Let us introduce the cutoff function $\rho_{\delta}(y)$ defined by
\begin{equation}
\rho_{\delta}(y)=
\left\{\begin{aligned}
1,\quad &|y|\leq 1-\delta,\\
\delta^{-1}(1-|y|),\quad &|y|\in(1-\delta,1).
\end{aligned}
\right.
\end{equation}
Setting $\delta=\nu^{\frac{1}{4}}$, we derive
\begin{align}\label{equ:u,LinfLinf,wL1}
\n{u}_{L_{y}^{\infty}}\lesssim \n{\om}_{L_{y}^{1}}\leq& \int_{ 1-|y|\leq \nu^{\frac{1}{2}}}|\om(y)|dy+\int_{\nu^{\frac{1}{2}} \leq 1-|y|\leq 1}|\om(y)|dy\\
\lesssim& \nu^{\frac{1}{4}}\n{\om}_{L_{y}^{2}}+\n{\rho_{\delta}\om}_{L_{y}^{2}}\lrs{\int_{\nu^{\frac{1}{2}} \leq 1-|y|\leq 1} \rho_{\delta}^{-2}(y)dy}^{\frac{1}{2}}\notag\\
\lesssim& \nu^{\frac{1}{4}}\n{\om}_{L_{y}^{2}}+(\delta \nu^{-\frac{1}{4}}+1)\n{\rho_{\delta}\om}_{L_{y}^{2}}\notag\\
\lesssim& \nu^{\frac{1}{4}}\n{\om}_{L_{y}^{2}}+\n{\rho_{\delta}\om}_{L_{y}^{2}}.\notag
\end{align}
To bound $\n{\rho_{\delta}w}_{L_{y}^{2}}$, we introduce
a $C^{2}$-smooth function $\wt{\rho}_{\delta}(y)$ satisfying
\begin{align}\label{equ:smooth,C2,cutoff}
\wt{\rho}_{\delta}(y)\geq \rho_{\delta}(y),\quad |\pa_{y}\wt{\rho}_{\delta}(y)|\lesssim \delta^{-1},\quad
\quad |\pa_{y}^{2}\wt{\rho}_{\delta}(y)|\lesssim \delta^{-2}.
\end{align}
For instance, the following example meets these requirements
\begin{equation}
\wt{\rho}_{\delta}(y)=
\left\{\begin{aligned}
1,\quad &|y|\leq 1-\delta,\\
(\delta^{-1}(1-|y|)-1)^{3}+1,\quad &|y|\in(1-\delta,1).
\end{aligned}
\right.
\end{equation}

Testing equation \eqref{equ:linear,nonslip,NS,fj,used} with $\f{\wt{\rho}_{\delta}^{2}}{U''}\om$
	and taking the real part, we derive
	\begin{align*}
		&\frac{1}{2}\f{d}{dt}\bbn{\sqrt{\f{\wt{\rho}_{\delta}^{2}}{U''}}\om}^2_{L^2}+\nu\bbn{\sqrt{\f{\wt{\rho}_{\delta}^{2}}{U''}}(\pa_y,\abs{\al})\om}^2_{L^2}\\
		=&
		\operatorname{Re}\lrs{\bblra{\nu \om',\om\lrs{\frac{\wt{\rho}_{\delta}^{2}}{U''}}'}+\lra{i\al \wt{\rho}_{\delta}^{2}\phi,\om}} + \operatorname{Re}\lrs{\bblra{i\al f_{1},\f{\wt{\rho}_{\delta}^{2}}{U''}\om}
			-\bblra{f_{2},\lrs{\f{\wt{\rho}_{\delta}^{2}}{U''}\om}'}}\\
		=&
		\operatorname{Re}\lrs{-\bblra{\nu |\om|^{2},\lrs{\frac{\wt{\rho}_{\delta}^{2}}{U''}}''}-\lra{i\al (\wt{\rho}_{\delta}^{2})'\phi,\pa_{y}\phi}} + \operatorname{Re}\lrs{\bblra{i\al f_{1},\f{\wt{\rho}_{\delta}^{2}}{U''}\om}
			-\bblra{f_{2},\lrs{\f{\wt{\rho}_{\delta}^{2}}{U''}\om}'}}.
	\end{align*}
Using Cauchy-Schwarz inequality  and \eqref{equ:smooth,C2,cutoff}, we obtain
\begin{align*}
\bbabs{\bblra{\nu |\om|^{2},\lrs{\frac{\wt{\rho}_{\delta}^{2}}{U''}}''}}\lesssim& \nu\delta^{-2}\n{\om}_{L^{2}}^{2},\\
\abs{\lra{i\al (\wt{\rho}_{\delta}^{2})'\phi,\pa_{y}\phi}}\lesssim& |\al|\delta^{-\frac{1}{2}}\n{1_{\lr{1-|y|\leq \delta}}\phi}_{L^{\infty}}\n{\pa_{y}\phi}_{L^{2}}
\lesssim |\al|\n{u}_{L^{2}}^{2},\\
\bbabs{\bblra{i\al f_{1},\f{\wt{\rho}_{\delta}^{2}}{U''}\om}
			-\bblra{f_{2},\lrs{\f{\wt{\rho}_{\delta}^{2}}{U''}\om}'}}\lesssim & \n{(f_{1},f_{2})}_{L^{2}}\bbn{\sqrt{\frac{\wt{\rho}_{\delta}}{U''}}(\pa_{y},|\al|)w}_{L^{2}}
+\delta^{-1}\n{f_{2}}_{L^{2}}\n{\om}_{L^{2}},
\end{align*}
	which gives
	\begin{align*}
		&\frac{1}{2}\f{d}{dt}\n{\wt{\rho}_{\delta}\om}^2_{L^2}
+\nu\n{\wt{\rho}_{\delta}(\pa_y,\abs{\al})\om}_{L^{2}}\\
		\lesssim &
		\nu\delta^{-2}\n{\om}_{L^{2}}^{2}+|\al| \n{u}_{L^{2}}^{2}+
		\n{(f_{1},f_{2})}_{L^{2}}\bbn{\sqrt{\frac{\wt{\rho}_{\delta}}{U''}}(\pa_{y},|\al|)w}_{L^{2}}
+\delta^{-1}\n{f_{2}}_{L^{2}}\n{\om}_{L^{2}}.
\end{align*}		
  Applying Young's inequality and estimates for $\n{\om}_{L_{t,y}^{2}}$ and
$\n{u}_{L_{t,y}^{2}}$, with $\delta=\nu^{\frac{1}{4}}$ we get
\begin{align*}
\n{e^{\varepsilon \nu^{\f12}t}\wt{\rho}_{\delta}\om}^2_{L_{t}^{\infty}L_{y}^{2}}\lesssim &
\nu^{\frac{1}{2}}\n{e^{\varepsilon \nu^{\f12}t} w}_{L_{t,y}^{2}}^{2}+|\al|\n{e^{\varepsilon \nu^{\f12}t} u}_{L_{t,y}^{2}}^{2}+
\nu^{-1}\n{e^{\varepsilon \nu^{\f12}t}(f_1,f_2)}^2_{L_{t,y}^{2}}\\
\lesssim&\n{\om^{\mathrm{in}}}_{H_{\al}^{4}}^{2}+\nu^{-1}\n{e^{\varepsilon \nu^{\f12}t}(f_1,f_2)}^2_{L_{t,y}^{2}}.
	\end{align*}
By \eqref{equ:u,LinfLinf,wL1}, \eqref{equ:smooth,C2,cutoff}, and estimate for $\n{\om}_{L_{t}^{2}L_{y}^{2}}$, we conclude that
\begin{align*}
\n{e^{\varepsilon \nu^{\f12}t}u}_{L_{t}^{\infty}L_{y}^{\infty}}^{2}\lesssim \nu^{\frac{1}{4}}\n{e^{\varepsilon \nu^{\f12}t}w}_{L_{t}^{\infty}L_{y}^{2}}^{2}+\n{e^{\varepsilon \nu^{\f12}t}\wt{\rho}_{\delta}\om}_{L_{t}^{\infty}L_{y}^{2}}^{2}
\lesssim&\n{\om^{\mathrm{in}}}_{H_{\al}^{4}}^{2}+\nu^{-1}\n{e^{\varepsilon \nu^{\f12}t}(f_1,f_2)}^2_{L_{t,y}^{2}}.
\end{align*}

Therefore, we have completed the entire proof.
\end{proof}

\section{Asymptotic stability of the symmetric flow}\label{sec:Nonlinear Stability}

In this section, we prove Theorem \ref{theorem:nonlinear,stability}. The idea is to provide the stability estimate on
 the stability norm $\sum_{\al\in \Z}\mathcal{E}_{\al}$, where
\begin{equation}\label{equ:definition,E,al}
	\mathcal{E}_{\al}:=
			\left\{
		\begin{aligned}
		&\n{{\om}_0}_{L_{t}^{\infty}L_{y}^{2}},
		&   {\al}= 0,\\
		&\abs{\al}^{\f12}\n{ e^{\ve_{0} \nu^{\f12}t}u_{\al}}_{L_{t}^{\infty}L_{y}^{2}}+\abs{\al}^{\f12}\n{ e^{\ve_{0} \nu^{\f12}t}u_{\al}}_{L_{t,y}^{2}}+\n{e^{\ve_{0} \nu^{\f12}t} u_{\al}}_{L_{t}^{\infty}L_{y}^{\infty}}+\\
		&\nu^{\f14}\abs{\al}^{\f14}\n{e^{\ve_{0} \nu^{\f12}t}\om_{\al}}_{L_{t,y}^{2}}
		+\nu^{\f14}\n{e^{\ve_{0}\nu^{\f12}t}\om_{\al}}_{L_{t}^{\infty}L_{y}^{2}}+ \n{e^{\ve_{0} \nu^{\f12}t}\sqrt{1-y^{2}}\om_{\al}}_{L_{t}^{\infty}L_{y}^{2}},  &{\al}\neq 0.
		\end{aligned}\right.
\end{equation}

\begin{proof}[\textbf{Proof of Theorem $\ref{theorem:nonlinear,stability}$}]
Recalling \eqref{equ:omega,full,NS}, which can be rewritten as
\begin{equation}\label{equ:omega,full,NS,curl}
	\pa_{t} \om-\nu\Delta\om+U(y)\pa_x\om -U''(y)\pa_x\phi= -
    \operatorname{curl} (u\cdot\nabla u),
\end{equation}
we consider the $\al$-frequency of equation \eqref{equ:omega,full,NS,curl} as follow
\begin{align*}
&\pa_{t} \om_{\al}(t,y)-\nu(\pa_y^2-\al^2)\om_{\al}(t,y)+i\al U(y)\om_{\al}(t,y)-i\al U''(y)\phi_{\al}(t,y)\\
=&-i\al (f_{\al}^{1,2}(t,y)+f_{\al}^{2,2}(t,y))+\pa_y (f^{1,1}_{\al}(t,y)+f^{2,1}_{\al}(t,y)),
\end{align*}
where
\begin{align*}
	&f^{1,2}_{\al}(t,y)=i\sum_{l\in \Z}u^1_l(t,y)(\al-l) u^{2}_{\al-l}(t,y),\qquad
	 f^{2,2}_{\al}(t,y)=\sum_{l\in \Z}u^2_l(t,y)\pa_yu^2_{\al-l}(t,y),\\
    &f^{1,1}_{\al}(t,y)=i\sum_{l\in \Z}u^1_l(t,y)(\al-l) u^1_{\al-l}(t,y) ,\qquad
	f^{2,1}_{\al}(t,y)=\sum_{l\in \Z}u^2_l(t,y)\pa_yu^1_{\al-l}(t,y).
\end{align*}

We now analyze the energy term $\mathcal{E}_{0}$.
The condition $\operatorname{div} u=0$ implies that $u_0^2(t,y)=0$. Combined with $P_0(u^1\pa_x u^1)=0,$ the zero-frequency mode of \eqref{equ:per,INS} for $u^{1}$ yields
\begin{align}\label{new,equ:u01,def}
	\pa_{t}{u^1_0}(t,y)-\nu\pa_y^2{u^1_0}(t,y)=&-\sum_{l\in \Z\setminus\{0\} }u^2_l(t,y)\pa_yu^1_{-l}(t,y)=-f_0^{2,1}(t,y).
\end{align}

Testing equation \eqref{new,equ:u01,def} by $\pa_y^2u_0^1$ and using the integration by parts, we obtain
\begin{align*}
	\f{d}{dt}\n{\pa_yu_0^1(t)}^2_{L_{y}^2}+\nu\n{\pa_y^2u_0^1(t)}^2_{L_{y}^2}\lesssim \nu^{-1}\n{f_0^{2,1}(t,y)}^2_{L_{y}^2},
\end{align*}
which, together with  $\pa_yu_0^1(t,y)=\om_0(t,y)$, gives
\begin{align}\label{equ:est,E0,bound}
	\mathcal E_0^2=\n{\om_0}^2_{L_{t}^{\infty}L_{y}^{2}}
	\lesssim \n{{\om}^{\mathrm{in}}_0}^2_{L_{y}^2}+\nu^{-1}\n{f_0^{2,1}}^2_{L_{t,y}^2}.
\end{align}

Now we estimate $\mathcal E_{\al}$. Using the space-time estimates in Proposition \ref{proposition:est,u,om,timespace,f1234} and the definition \eqref{equ:definition,E,al} of \(\mathcal E_{\al}\), we get
\begin{align}\label{new,equ:est,E,al,bound}
	\mathcal E_{\al}\lesssim&
	\n{\om_{\al}^{\mathrm{in}}}_{H_{\al}^4}+\nu^{-\frac{1}{2}}(\n{e^{\ve_{0} \nu^{\f12}t}f^{1,2}_{\al}}_{L_{t,y}^2}
	 +\n{e^{\ve_{0} \nu^{\f12}t}f^{2,2}_{\al}}_{L_{t,y}^2})\\
&	+\nu^{-\frac{1}{2}}(\n{e^{\ve_{0} \nu^{\f12}t}f^{1,1}_{\al}}_{L_{t,y}^2}+
	\n{e^{\ve_{0}
        \nu^{\f12}t}f^{2,1}_{\al}}_{L_{t,y}^2})
	.\notag
\end{align}

For $\n{e^{\ve_{0} \nu^{\f12}t}f^{1,j}_{\al}}_{L_{t,y}^{2}}$ with $j=1,2$, we have
\begin{align}\label{equ:f1j,estimate}
&	\n{e^{\ve_{0} \nu^{\f12}t}f^{1,j}_{\al}}_{L_{t,y}^{2}}\\
\leq &
	\sum_{l\in \Z} \n{e^{\ve_{0} \nu^{\f12}t}u^1_l |\al-l|u^{j}_{{\al}-l}}_{L_{t,y}^2}\notag\\
	\lesssim &
	\sum_{l\in \Z}
	  \n{u^1_{l}}_{L_{t}^{\infty}L_{y}^{\infty}}
	\n{e^{\ve_{0} \nu^{\f12}t}|\al-l|u^{j}_{\al-l}}_{L_{t,y}^{2}}\notag\\
	\lesssim & \nu^{-\frac{1}{6}}
	\sum_{l\in \Z}
	\n{u^1_{l}}_{L_{t}^{\infty}L_{y}^{\infty}}
	(|\al-l|^{\f12}\n{e^{\ve_{0} \nu^{\f12}t}u^{j}_{\al-l}}_{L_{t,y}^2})^{\f13}
	(\nu^{\frac{1}{4}}|\al-l|^{\f14}\n{e^{\ve_{0} \nu^{\f12}t}\om_{\al-l}}_{L_{t,y}^2})^{\f23}\notag\\
    \lesssim &\nu^{-\f16}
	\sum_{l\in \Z}
	\mathcal E_{l}\mathcal E_{\al-l}.\notag
\end{align}

For $\n{e^{\ve_{0} \nu^{\f12}t}f^{2,j}_{\al}}_{L_{t,y}^{2}}$ with $j=1,2$,  by H\"{o}lder inequality, Hardy inequality \eqref{equ:hardy,Linf,L2}, Sobolev estimate \eqref{equ:proper,u,H1,H1weight} in Lemma \ref{lemma:proper,u,L2,Linf}, and the fact that $\pa_{y}u^{2}=-\pa_{x}u^{1}$, we have
\begin{align}\label{equ:f2j,estimate}
	\n{e^{\ve_{0} \nu^{\f12}t}f^{2,j}_{\al}}_{L_{t,y}^2}
	\leq& \sum_{l\in \Z} \bbn{\f{e^{\ve_{0} \nu^{\f12}t}u^{2}_{l}}{\sqrt{1-y^{2}}}
	}_{L_{t}^{2} L_{y}^{\infty}}\n{\sqrt{1-y^{2}}
		\pa_yu^{j}_{\al-l}}_{L_{t}^{\infty}L_{y}^{2}}\\
	\leq& \sum_{l\in \Z } \n{e^{\ve_{0} \nu^{\f12}t}\pa_{y}u^{2}_{l}}_{L_{t,y}^2}
	\n{\sqrt{1-y^{2}}
	\om_{\al-l}}_{L_{t}^{\infty}L_{y}^{2}}\notag\\
	\leq& \nu^{-\frac{1}{6}}\sum_{l\in \Z} (|l|^{\f12}\n{e^{\ve_{0} \nu^{\f12}t}u^1_l}_{L_{t,y}^2})^{\f13}
	(\nu^{\frac{1}{4}}|l|^{\f14}\n{e^{\ve_{0} \nu^{\f12}t}\om_l}_{L_{t,y}^2})^{\f23}\mathcal E_{\al-l}\notag\\
    \lesssim&\nu^{-\f16}\sum_{l\in \Z} \mathcal E_{l}\mathcal E_{\al-l}.  \notag
\end{align}

Collecting the estimates \eqref{equ:est,E0,bound}, \eqref{equ:f1j,estimate}, and \eqref{equ:f2j,estimate}, we arrive at
\begin{align*}
	\mathcal E_0\lesssim& \n{\om^{\mathrm{in}}_0}_{L_{y}^2}+\nu^{-\frac{1}{2}} \nu^{-\f{1}{6}}\sum_{l\in \Z}\mathcal E_{l}\mathcal E_{-l}
	\lesssim
	\n{\om^{\mathrm{in}}_0}_{H_{\al}^{4}}+ \nu^{-\f{2}{3}}\sum_{l\in \Z} \mathcal E_{l}\mathcal E_{-l},\\
	\mathcal E_{\al}
	\lesssim&
	\n{\om^{\mathrm{in}}_{\al}}_{H_{\al}^{4}}+ \nu^{-\f23}\sum_{l\in \Z} \mathcal E_{l}\mathcal E_{\al-l},
\end{align*}
which implies
\begin{align*}
	\sum_{\al\in \Z} \mathcal E_{\al}
	\lesssim& \sum_{\al\in \Z}\n{\om^{\mathrm{in}}_{\al}}_{H_{\al}^{4}}+ \nu^{-\f{2}{3}}\sum_{\al\in \Z}\sum_{l\in \Z} \mathcal E_{l}\mathcal E_{\al-l}.
\end{align*}

Given the initial condition that
\begin{align*}
	\sum_{\al\in \Z}
	\n{\om^{\mathrm{in}}_{\al}}_{H_{\al}^{4}}\leq c\nu^{\f{2}{3}},
\end{align*}
by choosing
$c$ sufficiently small, a standard continuity argument yields the uniform bound
\begin{align*}
	\sum_{\al\in \Z} \mathcal{E}_{\al}\lesssim
	\sum_{\al\in \Z}\n{\om^{\mathrm{in}}_{\al}}_{H_{\al}^{4}}.
\end{align*}
This concludes the proof.
\end{proof}

\appendix
\section{Integral estimates for the symmetric flow}\label{sec:Estimates for the Symmetric Flows}

 \begin{lemma}\label{lemma:proper,symme,flow}
Let $U(y)\in \mathrm{S}$. Then the following asymptotic behaviors hold
 		\begin{align}
 			U'(y)\sim y, \quad U(y)-U(y')\sim  (y-y')(y+y')
\label{u-y2}.
 		\end{align}
 \end{lemma}
 \begin{proof}
This result is elementary. For details, see for example \cite[Lemma 4.1]{WZZ-APDE}.
 \end{proof}

  \begin{lemma}\label{lemma:basic Lp estimate,U}
 	Assume that $-1\leq y_{1}\leq 0\leq y_{2}\leq 1$ with $U(y_i)=\lambda\in [0,1]$, $i=1,2$. Then for all $\delta\in[0,1]$,the following estimates hold
 \begin{align}
 \bbn{\frac{1}{U-\la}}_{L^{\infty}(-1,1)\setminus B_{1,2,\delta})}\lesssim& \frac{1}{(y_{2}+\delta)\delta},\label{equ:basic estimate,U,L0}\\
 \bbn{\frac{1}{U-\lambda}}_{L^{2}((-1,1)\setminus B_{1,2,\delta})}\lesssim&\frac{1}{\delta^{\frac{1}{2}}(y_{2}+\delta)},\label{equ:basic estimate,U,L2}\\
 \bbn{\frac{1}{U-\lambda}}_{L^{1}((-1,1)\setminus B_{1,2,\delta})}\lesssim&\frac{\mathrm{ln}(1+\frac{2y_{2}}{\delta})}{y_{2}} \lesssim \frac{1}{\delta},\label{equ:basic estimate,U,L1}\\
 \bbn{\frac{U'}{(U-\lambda)^2}}_{L^2((-1,1)\setminus B_{1,2,\delta})}\lesssim& \frac{1}{\delta^{\f32}(y_{2}+\delta)},\label{equ:basic estimate,U,H1}
 	\end{align}
where $B_{1,2,\delta}=B(y_{1},\delta)\cup B(y_{2},\delta)$. In particular, we have
 \begin{align}
\bbn{\lrs{\frac{\rho_{\delta}}{U''}}'}_{L^{2}}+\bbn{\lrs{\frac{\rho_{\delta}^{c}}{U''}}'}_{L^{2}}\lesssim& \frac{1}{\delta^{\frac{1}{2}}},\label{equ:basic estimate,U,rho,L2}\\
 \bbn{\lrs{\frac{\rho_{\delta}}{U-\la}}'}_{L^{2}}+\bbn{\lrs{\frac{\rho_{\delta}^{c}}{U-\la}}'}_{L^{2}}\lesssim & \frac{1}{(y_{2}+\delta)\delta^{\frac{3}{2}}},\label{equ:basic estimate,U,rho,H1}
 \end{align}
 where the smooth cutoff functions $\rho_{\delta}^{c}(y)$ and $\rho_{\delta}(y)$ are defined in \eqref{equ:smooth cutoff,ex,delta} and \eqref{equ:smooth cutoff,in,delta}.
 \end{lemma}
 \begin{proof}
  By Lemma \ref{lemma:proper,symme,flow}, we derive that $U(y)-\la\sim (y-y_{2})(y+y_{2})$. The estimates \eqref{equ:basic estimate,U,L0}--\eqref{equ:basic estimate,U,H1} then follow from direct calculations.
The bounds \eqref{equ:basic estimate,U,rho,L2}--\eqref{equ:basic estimate,U,rho,H1} are derived similarly.
\end{proof}

\section{Sobolev estimates}\label{sec:Sobolev Estimates}

\begin{lemma}[Hardy inequalities]\label{lemma:hardy inequality}
There holds that
\begin{align}\label{equ:hardy inequality}
\int_{0}^{\infty}\lrs{\frac{1}{x}\int_{0}^{x}f(y)dy}^{2}dx\lesssim\int_{0}^{\infty}|f(y)|^{2}dy.
\end{align}
Furthermore, for any function $f(y)\in H^{1}_0(-1,1)$, the following estimates hold
 	\begin{align}
 		&\Big\|\frac{f}{1-y^2}\Big\|_{L^2}\lesssim \|\partial_yf\|_{L^2},\quad \Big\|\frac{f}{\sqrt{1-y^{2}}}\Big\|_{L^2}\lesssim \|\partial_yf\|_{L^2},\label{proper,f,h1,1}\\
 &\bbn{\frac{f}{\sqrt{1-y^{2}}}}_{L^{\infty}}\lesssim \n{\pa_{y}f}_{L^{2}},\label{equ:hardy,Linf,L2}\\
 		&\|f\|_{L^2}\leq \|(1-y^2)^{\frac{1}{2}}f\|^{\f23}_{L^2}\|(1-y^2)^{-1}f\|^{\f13}_{L^2}\lesssim \|(1-y^2)^{\frac{1}{2}}f\|^{\f23}_{L^2}\|\partial_yf\|^{\f13}_{L^2}.\label{proper,f,h1,3}
 	\end{align}
\end{lemma}
\begin{proof}
The classical Hardy inequality \eqref{equ:hardy inequality} is well-established and widely documented in the literature. Hence, its proof is omitted here. For the specific estimates \eqref{proper,f,h1,1}--\eqref{proper,f,h1,3}, see for example \cite[Lemma A.2]{cdlz-arxiv2023}.
\end{proof}

	\begin{lemma}\label{lemma:proper,u,L2,Linf}
 		Let $(\partial_y^2-|\alpha|^2)\psi=w$ and $u=(-\pa_y\psi,i\al\psi)$ with $\psi(\pm 1)=0$ and $|\alpha|\geq1$. Then it holds that
 		\begin{align}
 			&\n{u}_{L^\infty}\lesssim \n{w}_{L^1},\qquad
 			\qquad \n{u}_{L^\infty}\lesssim |\alpha|^{-\frac{1}{2}}\n{w}_{L^2}\label{equ:proper,u,LinfL2}\\
             & \n{\psi}_{L^{2}}\lesssim |\al|^{-2}\n{w}_{L^{2}},\qquad \n{u}_{L^2}\lesssim |\alpha|^{-\frac{1}{2}}\n{w}_{L^1},\qquad
            \n{u}_{L^2}\lesssim|\al|^{-1}\n{w}_{L^{2}}
           ,\label{equ:proper,u,L2L1}\\
 			&\n{(\sqrt{1-y^{2}})(\pa_{y},|\al|) u}_{L^2}\lesssim \n{(\sqrt{1-y^{2}})w}_{L^2}\label{equ:proper,u,H1,H1weight}.
 		\end{align}
 	\end{lemma}
\begin{proof}
 For estimates \eqref{equ:proper,u,LinfL2} and \eqref{equ:proper,u,L2L1}, see for example \cite[Lemma 9.3]{CLWZ-ARMA}. For \eqref{equ:proper,u,H1,H1weight},
 using integration by parts, we get
\begin{align*}
        &\n{\sqrt{1-y^{2}}w}^2_{L^2}=\int_{-1}^1\big(1-y^{2}\big)(\pa_y^2\psi-|\al|^2\psi)^2dy\\
        =&\int_{-1}^1\big(1-y^{2}\big)(|\pa_y^2\psi|^2+|\al^2\psi|^2)dy-2\operatorname{Re}
        \int_{-1}^1\big(1-y^{2}\big)\pa_y^2\psi|\al|^2\ol{\psi} dy\\
        =&\int_{-1}^1\big(1-y^{2}\big)(|\pa_y^2\psi|^2+2|\al \pa_y\psi|^2+|\al^2\psi|^2)dy-2\operatorname{Re}\int_{-1}^1 y\pa_y\psi|\al|^2\ol{\psi} dy\\
        =&\int_{-1}^1\big(1-y^{2}\big)(|\pa_y^2\psi|^2+2|\al \pa_y\psi|^2+|\al^2\psi|^2)dy+\int_{-1}^1 |\al\psi|^2 dy,
        \end{align*}
which gives
\begin{align*}
    \n{\sqrt{1-y^{2}}(\pa_{y},|\al|) u}_{L^2}^2\leq
    \n{\sqrt{1-y^{2}}w}_{L^2}^{2}.
\end{align*}

\end{proof}

	\begin{lemma}
 [{\hspace{-0.01em}\cite[Lemma C.2]{CWZ-CMP}}]\label{elliptic-est-psidec12}
 	Let $\psi_{dec,1},\psi_{dec,2} \in H^{2}(-1,1)\cap H_0^1(-1,1)$, $g\in H^{2}(-1,1)$ satisfy $(\partial_y^2-|\alpha|^2)\psi_{dec,1}=g(\partial_y^2-|\alpha|^2)\psi_{dec,2}$ for $|\al|\geq 1$. Then it holds that
 	\begin{align}
 |\al|^{\frac{1}{2}}\|u_{dec,1}\|_{L^2}+|\partial_y\psi_{dec,1}(\pm1)|
 		\lesssim &|\al|^{\frac{1}{2}}\|g\|_{C^1}
 \|u_{dec,2}\|_{L^2}+
 \n{g}_{L^{\infty}}|\partial_y\psi_{dec,2}(\pm1)|\label{est-psidec1-1}
 ,\\
\|(\partial_y,\alpha)(\psi_{dec,1}-g\psi_{dec,2})\|_{L^2}\lesssim&  \|\partial_yg\|_{H^{1}}\|\psi_{dec,2}\|_{L^2}.\label{est-psidec1-2}
 	\end{align}
 \end{lemma}

\section{Estimates for Rayleigh and Euler equations}\label{sec:Estimates for the Rayleigh and Euler Equations}

  \begin{lemma}[{\hspace{-0.01em}\cite[Proposition 6.1]{WZZ-APDE}}]\label{lemma:proper,phi,ray,equ}
Let $U(y)\in \mathrm{S}$, $\mathrm{Re}\,\la\in\mathrm{Ran}\,U$, $0<\abs{\mathrm{Im}\,\la}\ll 1$,
 	and $\psi\in H^{1}(-1,1)$ be the solution to the Rayleigh equation
 	\begin{equation}\label{equ:def,Ray,U}
 		(U(y) - \la )(\pa_y^2-\abs{\al}^2)\psi-U''(y)\psi = g,  \quad \psi(\pm 1) = 0.
 	\end{equation}
Then it holds that
 	\begin{align}
 		\n{\pa_y\psi}_{L^2}+\abs{\al}\n{\psi}_{L^2}
 		\lesssim \n{\pa_yg}_{L^2}+\abs{\al}\n{g}_{L^2}.
 	\end{align}
 \end{lemma}
  \begin{lemma}[{\hspace{-0.01em}\cite[Proposition C.2]{cdlz-arxiv2023}}]\label{lemma:proper,phi,ray,equ,refined}
Let $U(y)\in \mathrm{S}$, $\mathrm{Re}\,\la\in\mathrm{Ran}\,U$, $0<\abs{\mathrm{Im}\,\la}\ll 1$,
 	and $\psi\in H^{1}(-1,1)$ be the solution to the Rayleigh equation \eqref{equ:def,Ray,U} with $g=-\frac{\sinh |\al|(1+y)}{\sinh |\al|}$.
Then it holds that
 	\begin{align}\label{equ:est,psi,H1,g,H1}
 		\n{\pa_y\psi}_{L^2}+\abs{\al}\n{\psi}_{L^2}
 		\lesssim |\al|^{-\frac{1}{2}}.
 	\end{align}
 \end{lemma}

Next, we provide the following space-time estimates for the linearized Euler equation
\begin{align}\label{equ:linear,euler,near,U,def}
	\pa_{t}\om +i\al U\om-i\al U''\phi=f,\qquad (\pa_{y}^{2}-\abs{\al}^2)\phi=\om.
\end{align}
\begin{lemma}\label{lemma:est,euler,w,u,l2,fH1}
Let $U(y)\in \mathrm{S}$ and $( \om,\phi)$ be the solution to the linearized Euler equation \eqref{equ:linear,euler,near,U,def}, then it holds that
	\begin{align}\label{equ:est,wu,l2,win,f,H1}
		\n{ \om(t)}_{L^2}^{2}+\abs{\al}\int_{0}^{+\infty}\n{(\pa_y,|\al|)\phi(t)}_{L^2}^{2}dt
		\lesssim &
		\n{\om^{\mathrm{in}}}^2_{H^{1}}
        +\abs{\al}^{-1}\int_{0}^{+\infty}
        \n{(\pa_y,|\al|)f(t)}_{L^2}^{2}dt.
	\end{align}
	If $f(t,\pm 1)=0$, then we can also obtain
	\begin{align}\label{equ:est,u,y=01,l2,win,f,H1}
		\int_{0}^{+\infty}\n{(\pa_y,|\al|)\phi(t,\pm 1)}_{L^2}^{2}dt
		\lesssim
		\n{ \om^{\mathrm{in}}}^2_{H_{\al}^{1}}+\abs{\al}^{-1}\int_{0}^{+\infty}\n{(\pa_y,|\al|)f(t)}_{L^2}^{2}dt.
	\end{align}
	If $f=0$, we have
	\begin{align}\label{equ:est,kpsi,l2,w0,H2,H1}
		\n{\pa_y \phi(t)}_{L^2}
		\lesssim \f{\n{ \om^{\mathrm{in}}}_{H_{\al}^{1}}}{\abs{\al}^{\frac{1}{2}}\lra{t}}\qquad  \text{and} \qquad
        \n{\al \phi(t)}_{L^2}
		\lesssim \min \lr{\f{\n{\om^{\mathrm{in}}}_{H_{\al}^{2}}}{\abs{\al}^{-\frac{1}{2}}\lra{t}^2},\f{\n{ \om^{\mathrm{in}}}_{H_{\al}^{1}}}{\abs{\al}^{\frac{1}{2}}\lra{t}},\f{\n{\om^{\mathrm{in}}}_{L^2}}{\abs{\al}}},
	\end{align}
	which in particular implies
	\begin{align}\label{equ:est,tkpsi,l2,w0,H1H2}
		\int_{0}^{+\infty}\lra{t}^{\frac{5}{2}}\n{\al\phi(t)}_{L^2}^{2}dt
		\lesssim \abs{\al}^{-1}
		\big(\n{\om^{\mathrm{in}}}_{H_{\al}^{2}}^{2}+\abs{\al}^2\n{ \om^{\mathrm{in}}}_{H_{\al}^{1}}^{2}\big).
	\end{align}

Furthermore, if $f=0$ and $\int_{\T} \om^{\mathrm{in}}dx=0$,
\begin{align}\label{equ:est,om,loc+nloc,depletion}
|\om(t,y)|\lesssim \abs{U'(y)}^{\f74}\n{\om^{\mathrm{in}}}_{H_{\al}^{4}}+\lra{t}^{-\f78}
\n{\om^{\mathrm{in}}}_{H_{\al}^{4}}.
\end{align}
\end{lemma}

\begin{proof}
 The proof of space-time estimates in \eqref{equ:est,wu,l2,win,f,H1}--\eqref{equ:est,u,y=01,l2,win,f,H1} is similar to those in \cite{WZZ-CMP} and \cite[Proposition D.1]{cdlz-arxiv2023}, which we here omit it for brevity.
For the inviscid damping estimates \eqref{equ:est,kpsi,l2,w0,H2,H1}, we appeal directly to \cite[Theorem 1.2]{WZZ-APDE}, while the pointwise vorticity depletion estimate \eqref{equ:est,om,loc+nloc,depletion} is obtained from \cite[Theorem 1.2]{IIJ-VJM}.
\end{proof}

\section*{Acknowledgment}
The authors thank Dongyi Wei for the stimulating discussions.  Q. Chen is partially supported by NSF of China under Grant 12401268. S. Shen is partially supported by NSF of China under Grant 12501322. Z. Zhang is partially supported by NSF of China under Grant 12288101.

\medskip
\noindent\textbf{Data Availability Statement.}
Data sharing is not applicable to this article as no datasets were generated or analysed during the current study.\medskip

\noindent\textbf{Conflict of Interest.}
The authors declare that they have no conflict of interest.

\end{document}